\documentclass[11pt]{article}
\usepackage[a4paper,margin=1in]{geometry}
\usepackage{amsmath,amssymb,amsthm,mathtools}
\usepackage{mathrsfs}
\usepackage{microtype}
\usepackage{enumitem}
\usepackage{booktabs}

\usepackage[
    colorlinks=true,
    linkcolor=blue,
    citecolor=blue,
    urlcolor=magenta,
    filecolor=blue
]{hyperref}

\numberwithin{equation}{section}

\newtheorem{theorem}{Theorem}[section]
\newtheorem{proposition}[theorem]{Proposition}
\newtheorem{lemma}[theorem]{Lemma}
\newtheorem{corollary}[theorem]{Corollary}
\newtheorem{claim}{Claim}[section]

\theoremstyle{definition}
\newtheorem{definition}[theorem]{Definition}

\newtheorem{conjecture}[theorem]{Conjecture}

\newcommand{\one}{\mathbf 1}
\newcommand{\Q}{\mathbb Q}
\newcommand{\R}{\mathbb R}
\newcommand{\N}{\mathbb N}
\newcommand{\Z}{\mathbb Z}
\newcommand{\Pp}{\mathbb P}
\newcommand{\E}{\mathbb E}
\newcommand{\cV}{\mathscr V}
\newcommand{\cW}{\mathscr W}
\newcommand{\cL}{\mathscr L}
\newcommand{\cP}{\mathscr P}
\newcommand{\eps}{\varepsilon}
\newcommand{\ee}{\mathrm e}
\newcommand{\supp}{\operatorname{supp}}
\newcommand{\Span}{\operatorname{span}}

\newcommand{\rank}{\operatorname{rank}}
\newcommand{\Hh}{\mathbb H}

\newcommand{\lcm}{\operatorname{lcm}}

\title{Close Divisors of Typical Integers:\\
The Ford--Green--Koukoulopoulos Conjecture
\thanks{
This work was supported by the National Science Foundation
of China (Nos.~12471329 and 12061059).}}

\author{Yaping Mao \footnote{Academy of Plateau
Science and Sustainability, and School of Mathematics and Statistics, Qinghai Normal University, Xining, Qinghai 810008, China. {\tt yapingmao@outlook.com; myp@qhnu.edu.cn}}, \ \ Yanyan Song\footnote{Corresponding author: School of Mathematics and Statistics, Qinghai Normal University, Xining, Qinghai 810008, China. {\tt songyanyan@hrbeu.edu.cn}}
}
\date{}
\hypersetup{
 pdftitle={Close divisors of typical integers:
the Ford--Green--Koukoulopoulos conjecture},
  pdfauthor={Yaping Mao, Yanyan Song},
  pdfsubject={Multiplicative number theory; close divisors; logarithmic random sets; entropy systems},
  pdfkeywords={close divisors, logarithmic random sets, equal subset sums, entropy method, harmonic reconstruction}
}

\begin{document}
\maketitle

\begin{abstract}
For an integer $k\geq2$, let $\alpha_k$ be the supremum of the real
numbers $a$ for which almost every integer $n\geq2$ has divisors
$d_1<\cdots<d_k\mid n$ satisfying
$d_k\leq d_1\bigl(1+(\log n)^{-a}\bigr).$
Let $\mathcal A\subseteq\N$ be the logarithmic random set in which
the events $m\in\mathcal A$ are mutually independent and
$\Pp(m\in\mathcal A)=1/m$ for every $m\geq1$.
For a finite set $B\subseteq\N$, write $\Sigma(B)=\sum_{b\in B}b,
\Sigma(\varnothing)=0$ and 
$m(B)=\max_{s\in\Z}
\#\{C\subseteq B\mid\Sigma(C)=s\},$
and define
\[
\beta_k=
\sup\left\{
c<1\,\middle|\,
\lim_{D\to\infty}
\Pp\bigl(m(\mathcal A\cap(D^c,D])\geq k\bigr)=1
\right\}.
\]
Ford, Green and Koukoulopoulos proved
$\alpha_k\geq\beta_k/(1-\beta_k)$ and conjectured that equality
holds for every fixed $k\geq2$. In this paper, we confirm their
conjecture. More precisely, for every fixed
$a>\beta_k/(1-\beta_k)$, almost every integer $n\geq2$ has no
divisors $d_1<\cdots<d_k\mid n$ satisfying
$d_k\leq d_1\bigl(1+(\log n)^{-a}\bigr)$.
We also correct local errors in their paper
[\emph{Invent. Math.} 232 (2023), 1027--1160], concerning the
finite-subflag reduction, the residual-sum count, the moment
estimate and the lattice adjustment. These corrections preserve the
entropy-threshold
comparison used in our proof.\\[0.3mm]
\end{abstract}

\tableofcontents

\section{Introduction}\label{sec:intro}

The distribution and concentration of divisors have been studied by
Hooley \cite{Hooley1979}, Hall and Tenenbaum
\cite{HallTenenbaum1982}, and Maier and Tenenbaum
\cite{MaierTenenbaum1985,MaierTenenbaum2009}.
For general accounts, see \cite{HallTenenbaum1988,Tenenbaum2013}.
The present paper concerns a fixed number of distinct divisors in
a multiplicative interval whose relative length tends to zero.

\subsection{Close divisors and logarithmic random sets}
\label{subsec:close-divisors}

For an integer $k\geq2$, let $\alpha_k$ be the supremum of the real
numbers $a$ for which almost every integer $n\geq2$ has divisors
$d_1<\cdots<d_k\mid n$ satisfying
\begin{equation}\label{eq:alpha-window}
d_k\leq d_1\bigl(1+(\log n)^{-a}\bigr).
\end{equation}
Here and throughout, a property holds for almost every integer if the
number of exceptions in $[1,x]$ is $o(x)$ as $x\to\infty$.
Thus $\alpha_k$ measures the smallest relative scale, in powers of
$\log n$, on which $k$ distinct divisors occur for almost every $n$.

The case $k=2$ originates in work of Erd\H{o}s
\cite{Erdos1964}. Erd\H{o}s and Hall \cite{ErdosHall1979}
proved the upper bound $\alpha_2\leq\log3-1$, and Maier and
Tenenbaum \cite{MaierTenenbaum1984} established the lower
bound. Hence
$\alpha_2=\log3-1.$

The structure of the ordered divisor sequence and the ratios of
consecutive divisors were also studied by Erd\H{o}s and Tenenbaum
\cite{ErdosTenenbaum1981,ErdosTenenbaum1983}.
For larger fixed $k$, simultaneous relations among several divisors
lead to an additive problem involving several subset sums.

Let $\mathcal A\subseteq\N$ be the \emph{logarithmic random set} in which
\[
\Pp(m\in\mathcal A)=\frac1m
\qquad(m\geq1),
\]
and the events are mutually independent. For any finite set $S$, let $\#S$ denote the number of elements of $S$.
For a finite set $B\subseteq\N$, write
\[
\Sigma(B)=\sum_{b\in B}b,
\qquad
\Sigma(\varnothing)=0,
\qquad
m(B)=\max_{s\in\Z}
\#\{C\subseteq B\mid\Sigma(C)=s\}.
\]
Thus $m(B)\geq k$ precisely when $B$ has $k$ pairwise distinct
subsets with the same sum. Ford, Green and Koukoulopoulos
\cite{FGK} introduced the threshold
\begin{equation*}
\beta_k=
\sup\left\{
 c<1\,\middle|\,
 \lim_{D\to\infty}
 \Pp\bigl(m(\mathcal A\cap(D^c,D])\geq k\bigr)=1
\right\}.
\end{equation*}

Ford, Green and Koukoulopoulos proved in
\cite[Theorem~6]{FGK} that
\begin{equation}\label{eq:fgk-lower}
\alpha_k\geq\frac{\beta_k}{1-\beta_k}
\qquad(k\geq2).
\end{equation}
They formulated the converse as the following conjecture
\cite[Conjecture~3]{FGK}.

\begin{conjecture}{\upshape \cite{FGK}}
For every fixed integer $k\geq2$, we have $\alpha_k=\frac{\beta_k}{1-\beta_k}.$
\end{conjecture}

For each real number $a$, define
\begin{equation}\label{eq:E-ak-definition}
\mathcal E_{a,k}
=
\left\{n\in\N\,\middle|\,
\begin{gathered}
 n\geq2,\quad \exists\,d_1<\cdots<d_k\mid n, \;
 d_k\leq d_1\bigl(1+(\log n)^{-a}\bigr)
\end{gathered}
\right\}.
\end{equation}
Then $\alpha_k$ is the supremum of the real numbers $a$ for which
$\mathcal E_{a,k}$ has natural density one, that is, 
$$
\lim_{x\to\infty}
\frac{\#\bigl(\mathcal E_{a,k}\cap[1,x]\bigr)}{x}=1.
$$
Our main result gives the reverse inequality in
\eqref{eq:fgk-lower} and the stronger density-zero conclusion above the threshold.

\begin{theorem}\label{thm:main}
For every fixed integer $k\geq2$, we have $$\alpha_k=\frac{\beta_k}{1-\beta_k}.$$
Moreover, for every fixed $a>\beta_k/(1-\beta_k)$, we have
\[
\lim_{x\to\infty}
\frac{\#\bigl(\mathcal E_{a,k}\cap[1,x]\bigr)}{x}=0.
\]
\end{theorem}

\subsection{Entropy systems and threshold identification}
\label{subsec:entropy-systems}

We use the entropy framework of Ford, Green and Koukoulopoulos
\cite[Definitions~3.2--3.6]{FGK}.
All vector spaces in this subsection are over $\Q$, and
$\one=(1,\ldots,1)\in\Q^k$.
We write
\[
\begin{aligned}
\{0,1\}^k
&=
\left\{
(x_1,\ldots,x_k)\in\Q^k
\,\middle|\,
x_i\in\{0,1\}\quad(1\leq i\leq k)
\right\},\\
\{-1,0,1\}^k
&=
\left\{
(x_1,\ldots,x_k)\in\Q^k
\,\middle|\,
x_i\in\{-1,0,1\}\quad(1\leq i\leq k)
\right\}.
\end{aligned}
\]

\begin{definition}[Finite nonnegative measures]
\label{def:finite-nonnegative-measure}
Let $S\subseteq\Q^k$ be finite. For each $\omega\in S$,
let $\mu(\omega)$ be a nonnegative real number.
The family $\mu=(\mu(\omega))_{\omega\in S}$ is called a \emph{finite
nonnegative measure} on $S$. For every $A\subseteq S$, define
\[
\mu(A)=\sum_{\omega\in A}\mu(\omega).
\]
The measure $\mu$ is a probability measure if $\mu(S)=1$.
It is the zero measure if $\mu(\omega)=0$ for every $\omega\in S$.
\end{definition}

All logarithms are natural. All measures are extended by zero outside
their specified finite sets. We use $0\log0=0$ in every entropy formula.

\begin{definition}[Flags and entropy systems]\label{def:system}
A \emph{flag} $\cV=(V_0,\ldots,V_r)$ is a sequence of
vector subspaces satisfying
\[
\langle\one\rangle=V_0\leq V_1\leq\cdots\leq V_r\leq\Q^k.
\]
It is called \emph{complete} if
\[
\dim(V_j/V_{j-1})=1
\qquad(1\leq j\leq r).
\]

An \emph{entropy system} is a triple
$(\cV,\mathbf c,\boldsymbol\mu)$, where
$\cV=(V_0,\ldots,V_r)$ is a flag,
$\mathbf c=(c_1,\ldots,c_{r+1})$, and
$\boldsymbol\mu=(\mu_1,\ldots,\mu_r)$,
satisfying the following conditions.
\begin{enumerate}[label=\textup{(\alph*)},leftmargin=2.2em]
\item The spaces $V_0,\ldots,V_r$ are pairwise distinct, that is, $\langle\one\rangle=V_0<V_1<\cdots<V_r\leq\Q^k.$
and each $V_j$ is spanned by points of $\{0,1\}^k$.

\item The terminal space $V_r$ is \emph{non-degenerate},
meaning that, for every $1\leq s<t\leq k$, there exists
$x=(x_1,\ldots,x_k)\in V_r$ with $x_s\neq x_t$.

\item The thresholds satisfy
$1\geq c_1\geq c_2\geq\cdots\geq c_{r+1}\geq0$.

\item For every $1\leq j\leq r$, the measure $\mu_j$
is a probability measure on $V_j\cap\{0,1\}^k$.
\end{enumerate}
An entropy system $(\cV,\mathbf c,\boldsymbol\mu)$
is called \emph{complete} if $\cV$ is complete.
\end{definition}
For a finitely supported probability measure $\mu$ and a subspace
$W\leq\Q^k$, put
\begin{equation}\label{eq:defini-mu}
\mu(W+x)=\sum_{y\in W+x}\mu(y).
\end{equation}
The entropy of the induced distribution on the cosets of $W$ is
\begin{equation}\label{eq:coset-entropy}
\Hh_\mu(W)=-\sum_x\mu(x)\log\mu(W+x).
\end{equation}
This is Shannon entropy applied to the coset distribution.
We use its standard concavity and continuity properties, for which
see \cite[Chapter~2]{CoverThomas2006}.

A \emph{subflag} of $\cV$, denoted by $\cV'\leq\cV$,
is a sequence
\[
\langle\one\rangle=V'_0\leq V'_1\leq\cdots\leq V'_r,
\qquad
V'_j\leq V_j\quad(0\leq j\leq r).
\]
It is called a \emph{proper subflag}, denoted by
$\cV'<\cV$, if $V'_j<V_j$ for at least one
$1\leq j\leq r$. For subspaces $U\leq V$, the notation $V/U$ denotes
the quotient space, with
\[
\dim(V/U)=\dim V-\dim U.
\]
For a subflag $\cV'\leq\cV$, define
\begin{equation}\label{eq:e-value}
\mathrm e(\cV',\mathbf c,\boldsymbol\mu)
=
\sum_{j=1}^{r}(c_j-c_{j+1})\Hh_{\mu_j}(V'_j)
+
\sum_{j=1}^{r}c_j\dim(V'_j/V'_{j-1}).
\end{equation}
Since $\mu_j$ is supported on $V_j$, one has
$\Hh_{\mu_j}(V_j)=0$, and therefore
\begin{equation}\label{eq:full-e}
\mathrm e(\cV,\mathbf c,\boldsymbol\mu)
=
\sum_{j=1}^{r}c_j\dim(V_j/V_{j-1}).
\end{equation}

The system satisfies the entropy condition if
\begin{equation}\label{e1}
\mathrm e(\cV',\mathbf c,\boldsymbol\mu)
\geq
\mathrm e(\cV,\mathbf c,\boldsymbol\mu)
\qquad(\cV'\leq\cV).
\end{equation}
It satisfies the strict entropy condition if $\mathrm e(\cV',\mathbf c,\boldsymbol\mu)
>
\mathrm e(\cV,\mathbf c,\boldsymbol\mu) $ for $\cV'<\cV.$

Define
\[
\gamma_k
=
\sup\left\{
 c_{r+1}\,\middle|\,
 (\cV,\mathbf c,\boldsymbol\mu)
 \text{ satisfies the entropy condition}
\right\}
\]
and
\[
\widetilde\gamma_k
=
\sup\left\{
 c_{r+1}\,\middle|\,
 (\cV,\mathbf c,\boldsymbol\mu)
 \text{ satisfies the strict entropy condition}
\right\}.
\]
The suprema range over all entropy systems for the given $k$.
Ford, Green and Koukoulopoulos 
\cite[Theorem~7]{FGK} 
state the comparison
\begin{equation}\label{eq:FGK-sandwich}
\widetilde\gamma_k\leq\beta_k\leq\gamma_k.
\end{equation}

Proposition~\ref{prop:threshold-comparison} below establishes
this comparison using the finite-subflag reduction of
Lemma~\ref{lem:finite-subflags} and the residual-class estimate of
Lemma~\ref{lem:residual-sum-bound}, with the local corrections
explained in Section~\ref{sec:FGK-formalism}.
For the lower inequality, we verify the hypotheses of
\cite[Proposition~5.5]{FGK} and use its proof with the local
adjustments recorded in Appendix~\ref{app:fgk-lower-corrections}.
These corrections retain the threshold inequalities in
\eqref{eq:FGK-sandwich}.

They asked whether equality holds throughout
\cite[Remark~3.1]{FGK}.
The entropy part of our argument establishes the following result.

\begin{theorem}\label{thm:entropy-equality}
For every fixed integer $k\geq2$, we have
\begin{equation}\label{eq:entropy-equality}
\widetilde\gamma_k=\beta_k=\gamma_k.
\end{equation}
\end{theorem}

Theorem~\ref{thm:strictification} proves that a system satisfying
the entropy condition with positive final threshold $b$ can be
replaced by a strict system on a complete refinement with any
prescribed final threshold in $(0,b)$.
Taking suprema and applying \eqref{eq:FGK-sandwich} gives
Theorem~\ref{thm:entropy-equality}.
The same result, together with a finite reduction of the subflags
and compactness of the system parameters, yields a uniform negative
gap when the final threshold is bounded away from $\beta_k$ on
the right.

This paper is organized as follows.
Section~\ref{sec:FGK-formalism} collects the entropy and counting
preliminaries, explains the required local corrections to the FGK
framework, and establishes the threshold comparison.
Section~\ref{sec:arithmetic-setup} develops the arithmetic
reduction.
Sections~\ref{sec:nonsat} and \ref{sec:sat} give the counting
estimates for the two ranges of the least common multiple
of the external parts.
In Section~\ref{sec:main-proof}, we complete the proof of
Theorem~\ref{thm:main}.

\section[Entropy preliminaries and local corrections to the FGK framework]{Entropy preliminaries and local corrections to the FGK framework}\label{sec:FGK-formalism}
We first normalize the thresholds and then give corrected versions
of the finite-subflag reduction and the residual-sum count used in
\cite[Section~4]{FGK}. The examples below explain why the differences
of cube points and the additional support condition are needed.
Proposition~\ref{prop:threshold-comparison} then makes explicit the
connection with the threshold inequalities.

\subsection{Normalization}

\begin{lemma}\label{lem:threshold-normalization}
Let $(\cV,\mathbf c,\boldsymbol\mu)$ be an entropy system with
$c_1>0$.  Define the normalized threshold sequence $\widetilde{\mathbf c}
=
(\widetilde c_1,\ldots,\widetilde c_{r+1}),$ where $\widetilde c_j=\frac{c_j}{c_1}.$
Then the system $(\cV,\widetilde{\mathbf c},\boldsymbol\mu)$ has the same entropy or strict entropy property as
$(\cV,\mathbf c,\boldsymbol\mu)$, and every entropy gap is multiplied by
the factor $c_1^{-1}$.  In particular, in the optimization defining
$\gamma_k$ or $\widetilde\gamma_k$, one may restrict to systems satisfying
$c_1=1$ without decreasing the final threshold.
\end{lemma}

\begin{proof}
Fix an arbitrary subflag $\cV'\leq\cV$.
Replacing each threshold $c_j$ by
$\widetilde c_j=\frac{c_j}{c_1},$
we obtain $\widetilde c_j-\widetilde c_{j+1}
=\frac{c_j-c_{j+1}}{c_1}.$
Therefore, applying \eqref{eq:e-value} gives
\[
\begin{aligned}
\mathrm e(\cV',\widetilde{\mathbf c},\boldsymbol\mu)
&=
\sum_{j=1}^r
\frac{c_j-c_{j+1}}{c_1}
\Hh_{\mu_j}(V'_j)
+
\sum_{j=1}^r
\frac{c_j}{c_1}
\dim(V'_j/V'_{j-1})\\
&=
\frac1{c_1}
\mathrm e(\cV',\mathbf c,\boldsymbol\mu).
\end{aligned}
\]
For the full flag, the same identity gives $\mathrm e(\cV,\widetilde{\mathbf c},\boldsymbol\mu)
=\frac1{c_1}
\mathrm e(\cV,\mathbf c,\boldsymbol\mu).$
Consequently, the entropy gap satisfies
\[
\begin{aligned}
&
\mathrm e(\cV',\widetilde{\mathbf c},\boldsymbol\mu)
-
\mathrm e(\cV,\widetilde{\mathbf c},\boldsymbol\mu)
=
\frac1{c_1}
\left(
\mathrm e(\cV',\mathbf c,\boldsymbol\mu)
-
\mathrm e(\cV,\mathbf c,\boldsymbol\mu)
\right).
\end{aligned}
\]
Since $c_1>0$, multiplication by $c_1^{-1}$ preserves the sign of every
entropy gap.  Hence the entropy condition and the strict entropy condition
are both preserved under this normalization.

It remains to verify that the normalized thresholds are admissible.  Since $1\geq c_1\geq c_j\geq0,$
we have $0\leq\frac{c_j}{c_1}\leq1,$
and the normalized sequence still satisfies
$1=\widetilde c_1
\geq\cdots\geq
\widetilde c_{r+1}\geq0.$
Moreover, the final threshold becomes
$\widetilde c_{r+1}=\frac{c_{r+1}}{c_1}.$
Because $c_{r+1}\leq c_1$ and $c_1\leq1$, we have $\frac{c_{r+1}}{c_1}\geq c_{r+1}.$
Thus the normalization does not decrease the final threshold.  Therefore,
in both optimization problems defining $\gamma_k$ and $\widetilde\gamma_k$, it is
sufficient to consider systems with $c_1=1$.
\end{proof}

\subsection{Finite classification of coset entropies}\label{subsec:finite-correction}

The equivalence relation used in the proof of
\cite[Lemma~4.6]{FGK} identifies subflags whose corresponding
spaces have the same dimensions and the same intersections
with $\{0,1\}^k$. The following example shows that these data
do not determine the induced coset entropies or the
$\mathrm e$-values. Thus this criterion does not justify replacing all subflags in an equivalence class by a single representative.

\paragraph{A counterexample to the original classification.}
In $\Q^3$, define
$$
U=\Span_{\Q}\{\one,(1,-1,0)\},
\qquad
W=\Span_{\Q}\{\one,(1,0,-1)\}.
$$
Both spaces have dimension two. Their spanning vectors satisfy
$x_1+x_2=2x_3$ and $x_1+x_3=2x_2$, respectively. Since each
of these equations defines a two-dimensional subspace, we have
$$
\begin{aligned}
U&=\{(x_1,x_2,x_3)\in\Q^3:x_1+x_2=2x_3\},\\
W&=\{(x_1,x_2,x_3)\in\Q^3:x_1+x_3=2x_2\}.
\end{aligned}
$$
For a point of $U\cap\{0,1\}^3$, the case $x_3=0$ forces
$x_1=x_2=0$, while $x_3=1$ forces $x_1=x_2=1$.
The same argument for $W$, according to whether $x_2=0$ or
$x_2=1$, gives
$$
U\cap\{0,1\}^3=W\cap\{0,1\}^3=\{\mathbf0,\one\}.
$$

Set $\mathbf a=(1,0,0)$ and $\mathbf b=(0,1,0)$, and let
$\mu$ be the probability measure supported on
$\{\mathbf a,\mathbf b\}$ with
$\mu(\mathbf a)=\mu(\mathbf b)=\frac12.$
Thus $\mu(x)=0$ for every
$x\in\Q^3\setminus\{\mathbf a,\mathbf b\}$.

Since $\mathbf a-\mathbf b=(1,-1,0)\in U$ and $U$ is a
vector subspace, we also have $\mathbf b-\mathbf a\in U$.
Consequently,
$$
\mathbf a=\mathbf a+\mathbf0\in\mathbf a+U,
\qquad
\mathbf b=\mathbf a+(\mathbf b-\mathbf a)\in\mathbf a+U.
$$
Hence both points in the support of $\mu$ belong to
$\mathbf a+U$. By \eqref{eq:defini-mu}, we have
$$
\begin{aligned}
\mu(\mathbf a+U)
=\sum_{x\in\mathbf a+U}\mu(x)
=\mu(\mathbf a)+\mu(\mathbf b)
  +\sum_{\substack{x\in\mathbf a+U\\
x\notin\{\mathbf a,\mathbf b\}}}\mu(x)
=1.
\end{aligned}
$$
Moreover, $\mathbf a-\mathbf b\in U$ implies
$\mathbf a+U=\mathbf b+U$. Therefore,
$$
\mu(\mathbf a+U)=\mu(\mathbf b+U)=1.
$$

By \eqref{eq:coset-entropy},
$$
\Hh_\mu(U)=-\frac12\log1-\frac12\log1=0.
$$
On the other hand, $(1,-1,0)\notin W$, since
$1+0\neq2(-1)$. Hence $\mathbf a+W$ and $\mathbf b+W$ are
distinct cosets, each containing exactly one point of the support
of $\mu$. Therefore,
$$
\mu(\mathbf a+W)=\mu(\mathbf b+W)=\frac12,
\qquad
\Hh_\mu(W)
=-\frac12\log\frac12-\frac12\log\frac12=\log2.
$$

To place the example within the class of subflags under
consideration, define
$$
V_0=\langle\one\rangle,
\qquad
V_1=\Span_{\Q}\{\one,\mathbf a\},
\qquad
V_2=\Q^3.
$$
We have $\mathbf a\notin V_0$. Also, every vector of $V_1$
has equal second and third coordinates, so $\mathbf b\notin V_1$.
Thus $\one,\mathbf a,\mathbf b$ form a basis of $\Q^3$
consisting of cube points, and
$$
\cV:V_0<V_1<V_2
$$
is a complete cube-spanned flag with non-degenerate terminal
space. Since $V_0\leq U,W\leq V_2$, the sequences
$$
\mathcal U=(V_0,V_0,U),
\qquad
\mathcal W=(V_0,V_0,W)
$$
are proper subflags of $\mathcal V$.
Here the definition of a subflag permits repeated spaces and
does not require its component spaces to be cube-spanned.
These two subflags have the same dimension sequence $(1,1,2)$
and the same intersections with the cube at every level,
but their final coset entropies under $\mu$ differ.

This also yields different $\mathrm e$-values. Let $\mu_1$
be any probability measure supported on $V_1\cap\{0,1\}^3$,
set $\mu_2=\mu$ and $\boldsymbol\mu=(\mu_1,\mu_2)$, and choose
$\mathbf c=(c_1,c_2,c_3)$ with
$1\geq c_1>c_2>c_3\geq0$.
Both subflags have dimension increments $0$ and $1$, and their
first entropy terms coincide. Consequently,
$$
\begin{aligned}
\mathrm e(\mathcal W,\mathbf c,\boldsymbol\mu)
-\mathrm e(\mathcal U,\mathbf c,\boldsymbol\mu)
=(c_2-c_3)\bigl(\Hh_\mu(W)-\Hh_\mu(U)\bigr)
=(c_2-c_3)\log2>0.
\end{aligned}
$$

To obtain a classification preserving these quantities, we use
differences of cube points. For $T\leq\Q^k$ and
$\omega,\omega'\in\{0,1\}^k$, the points $\omega$ and $\omega'$
belong to the same $T$-coset if and only if $\omega'-\omega\in T$.
Since every such difference lies in $\{-1,0,1\}^k$,
$$
(\omega+T)\cap\{0,1\}^k
=
\left\{
\omega'\in\{0,1\}^k:
\omega'-\omega\in T\cap\{-1,0,1\}^k
\right\}.
$$
Thus $T\cap\{-1,0,1\}^k$ determines the coset partition of the
cube and hence the induced entropy for every probability measure
supported on the cube. We therefore record
$$
\bigl(\dim T,\;T\cap\{-1,0,1\}^k\bigr)
$$
at each level of a subflag. These data preserve both the entropy
terms and the dimension increments in $\mathrm e$.
The following lemma establishes the resulting finite-subflag
reduction.

\begin{lemma}\label{lem:finite-subflags}
Let $\cV$ be a flag
$\langle\one\rangle=V_0<V_1<\cdots<V_r\leq\Q^k.$
There exists a finite family $\mathscr F(\cV)$ of subflags
of $\cV$ such that
\[
\#\mathscr F(\cV)
\leq
\bigl((k+1)2^{3^k}\bigr)^k
=O_k(1),
\]
and, for every subflag $\cV'\leq\cV$, there exists
$\cV''\in\mathscr F(\cV)$ satisfying
\[
\dim V'_j=\dim V''_j,
\qquad
\Hh_\mu(V'_j)=\Hh_\mu(V''_j)
\qquad(0\leq j\leq r)
\]
for every probability measure $\mu$ supported on
$\{0,1\}^k$.

Consequently, for every threshold sequence $\mathbf c$
and every family $\boldsymbol\mu=(\mu_1,\ldots,\mu_r)$
of probability measures supported on $\{0,1\}^k$, we have
$\mathrm e(\cV',\mathbf c,\boldsymbol\mu)
=\mathrm e(\cV'',\mathbf c,\boldsymbol\mu).$
Moreover, if $\cV'<\cV$, then
$\cV''<\cV$.
\end{lemma}

\begin{proof}
For two subflags $\cV',\cV''\leq\cV$, write
$\cV'\sim\cV''$ if
\[
\dim V'_j=\dim V''_j,
\qquad
V'_j\cap\{-1,0,1\}^k
=
V''_j\cap\{-1,0,1\}^k
\qquad(0\leq j\leq r).
\]
These equalities define an equivalence relation on the set
of subflags of $\cV$.

For each $0\leq j\leq r$, we have
\[
\dim V'_j\in\{0,1,\ldots,k\},
\qquad
V'_j\cap\{-1,0,1\}^k\subseteq\{-1,0,1\}^k.
\]
Since $\#\{-1,0,1\}^k=3^k$, this set has $2^{3^k}$
subsets. Hence there are at most
$(k+1)2^{3^k}$ possible pairs
$\bigl(\dim V'_j,\,
V'_j\cap\{-1,0,1\}^k\bigr)$
at each index $j$. There are therefore at most
$\bigl((k+1)2^{3^k}\bigr)^{r+1}$
equivalence classes.

Since the inclusions in $\cV$ are strict and
$\dim V_0=1$,
\[
r
\leq
\sum_{j=1}^{r}
\bigl(\dim V_j-\dim V_{j-1}\bigr)
=
\dim V_r-1
\leq k-1.
\]
Choose one subflag from each equivalence class, and let
$\mathscr F(\cV)$ be the family of these subflags.
Then
\[
\#\mathscr F(\cV)
\leq
\bigl((k+1)2^{3^k}\bigr)^{r+1}
\leq
\bigl((k+1)2^{3^k}\bigr)^k.
\]
This construction depends only on $\cV$ and $k$.

Fix $\cV'\leq\cV$, and let
$\cV''\in\mathscr F(\cV)$ be the chosen subflag in the
same equivalence class. Thus
\[
\dim V'_j=\dim V''_j,
\qquad
V'_j\cap\{-1,0,1\}^k
=
V''_j\cap\{-1,0,1\}^k
\qquad(0\leq j\leq r).
\]
We next prove the entropy equalities. Fix $0\leq j\leq r$ and $\omega\in\{0,1\}^k$.
For every $\omega'\in\{0,1\}^k$, we have $\omega'-\omega\in\{-1,0,1\}^k.$
Consequently,
\[
\begin{aligned}
(V'_j+\omega)\cap\{0,1\}^k
&=
\left\{
\omega'\in\{0,1\}^k
\,\middle|\,
\omega'-\omega\in
V'_j\cap\{-1,0,1\}^k
\right\}\\
&=
\left\{
\omega'\in\{0,1\}^k
\,\middle|\,
\omega'-\omega\in
V''_j\cap\{-1,0,1\}^k
\right\}\\
&=
(V''_j+\omega)\cap\{0,1\}^k.
\end{aligned}
\]

Let $\mu$ be any probability measure supported on
$\{0,1\}^k$. Then \eqref{eq:defini-mu} gives
\[
\begin{aligned}
\mu(V'_j+\omega)
&=
\sum_{\omega'\in(V'_j+\omega)\cap\{0,1\}^k}
\mu(\omega')=
\sum_{\omega'\in(V''_j+\omega)\cap\{0,1\}^k}
\mu(\omega')
=
\mu(V''_j+\omega)
\qquad(\omega\in\{0,1\}^k).
\end{aligned}
\]
Since $\mu$ is a probability measure, $\mu(\omega)\geq0$
for every $\omega\in\{0,1\}^k$.
Terms with $\mu(\omega)=0$ contribute zero to the entropy sum.
For every $\omega\in\{0,1\}^k$ with $\mu(\omega)>0$,
the preceding equality gives
\[
\mu(V'_j+\omega)
=
\mu(V''_j+\omega)
\geq\mu(\omega)>0.
\]
Thus all logarithms in the following sums are defined.
By \eqref{eq:coset-entropy},
\[
\begin{aligned}
\Hh_\mu(V'_j)
=
-\sum_{\substack{
\omega\in\{0,1\}^k\\
\mu(\omega)>0
}}
\mu(\omega)\log\mu(V'_j+\omega)
=
-\sum_{\substack{
\omega\in\{0,1\}^k\\
\mu(\omega)>0
}}
\mu(\omega)\log\mu(V''_j+\omega)
=
\Hh_\mu(V''_j).
\end{aligned}
\]
Since $j$ and $\mu$ were arbitrary, these equalities hold
for every $0\leq j\leq r$ and every probability measure
supported on $\{0,1\}^k$.

Moreover, for $1\leq j\leq r$, the dimension equalities give
\[
\begin{aligned}
\dim(V'_j/V'_{j-1})
&=
\dim V'_j-\dim V'_{j-1}=
\dim V''_j-\dim V''_{j-1}=
\dim(V''_j/V''_{j-1}).
\end{aligned}
\]
Therefore, for every threshold sequence $\mathbf c$
and every family
$\boldsymbol\mu=(\mu_1,\ldots,\mu_r)$
as in the statement, \eqref{eq:e-value} yields
\[
\begin{aligned}
\mathrm e(\cV',\mathbf c,\boldsymbol\mu)
&=
\sum_{j=1}^{r}
(c_j-c_{j+1})\Hh_{\mu_j}(V'_j)
+
\sum_{j=1}^{r}
c_j\dim(V'_j/V'_{j-1})\\
&=
\sum_{j=1}^{r}
(c_j-c_{j+1})\Hh_{\mu_j}(V''_j)
+
\sum_{j=1}^{r}
c_j\dim(V''_j/V''_{j-1})\\
&=
\mathrm e(\cV'',\mathbf c,\boldsymbol\mu).
\end{aligned}
\]

Finally, suppose that $\cV'<\cV$.
Then there exists $j_0\in\{1,\ldots,r\}$ such that
$V'_{j_0}$ is a proper subspace of $V_{j_0}$.
Hence
\[
\dim V''_{j_0}
=
\dim V'_{j_0}
<
\dim V_{j_0}.
\]
Since $V''_{j_0}\leq V_{j_0}$, it follows that
$V''_{j_0}$ is also a proper subspace of $V_{j_0}$.
Thus $\cV''<\cV$, completing the proof.
\end{proof}

\subsection{The support condition in the residual-sum count}\label{subsec:residual-correction}

The following example shows that the layer-frequency conditions
in \cite[Definition~4.3]{FGK} alone do not imply the counting
bound in \cite[Lemma~4.5 and Corollary~4.7]{FGK}.
These conditions prescribe the labels only on the layers
$(D^{c_{j+1}},D^{c_j}]$, leaving the elements above $D^{c_1}$
unrestricted when $c_1<1$. In the example below, these elements
produce at least $D^{(\log2)/2}$ distinct classes modulo
$\langle\one\rangle$, although the entropy exponent in the
asserted bound is zero. Thus the labels above $D^{c_1}$ must
also be controlled.

\paragraph{A counterexample without the support condition.}
We construct a complete entropy system and a sequence of regular sets for which the conditions in \cite[Definition~4.3]{FGK} hold but the bound in \cite[Lemma~4.5]{FGK} fails.

Consider the entropy system
$(\cV,\mathbf c,\boldsymbol\mu)$ given by
$$
\cV:\quad V_0=\langle\one\rangle<V_1=\Q^2,
\qquad
\mathbf c=\left(\frac12,\frac14\right),
\qquad
\boldsymbol\mu=(\mu_1),
$$
where
$$
\mu_1(\mathbf0)=1,
\qquad
\mu_1(\omega)=0
\quad
\bigl(\omega\in\{0,1\}^2\setminus\{\mathbf0\}\bigr).
$$
Both spaces are spanned by points of $\{0,1\}^2$,
the terminal space $V_1$ is non-degenerate, and
$\dim(V_1/V_0)=1$.

Let $N$ be a sufficiently large positive integer divisible
by four, and put
$$
D=e^N,
\qquad
x_j=\lfloor e^j\rfloor\quad(1\leq j\leq N),
\qquad
B=\{x_j:N/4<j\leq N\}.
$$
Since $e^{j-1}<x_j\leq e^j$, the integers $x_j$ are
strictly increasing. For $N/4<j\leq N$, we have
$j-1\geq N/4$, so
$$
D^{1/4}\leq e^{j-1}<x_j\leq e^j\leq D.
$$
Thus $B\subseteq(D^{1/4},D]\cap\Z$.

Fix a real number $t\in[N/4,N]$.
There are exactly $\lfloor t\rfloor-N/4$ integers $j$
with $N/4<j\leq t$, and each satisfies
$x_j\leq e^j\leq e^t$.
If an index $j>t$ also satisfies $x_j\leq e^t$, then
$e^{j-1}<x_j\leq e^t$ implies $j<t+1$.
Thus the only possible additional index is
$j=\lfloor t\rfloor+1$, and consequently
$$
\lfloor t\rfloor-\frac N4
\leq
\#\bigl(B\cap(D^{1/4},e^t]\bigr)
\leq
\lfloor t\rfloor-\frac N4+1,
$$
and therefore
$$
\left|
\#\bigl(B\cap(D^{1/4},e^t]\bigr)
-\left(t-\frac N4\right)
\right|\leq1.
$$

Let $1/4\leq\alpha\leq\beta\leq1$. Then
$N/4\leq\alpha N\leq\beta N\leq N$, so the preceding
inequality applies at both $t=\alpha N$ and $t=\beta N$.
Since $D=e^N$, we have
$$
\#\bigl(B\cap(D^\alpha,D^\beta]\bigr)
=
\#\bigl(B\cap(D^{1/4},e^{\beta N}]\bigr)
-
\#\bigl(B\cap(D^{1/4},e^{\alpha N}]\bigr).
$$
Subtracting
$(\beta-\alpha)\log D
=(\beta N-N/4)-(\alpha N-N/4)$
and applying the triangle inequality, we obtain
$$
\begin{aligned}
\left|
\#\bigl(B\cap(D^\alpha,D^\beta]\bigr)
-(\beta-\alpha)\log D
\right|
&\leq
\left|
\#\bigl(B\cap(D^{1/4},e^{\beta N}]\bigr)
-\left(\beta N-\frac N4\right)
\right|\\
&\qquad+
\left|
\#\bigl(B\cap(D^{1/4},e^{\alpha N}]\bigr)
-\left(\alpha N-\frac N4\right)
\right|\\
&\quad\leq1+1=2.
\end{aligned}
$$
Thus $B$ satisfies the hypothesis of
\cite[Lemma~4.5]{FGK}.
Moreover,
$$
B\cap(D^{1/4},D^{1/2}]
=\{x_j:N/4<j\leq N/2\},
\qquad
B\cap(D^{1/2},D]
=\{x_j:N/2<j\leq N\}.
$$
These two sets have $N/4$ and $N/2$ elements,
respectively.

Following \cite[Definition~4.3]{FGK}, let
$\mathcal L_{\cV,\mathbf c,\boldsymbol\mu}(B)$ denote
the set of all cosets
$$
\left(
\sum_{\omega\in\{0,1\}^2}
\omega\sum_{x\in B_\omega}x
\right)+V_0
\in\Q^2/V_0,
$$
where $(B_\omega)_{\omega\in\{0,1\}^2}$ ranges over
all partitions
$B=\cup_{\omega\in\{0,1\}^2}B_\omega$
with empty parts allowed, satisfying
$$
\#\bigl(B_\omega\cap(D^{1/4},D^{1/2}]\bigr)
=
\mu_1(\omega)\#(B\cap(D^{1/4},D^{1/2}])
\qquad(\omega\in\{0,1\}^2).
$$
For each subset $S\subseteq B\cap(D^{1/2},D]$, define
a partition of $B$ by
$$
B_{(1,0)}=S,\qquad
B_{\mathbf0}=B\setminus S,\qquad
B_{(0,1)}=B_{\one}=\varnothing.
$$
Since $S\cap(D^{1/4},D^{1/2}]=\varnothing$, all elements
of $B\cap(D^{1/4},D^{1/2}]$ belong to $B_{\mathbf0}$.
Together with $\mu_1(\mathbf0)=1$, this verifies the
frequency conditions.

Since $B_{(0,1)}=B_{\one}=\varnothing$, these two parts
contribute zero to the vector sum. Using
$B_{\mathbf0}=B\setminus S$ and $B_{(1,0)}=S$, we obtain
$$
\begin{aligned}
\sum_{\omega\in\{0,1\}^2}
\omega\sum_{x\in B_\omega}x
&=
\mathbf0\sum_{x\in B\setminus S}x
+(1,0)\sum_{x\in S}x=
\left(\sum_{x\in S}x,\,0\right).
\end{aligned}
$$
The partition satisfies the required frequency conditions.
Therefore, by the definition of
$\mathcal L_{\cV,\mathbf c,\boldsymbol\mu}(B)$, we have
$$
\left(\sum_{x\in S}x,\,0\right)+V_0
\in\mathcal L_{\cV,\mathbf c,\boldsymbol\mu}(B).
$$

To show that different choices of $S$ give different
classes, Let $S,T\subseteq B\cap(D^{1/2},D]$ with $S\neq T$, and put
$$
j=\max\{\ell:x_\ell\in(S\setminus T)\cup(T\setminus S)\}.
$$
Without loss of generality, assume that
$x_j\in S\setminus T$. By the maximality of $j$,
$T\setminus S\subseteq\{x_\ell:N/2<\ell<j\}.$
Moreover,
$$
\sum_{N/2<\ell<j}x_\ell
<
\sum_{\ell=0}^{j-1}e^\ell
=
\frac{e^j-1}{e-1}
<
e^j-1
\leq x_j.
$$
Consequently,
$$
\begin{aligned}
\sum_{x\in S}x-\sum_{x\in T}x
&=
\sum_{x\in S\setminus T}x
-
\sum_{x\in T\setminus S}x
\geq
x_j-\sum_{N/2<\ell<j}x_\ell
>0.
\end{aligned}
$$
Since $V_0=\{(t,t):t\in\Q\}$, the difference
$
\left(
\sum_{x\in S}x-\sum_{x\in T}x,\,0
\right)
$
does not belong to $V_0$. Thus distinct subsets $S\subseteq B\cap(D^{1/2},D]$
give distinct elements of
$\mathcal L_{\cV,\mathbf c,\boldsymbol\mu}(B)$.
Since
$$
B\cap(D^{1/2},D]=\{x_j:N/2<j\leq N\}
$$
has $N/2$ elements, it has exactly $2^{N/2}$ subsets.
Each of these subsets yields a class in
$\mathcal L_{\cV,\mathbf c,\boldsymbol\mu}(B)$, and
the classes are pairwise distinct. Therefore,
$$
\#\mathcal L_{\cV,\mathbf c,\boldsymbol\mu}(B)
\geq
2^{\#(B\cap(D^{1/2},D])}
=
2^{N/2}.
$$

On the other hand, $\mu_1$ is concentrated at
$\mathbf0\in V_0$. Hence $\mu_1(V_0)=1$ and
$$
\Hh_{\mu_1}(V_0)
=
-\mu_1(\mathbf0)\log\mu_1(V_0)
=0.
$$
The subflag $\cV'=(V_0,V_0)$ therefore satisfies
$$
\mathrm e(\cV',\mathbf c,\boldsymbol\mu)
=
(c_1-c_2)\Hh_{\mu_1}(V_0)
+c_1\dim(V_0/V_0)
=0.
$$
Since the system and the subflag are independent of
$N$, the bound in \cite[Lemma~4.5]{FGK} would give
$$
2^{N/2}
\leq
\#\mathcal L_{\cV,\mathbf c,\boldsymbol\mu}(B)
\leq
\exp\{CN^{3/4}\}
$$
for some constant $C>0$ independent of $N$.
Taking logarithms and dividing by $N$ yields
$$
\frac{\log2}{2}\leq CN^{-1/4},
$$
which is impossible as $N\to\infty$.

Let $D>1$, let $(\cV,\mathbf c,\boldsymbol\mu)$ be a system, put
$c=c_{r+1}$, and let $B\subseteq(D^c,D]\cap\Z$.
For $1\leq j\leq r$, put
\[
B^{(j)}=B\cap(D^{c_{j+1}},D^{c_j}],
\qquad N_j=\#B^{(j)}.
\]
Let $\cL_{\cV,\mathbf c,\boldsymbol\mu}(B)$ be the set of classes
\[
\left(\sum_{\omega\in\{0,1\}^k}\omega\,\Sigma(B_\omega)\right)
+\langle\one\rangle
\in\Q^k/\langle\one\rangle,
\]
where $(B_\omega)_{\omega\in\{0,1\}^k}$ ranges over partitions of
$B$ satisfying
\begin{equation}\label{eq:partition-compatibility}
\begin{aligned}
\#\bigl(B_\omega\cap(D^{c_{j+1}},D^{c_j}]\bigr)
&=\mu_j(\omega)N_j
&& (1\leq j\leq r),\\
B_\omega\cap(D^{c_1},D]&=\varnothing
&& (\omega\notin V_0).
\end{aligned}
\end{equation}
The first equality is required for every $\omega\in\{0,1\}^k$
and is interpreted as $0=0$ when $N_j=0$.
The second condition implies that the contribution from
$B\cap(D^{c_1},D]$ belongs to $\langle\one\rangle$.

\begin{lemma}\label{lem:residual-sum-bound}
Let $k\geq2$, let $D$ be sufficiently large in terms of $k$, and let
$(\cV,\mathbf c,\boldsymbol\mu)$ be an entropy system in
$\Q^k$, where $\mathbf c=(c_1,\ldots,c_{r+1})$.
Put $c=c_{r+1}$.
Suppose that $B\subseteq(D^c,D]\cap\Z$ satisfies
\[
\left|
\#\bigl(B\cap(D^\alpha,D^\beta]\bigr)
-(\beta-\alpha)\log D
\right|
\leq2(\log D)^{3/4}
\qquad(c\leq\alpha\leq\beta\leq1).
\]
Then, for all sufficiently large $D$,
\[
\#\cL_{\cV,\mathbf c,\boldsymbol\mu}(B)
\leq
\exp\bigl(C_k(\log D)^{3/4}\bigr)
D^{\min_{\cV'\leq\cV}
\mathrm e(\cV',\mathbf c,\boldsymbol\mu)},
\]
where $C_k>0$ and the lower bound on $D$ depend only on $k$.
\end{lemma}

\begin{proof}
If $\cL_{\cV,\mathbf c,\boldsymbol\mu}(B)=\varnothing$,
the conclusion is immediate. Assume that at least one
partition satisfies \eqref{eq:partition-compatibility}.
Fix a subflag $\cV'\leq\cV$, and put
$d_j=\dim V'_j$ for $0\leq j\leq r$.
Then
\[
1=d_0\leq d_1\leq\cdots\leq d_r\leq k.
\]
By successively extending bases, choose a basis
$v_1,\ldots,v_{d_r}$ of $V'_r$ over $\Q$ such that
\[
v_1=\one,
\qquad
V'_j=\Span_{\Q}\{v_1,\ldots,v_{d_j}\}
\qquad(0\leq j\leq r).
\]
For each $1\leq j\leq r$, extend
$v_1,\ldots,v_{d_j}$ to a basis
$v_1,\ldots,v_{d_j},
u_{j,1},\ldots,u_{j,\dim V_j-d_j}$
of $V_j$ over $\Q$.

For each $1\leq j\leq r$, let $P_j$ be the linear map
from $V_j$ to $V'_j$ satisfying
\[
P_j(v_\ell)=v_\ell\quad(1\leq\ell\leq d_j),
\qquad
P_j(u_{j,\ell})=\mathbf0_k
\quad(1\leq\ell\leq\dim V_j-d_j).
\]

Define $Q_j(x)=x-P_j(x)$ for $x\in V_j$.
Let $x,y\in V_j$ and suppose that $Q_j(x)=Q_j(y)$.
Using the linearity of $P_j$, we obtain
\begin{equation}\label{eq1}
x-y=P_j(x)-P_j(y)=P_j(x-y)\in V'_j,
\end{equation}
where the last inclusion follows from $x-y\in V_j$
and $P_j(V_j)\subseteq V'_j$.
Conversely, if $x-y\in V'_j$, then
\begin{equation}\label{eq:residual-Q-coset}
Q_j(x)-Q_j(y)
=
x-y-P_j(x-y)
=
\mathbf0_k.
\end{equation}
Hence $Q_j$ is constant on each coset of $V'_j$ in $V_j$,
and distinct cosets have distinct $Q_j$-values.

Since $\mu_j$ is supported on $V_j\cap\{0,1\}^k$ and is
extended by zero outside this set, we have $\mu_j(\omega)=0$
for $\omega\in\{0,1\}^k\setminus V_j$.
Since $B_\omega\subseteq B$, the definition of $B^{(j)}$
and \eqref{eq:partition-compatibility} give
\[
\begin{aligned}
\#(B_\omega\cap B^{(j)})
&=
\#\bigl(B_\omega\cap(D^{c_{j+1}},D^{c_j}]\bigr)
=
N_j\mu_j(\omega)
=
0
\qquad
(\omega\in\{0,1\}^k\setminus V_j).
\end{aligned}
\]
Hence $B_\omega\cap B^{(j)}=\varnothing$ for
$\omega\in\{0,1\}^k\setminus V_j$.

Since $V_0=\langle\one\rangle$, we have
$V_0\cap\{0,1\}^k=\{\mathbf0_k,\one\}$.
Consequently, by the second condition in
\eqref{eq:partition-compatibility}, we have
\[
\begin{aligned}
\sum_{\omega\in\{0,1\}^k}
\omega\,\Sigma\bigl(B_\omega\cap(D^{c_1},D]\bigr)
=\sum_{\omega\in\{\mathbf0_k,\one\}}
\omega\,\Sigma\bigl(B_\omega\cap(D^{c_1},D]\bigr)
=\Sigma\bigl(B_{\one}\cap(D^{c_1},D]\bigr)\one
\in\langle\one\rangle.
\end{aligned}
\]

For this partition, define
\begin{equation}\label{eq:residual-sum-components}
\begin{aligned}
S_1
&=
\sum_{j=1}^{r}
\sum_{\omega\in V_j\cap\{0,1\}^k}
P_j(\omega)\,
\Sigma(B_\omega\cap B^{(j)}),\\
S_2
&=
\sum_{j=1}^{r}
\sum_{\omega\in V_j\cap\{0,1\}^k}
Q_j(\omega)\,
\Sigma(B_\omega\cap B^{(j)}).
\end{aligned}
\end{equation}
Since $P_j(\omega)+Q_j(\omega)=\omega$, we obtain
\[
\left(
\sum_{\omega\in\{0,1\}^k}
\omega\,\Sigma(B_\omega)
\right)
+\langle\one\rangle
=
S_1+S_2+\langle\one\rangle.
\]

Let $\mathscr X$ be the set of classes
$S_1+\langle\one\rangle$, and let $\mathscr Y$ be the set
of vectors $S_2$, as the partition ranges over those
satisfying \eqref{eq:partition-compatibility}.
Every element of $\cL_{\cV,\mathbf c,\boldsymbol\mu}(B)$
is obtained by adding a vector in $\mathscr Y$ to a class
in $\mathscr X$. Therefore,
\[
\#\cL_{\cV,\mathbf c,\boldsymbol\mu}(B)
\leq
\#\mathscr X\,\#\mathscr Y.
\]

\begin{claim}\label{clm:residual-coset-count}
With $C_k^{(1)}=2k(k-1)\log2$, we have
\[
\#\mathscr Y
\leq
\exp\bigl(C_k^{(1)}(\log D)^{3/4}\bigr)
D^{\sum_{j=1}^{r}
(c_j-c_{j+1})\Hh_{\mu_j}(V'_j)}.
\]
\end{claim}

\begin{proof}[Proof of the claim]
Fix $1\leq j\leq r$.
Let $C_1,\ldots,C_m$ be the distinct cosets of $V'_j$
in $V_j$ having positive $\mu_j$-measure, and put
\[
\theta_i=\mu_j(C_i)
\qquad(1\leq i\leq m).
\]
Since $\mu_j$ is a probability measure supported on
$V_j\cap\{0,1\}^k$, it follows that
\[
\theta_i>0,
\qquad
\sum_{i=1}^{m}\theta_i=1,
\qquad
1\leq m\leq2^k.
\]
By \eqref{eq:coset-entropy} and the concavity of the
logarithm, we have
\begin{equation}\label{eq:residual-coset-entropy-bound}
\begin{aligned}
0\leq\Hh_{\mu_j}(V'_j)
&=
-\sum_{i=1}^{m}
\sum_{\omega\in C_i\cap\{0,1\}^k}
\mu_j(\omega)\log\theta_i\\
&=
-\sum_{i=1}^{m}\theta_i\log\theta_i
=
\sum_{i=1}^{m}\theta_i\log\frac1{\theta_i}\\
&\leq
\log\left(
\sum_{i=1}^{m}\theta_i\frac1{\theta_i}
\right)
=
\log m
\leq k\log2.
\end{aligned}
\end{equation}

For a partition $(B_\omega)_{\omega\in\{0,1\}^k}$
satisfying \eqref{eq:partition-compatibility}, put
\[
T_i
=
\bigcup_{\omega\in C_i\cap\{0,1\}^k}
(B_\omega\cap B^{(j)})
\qquad(1\leq i\leq m).
\]
The sets $T_1,\ldots,T_m$ are pairwise disjoint subsets
of $B^{(j)}$. By \eqref{eq:partition-compatibility},
\begin{equation}\label{eq:residual-coset-cardinality}
\begin{aligned}
\#T_i
&=
\sum_{\omega\in C_i\cap\{0,1\}^k}
\#(B_\omega\cap B^{(j)})=
N_j\sum_{\omega\in C_i\cap\{0,1\}^k}\mu_j(\omega)
=
N_j\theta_i.
\end{aligned}
\end{equation}
Thus $N_j\theta_i\in\Z_{\geq0}$ for every $i$, and
\[
\sum_{i=1}^{m}\#T_i
=
N_j\sum_{i=1}^{m}\theta_i
=
N_j
=
\#B^{(j)}.
\]
Consequently, $B^{(j)}=\bigcup_{i=1}^{m}T_i.$

For every $1\leq i\leq m$, fix
$x_i\in C_i\cap\{0,1\}^k$.
For $\omega\in C_i\cap\{0,1\}^k$, we have
$\omega-x_i\in V'_j$, so \eqref{eq:residual-Q-coset}
gives $Q_j(\omega)=Q_j(x_i).$
If $\omega\in V_j\cap\{0,1\}^k$ belongs to none of
$C_1,\ldots,C_m$, then $\mu_j(\omega)=0$.
Hence \eqref{eq:partition-compatibility} gives
$B_\omega\cap B^{(j)}=\varnothing$.
Since the sets in the definition of each $T_i$ are
pairwise disjoint, it follows that
\[
\begin{aligned}
&\sum_{\omega\in V_j\cap\{0,1\}^k}
Q_j(\omega)\Sigma(B_\omega\cap B^{(j)})
\quad=\sum_{i=1}^{m}Q_j(x_i)
\sum_{\omega\in C_i\cap\{0,1\}^k}
\Sigma(B_\omega\cap B^{(j)})
=\sum_{i=1}^{m}Q_j(x_i)\Sigma(T_i).
\end{aligned}
\]
The vectors $Q_j(x_1),\ldots,Q_j(x_m)$ are fixed.
Therefore, the $j$-th sum in
\eqref{eq:residual-sum-components} defining $S_2$ is
determined by the ordered partition $(T_1,\ldots,T_m)$.

If $N_j=0$, then $B^{(j)}=\varnothing$, and this sum
has the single possible value $\mathbf0_k$.
Suppose that $N_j>0$.
By \eqref{eq:residual-coset-cardinality},
\[
N_j\theta_i\in\Z_{\geq0}
\qquad(1\leq i\leq m),
\qquad
\sum_{i=1}^{m}N_j\theta_i=N_j.
\]
After $T_1,\ldots,T_{i-1}$ have been chosen, there remain
$N_j-\sum_{\ell=1}^{i-1}N_j\theta_\ell$ elements from
which to choose the $N_j\theta_i$ elements of $T_i$.
Thus the number of ordered partitions
$(T_1,\ldots,T_m)$ with the prescribed cardinalities is
\[
\begin{aligned}
\prod_{i=1}^{m}
\binom{N_j-\sum_{\ell=1}^{i-1}N_j\theta_\ell}
{N_j\theta_i}
&=
\prod_{i=1}^{m}
\frac{
\left(N_j-\sum_{\ell=1}^{i-1}N_j\theta_\ell\right)!
}{
(N_j\theta_i)!
\left(N_j-\sum_{\ell=1}^{i}N_j\theta_\ell\right)!
}\\
&=
\frac{N_j!}{\prod_{i=1}^{m}(N_j\theta_i)!},
\end{aligned}
\]
where an empty sum is zero and $0!=1$.
Each such ordered partition determines a single vector
$\sum_{i=1}^{m}Q_j(x_i)\Sigma(T_i).$
Hence the number of possible values of the $j$-th sum
is at most
$\frac{N_j!}{\prod_{i=1}^{m}(N_j\theta_i)!}.$

Note that
\[
1
=
\left(\sum_{i=1}^{m}\theta_i\right)^{N_j}
\geq
\frac{N_j!}{\prod_{i=1}^{m}(N_j\theta_i)!}
\prod_{i=1}^{m}\theta_i^{N_j\theta_i}.
\]
Consequently, by
\eqref{eq:residual-coset-entropy-bound},
\[
\begin{aligned}
\frac{N_j!}{\prod_{i=1}^{m}(N_j\theta_i)!}
&\leq
\prod_{i=1}^{m}\theta_i^{-N_j\theta_i}
=\exp\left(-
N_j\sum_{i=1}^{m}\theta_i\log\theta_i\right)
=\exp\bigl(N_j\Hh_{\mu_j}(V'_j)\bigr).
\end{aligned}
\]
The same upper bound holds when $N_j=0$, since it
equals $1$.

By \eqref{eq:residual-sum-components}, each vector in
$\mathscr Y$ is a sum of one possible value from each
of the $r$ sums. Therefore,
\begin{equation}\label{eq:residual-Y-product}
\begin{aligned}
\#\mathscr Y
&\leq
\prod_{j=1}^{r}
\exp\bigl(N_j\Hh_{\mu_j}(V'_j)\bigr)
=\exp\left(\sum_{j=1}^{r}N_j\Hh_{\mu_j}(V'_j)
\right).
\end{aligned}
\end{equation}

Since
\[
c=c_{r+1}\leq c_{j+1}\leq c_j\leq1,
\qquad
N_j=\#\bigl(B\cap(D^{c_{j+1}},D^{c_j}]\bigr),
\]
the hypothesis with $\alpha=c_{j+1}$ and $\beta=c_j$ gives
$$
\left|
N_j-(c_j-c_{j+1})\log D
\right|
\leq2(\log D)^{3/4}.
$$
Consequently,
\begin{equation}\label{eq:residual-layer-count}
N_j
\leq
(c_j-c_{j+1})\log D+2(\log D)^{3/4}
\qquad(1\leq j\leq r).
\end{equation}

Combining \eqref{eq:residual-layer-count} with
\eqref{eq:residual-coset-entropy-bound} and
$r\leq k-1$, we obtain
\[
\begin{aligned}
\sum_{j=1}^{r}N_j\Hh_{\mu_j}(V'_j)
&\leq
(\log D)
\sum_{j=1}^{r}
(c_j-c_{j+1})\Hh_{\mu_j}(V'_j)+
2(\log D)^{3/4}
\sum_{j=1}^{r}\Hh_{\mu_j}(V'_j)\\
&\leq
(\log D)
\sum_{j=1}^{r}
(c_j-c_{j+1})\Hh_{\mu_j}(V'_j)+
C_k^{(1)}(\log D)^{3/4}.
\end{aligned}
\]
Substituting this estimate into
\eqref{eq:residual-Y-product} yields
\[
\#\mathscr Y
\leq
\exp\bigl(C_k^{(1)}(\log D)^{3/4}\bigr)
D^{\sum_{j=1}^{r}
(c_j-c_{j+1})\Hh_{\mu_j}(V'_j)}.
\]
This proves the claim.
\end{proof}

\begin{claim}\label{clm:residual-subspace-count}
There exists a constant $A_{\cV,\cV'}\geq1$
such that
\[
\#\mathscr X
\leq
A_{\cV,\cV'}(\log D)^{k-1}
D^{\sum_{j=1}^{r}c_j\dim(V'_j/V'_{j-1})}.
\]
\end{claim}

\begin{proof}
If $d_r=1$, then $V'_j=\langle\one\rangle$ for every
$0\leq j\leq r$. Hence $S_1\in\langle\one\rangle$,
so $\#\mathscr X=1$, while
$\sum_{j=1}^{r}c_j\dim(V'_j/V'_{j-1})=0.$
The claim follows with $A_{\cV,\cV'}=1$ for $\log D\geq1$.
Assume that $d_r\geq2$.

For the fixed bases and projections,
\[
\#\left(
\bigcup_{j=1}^{r}P_j(V_j\cap\{0,1\}^k)
\right)
\leq
\sum_{j=1}^{r}\#(V_j\cap\{0,1\}^k)
\leq r2^k.
\]
Every vector in this finite set has rational coordinates
in the basis $v_1,\ldots,v_{d_r}$.
Consequently, there exists an integer $M\geq1$ such that
\begin{equation}\label{eq:residual-coefficient-bounds}
P_j(V_j\cap\{0,1\}^k)
\subseteq
\left\{
\sum_{\ell=1}^{d_r}a_\ell v_\ell
\,\middle|\,
a_\ell\in\frac1M\Z\cap[-M,M]
\quad(1\leq\ell\leq d_r)
\right\}
\end{equation}
for every $1\leq j\leq r$.

For every $1\leq j\leq r$ and
$\omega\in V_j\cap\{0,1\}^k$, we have $P_j(\omega)\in V'_j\subseteq V'_r$ and $\Sigma(B_\omega\cap B^{(j)})\in\Z.$
Hence \eqref{eq:residual-sum-components} gives
$S_1\in V'_r$ for every partition satisfying
\eqref{eq:partition-compatibility}.
Since $v_1,\ldots,v_{d_r}$ form a basis of $V'_r$
over $\Q$, there exist unique rational numbers
$\xi_1,\ldots,\xi_{d_r}$ such that
\begin{equation}\label{eq:residual-S1-coordinates}
S_1=\sum_{\ell=1}^{d_r}\xi_\ell v_\ell.
\end{equation}

By \eqref{eq:residual-sum-components}, $S_1$ is a sum
of the vectors $P_i(\omega)$ multiplied by the integers
$\Sigma(B_\omega\cap B^{(i)})$.
Consequently, \eqref{eq:residual-coefficient-bounds} gives
$M\xi_\ell\in\Z$ for $1\leq\ell\leq d_r.$ For $1\leq j\leq r$ and $d_{j-1}<\ell\leq d_j$, the relations
\[
P_i(\omega)\in V'_i\subseteq V'_{j-1}
=
\Span_{\Q}\{v_1,\ldots,v_{d_{j-1}}\}
\qquad(i<j)
\]
imply that the coefficient of $v_\ell$ in $P_i(\omega)$
is zero whenever $i<j$.
For $i\geq j$ and $a\in B^{(i)}$, we have
\[
0<a\leq D^{c_i}\leq D^{c_j}.
\]

Since the sets $B^{(i)}$ are pairwise disjoint and
$c_{r+1}=c$, it follows that
$\bigcup_{i=j}^{r}B^{(i)}
=B\cap(D^c,D^{c_j}].$
Applying the hypothesis with $\alpha=c$ and $\beta=1$,
we obtain
\begin{equation}\label{eq:residual-tail-count}
\begin{aligned}
\sum_{i=j}^{r}N_i
&=
\#\bigl(B\cap(D^c,D^{c_j}]\bigr)
\leq
(1-c)\log D+2(\log D)^{3/4}
\leq2\log D
\end{aligned}
\end{equation}
for all sufficiently large $D$.

For $d_{j-1}<\ell\leq d_j$, the coefficient of
$v_\ell$ in $P_i(\omega)$ is zero when $i<j$ and has
absolute value at most $M$ when $i\geq j$.
Moreover,
$\Sigma(B_\omega\cap B^{(i)})\geq0$.
Taking the $v_\ell$-coordinate in
\eqref{eq:residual-sum-components} and applying the
triangle inequality, we obtain
\[
\begin{aligned}
|\xi_\ell|
&\leq
M\sum_{i=j}^{r}
\sum_{\omega\in V_i\cap\{0,1\}^k}
\Sigma(B_\omega\cap B^{(i)})
=
M\sum_{i=j}^{r}\Sigma(B^{(i)})\\
&\leq
M D^{c_j}\sum_{i=j}^{r}N_i\\
&\leq
2M D^{c_j}\log D,
\end{aligned}
\]
where the last inequality follows from
\eqref{eq:residual-tail-count}.
Together with $M\xi_\ell\in\Z$, this gives
\begin{equation}\label{eq:residual-coordinate-range}
\xi_\ell
\in
\frac1M\Z
\cap
[-2M D^{c_j}\log D,\,
  2M D^{c_j}\log D]
\qquad(d_{j-1}<\ell\leq d_j).
\end{equation}

The interval in \eqref{eq:residual-coordinate-range}
contains at most
\[
\begin{aligned}
2\left\lfloor2M^2D^{c_j}\log D\right\rfloor+1
\leq
4M^2D^{c_j}\log D+1
\leq
(4M^2+1)D^{c_j}\log D
\end{aligned}
\]
elements of $\frac1M\Z$, since $c_j\geq0$ and
$\log D\geq1$.
Thus each of the $d_j-d_{j-1}$ coordinates
$\xi_{d_{j-1}+1},\ldots,\xi_{d_j}$
has at most $(4M^2+1)D^{c_j}\log D$ possible values.

Since $v_1=\one$, \eqref{eq:residual-S1-coordinates}
gives
\[
S_1+\langle\one\rangle
=
\left(\sum_{\ell=2}^{d_r}\xi_\ell v_\ell\right)
+\langle\one\rangle.
\]
Hence this class is determined by
$(\xi_2,\ldots,\xi_{d_r})$.
Taking the product of the coordinate bounds yields
\begin{equation}\label{eq:residual-X-coordinate-count}
\begin{aligned}
\#\mathscr X
&\leq
\prod_{j=1}^{r}
\left((4M^2+1)D^{c_j}\log D\right)^{d_j-d_{j-1}}\\
&=(4M^2+1)^{d_r-1}
(\log D)^{d_r-1}
D^{\sum_{j=1}^{r}c_j(d_j-d_{j-1})}\\
&\leq
(4M^2+1)^{k-1}
(\log D)^{k-1}
D^{\sum_{j=1}^{r}c_j\dim(V'_j/V'_{j-1})}.
\end{aligned}
\end{equation}
Here we used
\[
\sum_{j=1}^{r}(d_j-d_{j-1})=d_r-1\leq k-1,
\qquad
d_j-d_{j-1}=\dim(V'_j/V'_{j-1}).
\]
Taking $A_{\cV,\cV'}=(4M^2+1)^{k-1}$ proves the claim.
\end{proof}

Combining Claims~\ref{clm:residual-coset-count} and
\ref{clm:residual-subspace-count}, and using
\eqref{eq:e-value}, we obtain
\begin{equation}\label{eq:residual-fixed-subflag-bound}
\begin{aligned}
\#\cL_{\cV,\mathbf c,\boldsymbol\mu}(B) \leq
\#\mathscr X\,\#\mathscr Y
\leq
A_{\cV,\cV'}(\log D)^{k-1}
\exp\bigl(C_k^{(1)}(\log D)^{3/4}\bigr)
D^{\mathrm e(\cV',\mathbf c,\boldsymbol\mu)}.
\end{aligned}
\end{equation}

By Lemma~\ref{lem:finite-subflags},
\[
\min_{\cV'\leq\cV}
\mathrm e(\cV',\mathbf c,\boldsymbol\mu)
=
\min_{\cV'\in\mathscr F(\cV)}
\mathrm e(\cV',\mathbf c,\boldsymbol\mu).
\]
Every space in $\cV$ is spanned by a subset of
$\{0,1\}^k$, and $r\leq k-1$.
Thus only finitely many flags $\cV$ occur for fixed $k$.
For each such flag, fix the finite family
$\mathscr F(\cV)$ supplied by
Lemma~\ref{lem:finite-subflags}.
Then fix the bases and projections for every
$\cV'\in\mathscr F(\cV)$.

There are only finitely many pairs
$(\cV,\cV')$ under consideration.
Consequently, there exists $A_k\geq1$, depending only
on $k$, such that $A_{\cV,\cV'}\leq A_k \; (\cV'\in\mathscr F(\cV))$
for every such flag $\cV$.
Choosing a subflag in $\mathscr F(\cV)$ that minimizes
$\mathrm e(\cV',\mathbf c,\boldsymbol\mu)$ and applying
\eqref{eq:residual-fixed-subflag-bound} gives
\[
\begin{aligned}
\#\cL_{\cV,\mathbf c,\boldsymbol\mu}(B)
&\leq
A_k(\log D)^{k-1}
\exp\bigl(C_k^{(1)}(\log D)^{3/4}\bigr)
D^{\min_{\cV'\leq\cV}
\mathrm e(\cV',\mathbf c,\boldsymbol\mu)}.
\end{aligned}
\]

Finally, since $k$ is fixed,
\[
\log A_k+(k-1)\log\log D
=
o\bigl((\log D)^{3/4}\bigr).
\]
Hence, for all sufficiently large $D$,
\[
A_k(\log D)^{k-1}
\leq
\exp\bigl((\log D)^{3/4}\bigr).
\]
Taking $C_k=C_k^{(1)}+1$, we conclude that
\[
\#\cL_{\cV,\mathbf c,\boldsymbol\mu}(B)
\leq
\exp\bigl(C_k(\log D)^{3/4}\bigr)
D^{\min_{\cV'\leq\cV}
\mathrm e(\cV',\mathbf c,\boldsymbol\mu)}.
\]
The constant $C_k$ and the lower bound on $D$ depend
only on $k$. This proves the lemma.
\end{proof}

\subsection{Comparison of the thresholds}

The following proposition relates the probabilistic threshold
$\beta_k$ to the entropy thresholds $\widetilde\gamma_k$
and $\gamma_k$.

\begin{proposition}\label{prop:threshold-comparison}
For every integer $k\geq2$, we have
$$
\widetilde\gamma_k\leq\beta_k\leq\gamma_k.
$$
\end{proposition}

\begin{proof}
Fix an integer $k\geq2$.
\setcounter{claim}{0}

\begin{claim}\label{clm:comparison-upper}
$\beta_k\leq\gamma_k$.
\end{claim}

\begin{proof}
The entropy condition holds when all thresholds are zero, so
$\gamma_k\geq0$. If $\gamma_k=1$, the claim follows from
$\beta_k\leq1$. Assume that $\gamma_k<1$.
By the definition of $\beta_k$, it suffices to show that, for
every $c\in(\gamma_k,1)$,
\begin{equation}\label{eq:comparison-upper-vanishing}
\Pp\bigl(m(\mathcal A\cap(D^c,D])\geq k\bigr)=o(1)
\qquad(D\to\infty).
\end{equation}

For a complete entropy system
$(\cV,\mathbf c,\boldsymbol\mu)$ with final threshold $c$,
define
$$
G(\cV,\mathbf c,\boldsymbol\mu)
=
\min_{\cV'\in\mathscr F(\cV)}
\left(
\mathrm e(\cV',\mathbf c,\boldsymbol\mu)
-\mathrm e(\cV,\mathbf c,\boldsymbol\mu)
\right),
$$
where $\mathscr F(\cV)$ is given by
Lemma~\ref{lem:finite-subflags}.
By Lemma~\ref{lem:finite-subflags}, this minimum equals the
minimum over all subflags $\cV'\leq\cV$.
Since $c>\gamma_k$, the definition of $\gamma_k$ implies that
the system does not satisfy the entropy condition.
Hence there exists a subflag $\cV'\leq\cV$ such that
$$
\mathrm e(\cV',\mathbf c,\boldsymbol\mu)
<
\mathrm e(\cV,\mathbf c,\boldsymbol\mu).
$$
Consequently,
$G(\cV,\mathbf c,\boldsymbol\mu)<0$.

For each fixed complete flag $\cV$ with non-degenerate terminal
space, the admissible parameters with final threshold $c$
form a compact set. By the continuity of coset entropy and
the finiteness of $\mathscr F(\cV)$, $G$ is continuous and
therefore attains a strictly negative maximum on this set.
Since only finitely many such cube-spanned flags occur for
fixed $k$, there exists $\eta=\eta(k,c)>0$ such that
\begin{equation}\label{eq:comparison-negative-gap}
\min_{\cV'\leq\cV}
\mathrm e(\cV',\mathbf c,\boldsymbol\mu)
\leq
\mathrm e(\cV,\mathbf c,\boldsymbol\mu)-\eta
\end{equation}
for every complete entropy system with final threshold $c$.

For sufficiently large $D$, put
$$
I_{c,D}=(D^c,D]\cap\Z,
\qquad
\mathcal A_D=\mathcal A\cap I_{c,D}.
$$
Let $\mathcal R_{c,D}$ be the family of sets
$B\subseteq I_{c,D}$ satisfying
$$
\left|
\#\bigl(B\cap(D^\alpha,D^\beta]\bigr)
-(\beta-\alpha)\log D
\right|
\leq(\log D)^{3/4}
\qquad(c\leq\alpha\leq\beta\leq1),
$$
and let $\widetilde{\mathcal R}_{c,D}$ be defined by the
same inequalities with $2(\log D)^{3/4}$ on the right.
By \cite[Lemma~A.5]{FGK},
\begin{equation}\label{eq:comparison-regularity}
\Pp(\mathcal A_D\notin\mathcal R_{c,D})
\ll_c
\exp\left\{-\frac14\sqrt{\log D}\right\}.
\end{equation}
Taking $\alpha=c$ and $\beta=1$ also gives
$$
\#B\leq(1-c)\log D+(\log D)^{3/4}\leq2\log D
\qquad(B\in\mathcal R_{c,D}).
$$

Suppose that $B\in\mathcal R_{c,D}$ contains pairwise distinct
subsets $A_1,\ldots,A_k$ with equal sums.  For each
$\omega=(\omega_1,\ldots,\omega_k)\in\{0,1\}^k$, define
$$
B_\omega=
\left\{
b\in B:
b\in A_t \text{ if and only if } \omega_t=1
\text{ for every }1\leq t\leq k
\right\}.
$$
Every element of $B$
belongs to exactly one $B_\omega$, so these sets form a
partition of $B$, with empty parts allowed.

We apply the extraction of \cite[Section~4.1]{FGK}.
Starting with $V_0=\langle\one\rangle$, at step $j$ choose
$\omega^j\notin V_{j-1}$ with $B_{\omega^j}\neq\varnothing$
so that $\max B_{\omega^j}$ is largest, and put
$$
K_j=\max B_{\omega^j},
\qquad
V_j=V_{j-1}+\Span_{\Q}\{\omega^j\}.
$$
Stop when every vector $\omega$ with $B_\omega\neq\varnothing$
belongs to the current space. Let $h$ be the number of steps
performed. The resulting flag is
$$
\cV:V_0<V_1<\cdots<V_h,
\qquad 1\leq h\leq k-1.
$$

For $1\leq j<h$, we have $\omega^{j+1}\notin V_{j-1}$,
so the maximal choice at step $j$ gives $K_{j+1}\leq K_j$.
Since $\omega^j\in V_j$ and $\omega^{j+1}\notin V_j$,
these two vectors are distinct. The corresponding parts
$B_{\omega^j}$ and $B_{\omega^{j+1}}$ are therefore disjoint,
which implies $K_{j+1}\neq K_j$. Thus
$$
K_1>\cdots>K_h.
$$

For each $1\leq t<u\leq k$, the inequality $A_t\neq A_u$
implies that there exists $\omega\in\{0,1\}^k$ such that $B_\omega\neq\varnothing$ and $\omega_t\neq\omega_u.$
The stopping condition gives $\omega\in V_h$. Therefore, $V_h$ is non-degenerate.

For $1\leq j\leq h$, define
$$
c_j=1+\frac{\lceil\log K_j-\log D\rceil}{\log D},
\qquad
c_{h+1}=c.
$$
Then
\begin{equation}\label{eq:comparison-rounded-thresholds}
1\geq c_1\geq\cdots\geq c_h>c,
\qquad
\ee^{-1}D^{c_j}<K_j\leq D^{c_j}
\quad(1\leq j\leq h).
\end{equation}
Put
$$
B'=B\setminus\{K_1,\ldots,K_h\},
\qquad
B'_\omega=B_\omega\cap B',
\qquad
N_j=\#\bigl(B'\cap(D^{c_{j+1}},D^{c_j}]\bigr),
$$
and define
$$
\mu_j(\omega)=
\begin{cases}
\displaystyle
\frac{\#\bigl(B'_\omega\cap(D^{c_{j+1}},D^{c_j}]\bigr)}{N_j},
&N_j>0,\\[2mm]
\one_{\omega=\mathbf0},&N_j=0
\end{cases}
\qquad(1\leq j\leq h).
$$

For $0\leq j<h$, the maximal choice at step $j+1$ gives
$$
\max B_\omega\leq K_{j+1}\leq D^{c_{j+1}}
\qquad
(\omega\notin V_j,\ B_\omega\neq\varnothing).
$$
For $j=h$, the stopping condition gives
$B_\omega=\varnothing$ whenever $\omega\notin V_h$.
Thus the definition of $\mu_j$ implies
$$
\mu_j(\omega)=0
\qquad
(1\leq j\leq h,\ \omega\in\{0,1\}^k\setminus V_j),
$$
including when $N_j=0$, since $\mathbf0\in V_j$.
Hence each $\mu_j$ is a probability measure on
$V_j\cap\{0,1\}^k$, and
$(\cV,\mathbf c,\boldsymbol\mu)$ is a complete entropy
system with final threshold $c$.
The estimate with $j=0$ also gives
$$
B'_\omega\cap(D^{c_1},D]=\varnothing
\qquad(\omega\in\{0,1\}^k\setminus V_0).
$$
Together with the frequency equalities defining $\mu_j$,
this proves that $(B'_\omega)_\omega$ satisfies
\eqref{eq:partition-compatibility}.

Since $B\in\mathcal R_{c,D}$ and
$\#(B\setminus B')=h\leq k-1$, it follows that
$$
\left|
\#\bigl(B'\cap(D^\alpha,D^\beta]\bigr)
-(\beta-\alpha)\log D
\right|
\leq(\log D)^{3/4}+h
\leq2(\log D)^{3/4}
$$
for $c\leq\alpha\leq\beta\leq1$ and sufficiently large $D$.
Therefore $B'\in\widetilde{\mathcal R}_{c,D}$.

Write
$$
\Lambda=(\cV,\mathbf c,\boldsymbol\mu,\omega^1,\ldots,\omega^h),
$$
where $\cV:V_0<V_1<\cdots<V_h$ is the complete flag constructed above,
$\mathbf c=(c_1,\ldots,c_{h+1})$ is the threshold sequence in
\eqref{eq:comparison-rounded-thresholds}, and
$\boldsymbol\mu=(\mu_1,\ldots,\mu_h)$ is the family of probability
measures defined above, with each $\mu_j$ supported on
$V_j\cap\{0,1\}^k$.
The vectors $\omega^1,\ldots,\omega^h$ are those selected in the
construction, so that
$$
\omega^j\in(V_j\setminus V_{j-1})\cap\{0,1\}^k,
\qquad
V_j=V_{j-1}+\Span_{\Q}\{\omega^j\}
\qquad(1\leq j\leq h).
$$
Thus fixing $\Lambda$ fixes the flag, the threshold sequence,
the probability measures, and the selected vectors.
For every $B'\subseteq I_{c,D}$, write
$$
\cL_\Lambda(B')=\cL_{\cV,\mathbf c,\boldsymbol\mu}(B').
$$

Fix one such $\Lambda$, and let $\mathcal B_\Lambda$ be the family
of sets $B\in\mathcal R_{c,D}$ for which the preceding
construction produces $\Lambda$ from some $k$ distinct
subsets of equal sum. For each $B\in\mathcal B_\Lambda$,
choose one such tuple $(A_1,\ldots,A_k)$, with associated
$B'$ and $(B'_\omega)_\omega$, and define
$$
L=\sum_{\omega\in\{0,1\}^k}
\omega\,\Sigma(B'_\omega)+V_0
\in\cL_\Lambda(B').
$$

For each $1\leq t\leq k$, the elements of $A_t$ consist
of the selected elements $K_j$ with $\omega_t^j=1$ and
the elements of $B'_\omega$ with $\omega_t=1$.
These sets are pairwise disjoint, so
$$
\Sigma(A_t)
=
\sum_{j=1}^hK_j\omega_t^j
+
\sum_{\omega\in\{0,1\}^k}
\omega_t\,\Sigma(B'_\omega).
$$
Since
$\Sigma(A_1)=\cdots=\Sigma(A_k)$, it follows that we obtain
$$
\begin{aligned}
\sum_{j=1}^hK_j\omega^j
+\sum_{\omega\in\{0,1\}^k}\omega\,\Sigma(B'_\omega)
&=\bigl(\Sigma(A_1),\ldots,\Sigma(A_k)\bigr)
=\Sigma(A_1)\one\in V_0.
\end{aligned}
$$

Therefore, we have
\begin{equation}\label{eq:comparison-pivot-equation}
\sum_{j=1}^hK_j\omega^j+V_0=-L.
\end{equation}

For fixed $\Lambda$ and $L$, the linear independence of
$\one,\omega^1,\ldots,\omega^h$ implies that
\eqref{eq:comparison-pivot-equation} determines at most
one tuple $(K_1,\ldots,K_h)$. Since
$B=B'\cup\{K_1,\ldots,K_h\}$, each pair $(B',L)$
determines at most one set $B\in\mathcal B_\Lambda$.
Consequently, for each fixed $B'$, there are at most
$\#\cL_\Lambda(B')$ sets $B\in\mathcal B_\Lambda$
with residual set $B'$, as in the proof of
\cite[Proposition~4.4]{FGK}.

For every $T\subseteq I_{c,D}$, by independence and $\Pp(i\in\mathcal A)=1/i$,
$$
\begin{aligned}
\Pp(\mathcal A_D=T)
=\prod_{i\in T}\frac1i
  \prod_{i\in I_{c,D}\setminus T}\left(1-\frac1i\right)
=\prod_{i\in I_{c,D}}\left(1-\frac1i\right)
  \prod_{i\in T}\frac1{i-1},
\end{aligned}
$$
where the second equality uses
$1/i=(1-1/i)/(i-1)$.
Since $B'=B\setminus\{K_1,\ldots,K_h\}$, we obtain
$$
\begin{aligned}
\Pp(\mathcal A_D=B)
&=\Pp(\mathcal A_D=B')
  \prod_{j=1}^h\frac1{K_j-1}
\leq
(2\ee)^hD^{-\sum_{j=1}^hc_j}
\Pp(\mathcal A_D=B'),
\end{aligned}
$$
where the inequality follows from $K_j-1\geq K_j/2$
and \eqref{eq:comparison-rounded-thresholds}.
Hence Lemma~\ref{lem:residual-sum-bound} gives
$$
\begin{aligned}
\sum_{B\in\mathcal B_\Lambda}\Pp(\mathcal A_D=B)
&\leq
(2\ee)^hD^{-\sum_{j=1}^hc_j}
\sum_{B'\in\widetilde{\mathcal R}_{c,D}}
\#\cL_\Lambda(B')\,\Pp(\mathcal A_D=B')\\
&\leq
(2\ee)^h\exp\{C_k(\log D)^{3/4}\}
D^{\min_{\cV'\leq\cV}
\mathrm e(\cV',\mathbf c,\boldsymbol\mu)
-\sum_{j=1}^hc_j},
\end{aligned}
$$
Note that
$$
\sum_{B'\in\widetilde{\mathcal R}_{c,D}}
\Pp(\mathcal A_D=B')
=
\Pp(\mathcal A_D\in\widetilde{\mathcal R}_{c,D})
\leq1.
$$
Since $\cV$ is complete, \eqref{eq:full-e} gives
$\mathrm e(\cV,\mathbf c,\boldsymbol\mu)=\sum_{j=1}^hc_j$.
Hence \eqref{eq:comparison-negative-gap} implies
$$
\min_{\cV'\leq\cV}
\mathrm e(\cV',\mathbf c,\boldsymbol\mu)
-\sum_{j=1}^hc_j
\leq-\eta.
$$
Substituting this into the preceding estimate, we obtain
$$
\sum_{B\in\mathcal B_\Lambda}\Pp(\mathcal A_D=B)
\leq
(2\ee)^h\exp\{C_k(\log D)^{3/4}\}D^{-\eta},
$$
where $C_k$ depends only on $k$.

Every $B\in\mathcal R_{c,D}$ satisfying $m(B)\geq k$
belongs to at least one $\mathcal B_\Lambda$.
There are $(\log D)^{O_k(1)}$ possible data $\Lambda$,
and $h\leq k-1$. Thus the preceding bound and
\eqref{eq:comparison-regularity} give
$$
\begin{aligned}
\Pp\bigl(m(\mathcal A_D)\geq k\bigr)
&\leq
\Pp(\mathcal A_D\notin\mathcal R_{c,D})
+\sum_\Lambda\sum_{B\in\mathcal B_\Lambda}
\Pp(\mathcal A_D=B)\\
&\ll_{k,c}
D^{-\eta}\exp\{C_k(\log D)^{3/4}\}
+\exp\left\{-\frac14\sqrt{\log D}\right\}=o(1),
\end{aligned}
$$
where $C_k$ is a sufficiently large constant depending only on $k$.
The last equality follows from $\eta>0$ and
$(\log D)^{3/4}=o(\log D)$.
Since $c\in(\gamma_k,1)$ was arbitrary, this establishes
\eqref{eq:comparison-upper-vanishing} and proves
$\beta_k\leq\gamma_k$.
\end{proof}

\begin{claim}\label{clm:comparison-lower}
$\widetilde\gamma_k\leq\beta_k$.
\end{claim}

\begin{proof}
Let $(\cV,\mathbf c,\boldsymbol\mu)$ satisfy the strict
entropy condition. Then $c_1>0$, since otherwise all
thresholds and all entropy gaps would vanish.
By Lemma~\ref{lem:threshold-normalization}, we may assume
$c_1=1$ without decreasing the final threshold.
Put $b=c_{r+1}$. It suffices to prove that $b\leq\beta_k$.

By Lemma~\ref{lem:finite-subflags} and the strict entropy condition, we have
$$
g=
\min_{\substack{
\cV'\in\mathscr F(\cV)\\
\cV'<\cV}}
\left(
\mathrm e(\cV',\mathbf c,\boldsymbol\mu)
-\mathrm e(\cV,\mathbf c,\boldsymbol\mu)
\right)>0.
$$
The same lemma therefore gives
\begin{equation}\label{eq:comparison-strict-gap}
\mathrm e(\cV',\mathbf c,\boldsymbol\mu)
\geq
\mathrm e(\cV,\mathbf c,\boldsymbol\mu)+g
\qquad(\cV'<\cV).
\end{equation}

We first show that the thresholds are strictly decreasing.
Suppose that $c_j=c_{j+1}$ for some $1\leq j\leq r$.
Replace $V_j$ by $V_{j-1}$ and leave all other spaces
unchanged, obtaining a proper subflag $\cV'<\cV$.
By \eqref{eq:e-value},
$$
\mathrm e(\cV',\mathbf c,\boldsymbol\mu)
-\mathrm e(\cV,\mathbf c,\boldsymbol\mu)
=
\begin{cases}
0,&j<r,\\[1mm]
-c_r\dim(V_r/V_{r-1}),&j=r.
\end{cases}
$$
Both cases contradict \eqref{eq:comparison-strict-gap}.
Hence
$$
1=c_1>c_2>\cdots>c_r>b\geq0.
$$

For each $1\leq j\leq r$, let $\sigma_j$ be the uniform
probability measure on $V_j\cap\{0,1\}^k$.
For $0\leq\zeta<\min\{1,c_r-b\}$, define
$$
\mu_j^{(\zeta)}=(1-\zeta)\mu_j+\zeta\sigma_j,
\qquad
\boldsymbol\mu^{(\zeta)}
=(\mu_1^{(\zeta)},\ldots,\mu_r^{(\zeta)}),
\qquad
\mathbf c^{(\zeta)}=(c_1,\ldots,c_r,b+\zeta).
$$
For each proper subflag $\cV'\in\mathscr F(\cV)$, the
difference
$
\mathrm e(\cV',\mathbf c^{(\zeta)},
\boldsymbol\mu^{(\zeta)})
-
\mathrm e(\cV,\mathbf c^{(\zeta)},
\boldsymbol\mu^{(\zeta)})
$
is continuous in $\zeta$ and is at least $g$ at $\zeta=0$
by \eqref{eq:comparison-strict-gap}.
Since $\mathscr F(\cV)$ is finite, we may fix
$0<\zeta<\min\{1,c_r-b\}$ sufficiently small that all
these differences are at least $g/2$.

For an arbitrary proper subflag $\cV'<\cV$,
Lemma~\ref{lem:finite-subflags} supplies a proper subflag
$\cV''\in\mathscr F(\cV)$ with the same $\mathrm e$-value. Hence
$$
\begin{aligned}
\mathrm e(\cV',\mathbf c^{(\zeta)},
\boldsymbol\mu^{(\zeta)})
=\mathrm e(\cV'',\mathbf c^{(\zeta)},
\boldsymbol\mu^{(\zeta)})
\geq
\mathrm e(\cV,\mathbf c^{(\zeta)},
\boldsymbol\mu^{(\zeta)})+\frac g2.
\end{aligned}
$$
Moreover, the choice of $\zeta$ gives
$$
1=c_1>\cdots>c_r>b+\zeta>0.
$$

Since $\sigma_j$ is uniform on $V_j\cap\{0,1\}^k$,
for every $\omega\in V_j\cap\{0,1\}^k$ we have
$$
\mu_j^{(\zeta)}(\omega)
=(1-\zeta)\mu_j(\omega)
+\frac{\zeta}{\#(V_j\cap\{0,1\}^k)}
\geq
\frac{\zeta}{\#(V_j\cap\{0,1\}^k)}>0,
$$
where we used $\mu_j(\omega)\geq0$ and $0<\zeta<1$.

Since $V_r$ is non-degenerate, the fixed system
$(\cV,\mathbf c^{(\zeta)},\boldsymbol\mu^{(\zeta)})$
satisfies conditions \textup{(i)}--\textup{(iii)} of
\cite[Proposition~5.5]{FGK}. That proposition therefore gives
$$
\Pp\left(
m\bigl(\mathcal A\cap[D^{b+\zeta},D]\bigr)\geq k
\right)=1-o(1).
$$

The sets $\mathcal A\cap[D^{b+\zeta},D]$ and
$\mathcal A\cap(D^{b+\zeta},D]$ can differ only at
$D^{b+\zeta}$. This point belongs to $\mathcal A$ with
probability at most $D^{-(b+\zeta)}=o(1)$.
Hence
$$
\begin{aligned}
\Pp\left(
m\bigl(\mathcal A\cap(D^{b+\zeta},D]\bigr)\geq k
\right)
&\geq
\Pp\left(
m\bigl(\mathcal A\cap[D^{b+\zeta},D]\bigr)\geq k
\right)
-D^{-(b+\zeta)}.
\end{aligned}
$$
The probability on the right tends to $1$ by the preceding
application of \cite[Proposition~5.5]{FGK}, while
$D^{-(b+\zeta)}\to0$ because $b+\zeta>0$. Therefore
$$
\lim_{D\to\infty}
\Pp\left(
m\bigl(\mathcal A\cap(D^{b+\zeta},D]\bigr)\geq k
\right)=1.
$$
Since $b+\zeta<1$, this exponent belongs to the set whose
supremum defines $\beta_k$. Thus
$$
\beta_k\geq b+\zeta\geq b.
$$
The original strict entropy system was arbitrary, and
normalization did not decrease its final threshold.
Therefore every final threshold in the definition of
$\widetilde\gamma_k$ is at most $\beta_k$.
Taking the supremum proves
$\widetilde\gamma_k\leq\beta_k$.

\end{proof}

Combining Claims~\ref{clm:comparison-upper}
and~\ref{clm:comparison-lower} proves the proposition.
\end{proof}

\section{Strict entropy inequalities and affine subset sums}\label{sec:strictification}

\subsection{Refinement to complete flags}

The upper-bound extraction of FGK \cite{FGK} produces
complete flags, whereas the optimization defining $\gamma_k$ is taken over
all entropy systems and therefore allows flags with arbitrary dimension
increments $\dim(V_j/V_{j-1})\geq 1.$
We remove this discrepancy by refining each flag into a complete flag
without changing the relevant $e$-values.

\begin{lemma}\label{lem:completion}
Let $(\cV,\mathbf c,\boldsymbol\mu)$ be an entropy system. There exists a
complete entropy system $(\cW,\widehat{\mathbf c},\widehat{\boldsymbol\mu})$
with the same terminal space $V_r$ and final threshold such that, for every
subflag $\cW'\leq\cW$, there exists a subflag $\cV'\leq\cV$ satisfying
\begin{equation*}
\mathrm e(\cW',\widehat{\mathbf c},\widehat{\boldsymbol\mu})
=
\mathrm e(\cV',\mathbf c,\boldsymbol\mu).
\end{equation*}
Moreover, $\mathrm e(\cW,\widehat{\mathbf c},\widehat{\boldsymbol\mu})
=\mathrm e(\cV,\mathbf c,\boldsymbol\mu).$
If $(\cV,\mathbf c,\boldsymbol\mu)$ satisfies the entropy condition,
then $(\cW,\widehat{\mathbf c},\widehat{\boldsymbol\mu})$ also satisfies it.
\end{lemma}

\begin{proof}
Put
\[
d_j=\dim(V_j/V_{j-1}),
\qquad
m_j=d_1+\cdots+d_j,
\qquad
m_0=0 .
\]
Set $W_0=V_0=\langle\one\rangle$.
For each $j$, choose $\omega_{j,1},\ldots,\omega_{j,d_j}\in V_j\cap\{0,1\}^k$
whose images form a basis of $V_j/V_{j-1}$.
Define
\[
W_{m_{j-1}+\ell}=
V_{j-1}
+
\operatorname{Span}_{\mathbb Q}
\{\omega_{j,1},\ldots,\omega_{j,\ell}\},
\qquad
1\leq \ell\leq d_j,
\]
where
$$
\operatorname{Span}_{\mathbb Q}
\{\omega_{j,1},\ldots,\omega_{j,\ell}\}
=
\left\{
\sum_{t=1}^{\ell}a_t\omega_{j,t}:a_t\in\mathbb Q
\right\}.$$
Since the images of
$\omega_{j,1},\ldots,\omega_{j,d_j}$
form a basis of $V_j/V_{j-1}$, they span the quotient $V_j/V_{j-1}$. Hence
$$
V_j=V_{j-1}+\operatorname{Span}_{\mathbb Q}
\{\omega_{j,1},\ldots,\omega_{j,d_j}\}.$$
By the definition of $W_{m_{j-1}+\ell}$ and $m_j=m_{j-1}+d_j$, we obtain
\begin{equation}\label{eq:completion-top-space} W_{m_j}=V_j \qquad(0\leq j\leq r).
\end{equation}

Moreover, for every $1\leq \ell\leq d_j$,
$\dim W_{m_{j-1}+\ell}-\dim W_{m_{j-1}+\ell-1}
=1.$
This follows because the images of
$\omega_{j,1},\ldots,\omega_{j,d_j}$
form a basis of $V_j/V_{j-1}$, so that each
$\omega_{j,\ell}$ adds exactly one new dimension modulo
$W_{m_{j-1}+\ell-1}$.

Hence, every step in the refined sequence increases the dimension by
exactly one, and therefore
\[
\cW:
\quad
W_0<W_1<\cdots<W_{m_r}
\]
is a complete flag. Since each inserted vector
$\omega_{j,\ell}$ belongs to $\{0,1\}^k$, every $W_i$ is spanned by cube
points. Thus $\cW$ is cube-spanned.

Define the new threshold sequence by
$\widehat c_\ell=c_j$ for
$1\leq j\leq r$ and $m_{j-1}<\ell\leq m_j$, and set
$\widehat c_{m_r+1}=c_{r+1}$.
For the measures, define
$\widehat\mu_{m_j}=\mu_j$ for $1\leq j\leq r$,
and choose arbitrary probability measures $\widehat\mu_\ell$ supported on
$W_\ell\cap\{0,1\}^k$ at the inserted levels
$m_{j-1}<\ell<m_j$. By construction,
\begin{equation}\label{eq:completion-threshold-difference}
\widehat c_\ell-\widehat c_{\ell+1}
=
\begin{cases}
0,
& m_{j-1}<\ell<m_j,\\[2mm]
c_j-c_{j+1},
& \ell=m_j,
\end{cases}
\qquad (1\leq j\leq r).
\end{equation}

Now let $\cW'\leq\cW$ be a subflag of $\cW$, and define
$V'_j=W'_{m_j}$ for $0\leq j\leq r$.
By \eqref{eq:completion-top-space}, we have
$W_{m_j}=V_j$. Since $\cW'\leq\cW$, we have
$W'_{m_j}\leq W_{m_j}$, and therefore
$V'_j=W'_{m_j}\leq W_{m_j}=V_j$.
Hence $\cV'\leq\cV$.

By \eqref{eq:completion-threshold-difference}, the entropy contribution
from the $j$-th block is
\begin{equation}\label{eq:completion-entropy-term}
\begin{aligned}
&\sum_{\ell=m_{j-1}+1}^{m_j}
(\widehat c_\ell-\widehat c_{\ell+1})
\Hh_{\widehat\mu_\ell}(W'_\ell)=
(\widehat c_{m_j}-\widehat c_{m_j+1})
\Hh_{\widehat\mu_{m_j}}(W'_{m_j})
=
(c_j-c_{j+1})
\Hh_{\mu_j}(V'_j).
\end{aligned}
\end{equation}

For the dimension contribution, $\widehat c_\ell=c_j$ gives
\begin{align}
\sum_{\ell=m_{j-1}+1}^{m_j}
\widehat c_\ell
\dim(W'_\ell/W'_{\ell-1})
&=
c_j
\sum_{\ell=m_{j-1}+1}^{m_j}
\dim(W'_\ell/W'_{\ell-1}) \notag\\
&=
c_j\dim(W'_{m_j}/W'_{m_{j-1}}) \notag\\
&=
c_j\dim(V'_j/V'_{j-1}).
\label{eq:completion-dimension-term}
\end{align}
Here the second equality follows from
$$
\dim(W'_{m_j})-\dim(W'_{m_{j-1}})
=\sum_{\ell=m_{j-1}+1}^{m_j}
\bigl(\dim W'_\ell-\dim W'_{\ell-1}\bigr).$$
Combining the entropy and dimension contributions in
\eqref{eq:completion-entropy-term} and
\eqref{eq:completion-dimension-term}, and using the definition
\eqref{eq:e-value} of the $e$-value, we obtain
\begin{equation}\label{eq:completion-equality-proof}
\mathrm e(\cW',\widehat{\mathbf c},\widehat{\boldsymbol\mu})
=
\mathrm e(\cV',\mathbf c,\boldsymbol\mu).
\end{equation}

Taking $\cW'=\cW$ in \eqref{eq:completion-equality-proof} gives
$\mathrm e(\cW,\widehat{\mathbf c},\widehat{\boldsymbol\mu})
=\mathrm e(\cV,\mathbf c,\boldsymbol\mu).$
Assume now that $(\cV,\mathbf c,\boldsymbol\mu)$ satisfies the
entropy condition. For $\cW'\leq\cW$ and the corresponding
subflag $\cV'\leq\cV$, \eqref{e1} and
\eqref{eq:completion-equality-proof} give
\[
\mathrm e(\cW',\widehat{\mathbf c},\widehat{\boldsymbol\mu})
\geq
\mathrm e(\cW,\widehat{\mathbf c},\widehat{\boldsymbol\mu}).
\]
Hence $(\cW,\widehat{\mathbf c},\widehat{\boldsymbol\mu})$ satisfies the
entropy condition.
\end{proof}

\subsection{Strict systems with zero final threshold}

Lemma~\ref{lem:completion} allows us to work with complete flags
while preserving the entropy condition. The following lemma
constructs a strict entropy system with final threshold zero
on such a flag.

\begin{lemma}\label{lem:strict-seed}
Let
$\mathscr V\colon
\langle\mathbf 1\rangle=V_0<V_1<\cdots<V_h$
be a complete flag in $\mathbb Q^k$.
Assume that each $V_j$ is spanned by points of $\{0,1\}^k$
and that the terminal space $V_h$ is non-degenerate.
Then there exist a threshold sequence
$\mathbf c^0=(c_1^0,\ldots,c_{h+1}^0)$
and probability measures
$\boldsymbol\mu^{(0)}=(\mu^{(0)}_1,\ldots,\mu^{(0)}_h)$,
with each $\mu^{(0)}_j$ supported on $V_j\cap\{0,1\}^k$,
such that
\[
1=c_1^0>c_2^0>\cdots>c_{h+1}^0=0
\]
and $(\mathscr V,\mathbf c^0,\boldsymbol\mu^{(0)})$ satisfies
the strict entropy condition.
The construction depends only on $\mathscr V$ and $k$.
\end{lemma}

\begin{proof}
Since $k\geq2$ and $V_h$ is non-degenerate, $h\geq1$.
For every $1\leq j\leq h$, choose a point
$\omega_j\in V_j\cap\{0,1\}^k$
such that $\omega_j\notin V_{j-1}.$
Since $\one\in V_0\subset V_j$ and
$\omega_j\in V_j\cap\{0,1\}^k$, we have
$\mathbf0_k,\omega_j,\one-\omega_j\in V_j\cap\{0,1\}^k$.
These three points are distinct. Otherwise, one of the relations
$\omega_j=\mathbf0_k$, $\one-\omega_j=\mathbf0_k$, or
$\omega_j=\one-\omega_j$ would imply
$\omega_j\in V_{j-1}$, contradicting the choice of $\omega_j$.

Define $\mu^{(0)}_j$ to be the probability measure on $\mathbb Q^k$ given by
\begin{equation}\label{lemma3-1}
\mu^{(0)}_j(x)=
\begin{cases}
1/3,
&x\in\{\mathbf0_k,\omega_j,\one-\omega_j\},\\
0,
&\text{otherwise}.
\end{cases}
\end{equation}
The three points belong to distinct cosets of $V_{j-1}$.
Indeed, $\omega_j\notin V_{j-1}$, while
$\one-\omega_j\notin V_{j-1}$ because $\one\in V_{j-1}$.
If $2\omega_j-\one\in V_{j-1}$, then $\omega_j\in V_{j-1}$,
a contradiction. Thus no difference between two of the three
points belongs to $V_{j-1}$.
Since the three points $\mathbf0_k,\omega_j,\one-\omega_j$
belong to distinct cosets of $V_{j-1}$, the measure $\mu^{(0)}_j$ assigns probability
$1/3$ to each of these cosets. Hence,
$\mu^{(0)}_j(V_{j-1}+x)=\frac13$
for every $x\in\{\mathbf0_k,\omega_j,\one-\omega_j\}.$
Therefore, by \eqref{eq:coset-entropy}, we have
\begin{equation}\label{eq:log3-entropy}
\Hh_{\mu^{(0)}_j}(V_{j-1})
=
-3\cdot\frac13\log\frac13
=
\log3 .
\end{equation}

Set $\vartheta=\frac{\log3-1}{4k}$ and define
$\lambda_j=
\frac{1-\vartheta}{1-\vartheta^h}\vartheta^{j-1}$ for $1\leq j\leq h.$ Then
\[
\begin{aligned}
\sum_{j=1}^{h}\lambda_j
=\frac{1-\vartheta}{1-\vartheta^h}
\sum_{j=1}^{h}\vartheta^{j-1}
=\frac{1-\vartheta}{1-\vartheta^h}
\cdot
\frac{1-\vartheta^h}{1-\vartheta}
=1 .
\end{aligned}
\]
Define the threshold sequence by
$c_j^0=\sum_{t=j}^{h}\lambda_t$ for $1\leq j\leq h$ and $c_{h+1}^0=0.$ It follows that
$c_1^0=1$ and $c_j^0-c_{j+1}^0=\lambda_j.$
We will show that $(\cV,\mathbf c^0,\boldsymbol\mu^{(0)})$
satisfies the strict entropy condition.

Let $\cV'\leq\cV$ be a subflag. Define $q_j=\dim(V_j/V'_j)$ for $1\leq j\leq h$, and set $q_0=0$.
Since $\cV$ is complete, we have
$\dim(V_j/V_{j-1})=1$. Moreover,
\[
q_j=\dim(V_j/V'_j)=\dim V_j-\dim V'_j,
\]
and hence $\dim V'_j=\dim V_j-q_j.$
Therefore,
\[
\begin{aligned}
\dim(V'_j/V'_{j-1})
&=
(\dim V_j-q_j)-(\dim V_{j-1}-q_{j-1})\\
&=
1-(q_j-q_{j-1}).
\end{aligned}
\]

Using the definitions \eqref{eq:e-value} and \eqref{eq:full-e} of the
$e$-value, together with
$c_j^0-c_{j+1}^0=\lambda_j$, we obtain
\begin{equation}\label{eq:seed-gap}
\begin{aligned}[c]
&\mathrm{e}(\mathscr{V}', \mathbf{c}^0, \boldsymbol{\mu}^{(0)})
- \mathrm{e}(\mathscr{V}, \mathbf{c}^0, \boldsymbol{\mu}^{(0)})\\
&=
\sum_{j=1}^{h} (c_j^0 - c_{j+1}^0) \mathbb{H}_{\mu^{(0)}_j}(V'_j)
+
\sum_{j=1}^{h}
c_j^0
\left(
\dim(V'_j / V'_{j-1})
- \dim(V_j / V_{j-1})
\right)\\
&=
\sum_{j=1}^{h} \lambda_j \mathbb{H}_{\mu^{(0)}_j}(V'_j)
-
\sum_{j=1}^{h} c_j^0 (q_j - q_{j-1})\\
&=
\sum_{j=1}^{h} \lambda_j \mathbb{H}_{\mu^{(0)}_j}(V'_j)
-
\sum_{j=1}^{h} \lambda_j q_j.
\end{aligned}
\end{equation}

Here, the last equality follows from the identity
\[
\begin{aligned}
\sum_{j=1}^{h}c_j^0(q_j-q_{j-1})
&=c_1^0q_1+c_2^0q_2-c_2^0q_1+\cdots
+c_h^0q_h-c_h^0q_{h-1}\\
&=(c_1^0-c_2^0)q_1
+(c_2^0-c_3^0)q_2
+\cdots
+(c_{h-1}^0-c_h^0)q_{h-1}
+c_h^0q_h\\
&=
\sum_{j=1}^{h}(c_j^0-c_{j+1}^0)q_j
=
\sum_{j=1}^{h}\lambda_jq_j,
\end{aligned}
\]
where we used $c_{h+1}^0=0$.

It remains to prove that the right-hand side of
\eqref{eq:seed-gap} is positive whenever $\cV'<\cV.$
Let $j_0$ be the smallest index such that $V'_{j_0}\neq V_{j_0}.$ By the choice of $j_0$,
$V'_{j_0-1}=V_{j_0-1}.$
Since $\cV$ is complete, we have
$\dim(V_{j_0}/V_{j_0-1})=1.$
Together with $V_{j_0-1}\leq V'_{j_0}\leq V_{j_0},$
this implies that $V'_{j_0}=V_{j_0-1}$ or $
V'_{j_0}=V_{j_0}.$
The latter is impossible by the definition of $j_0$, and hence $V'_{j_0}=V_{j_0-1}.$
Consequently,
\[
q_{j_0}
=
\dim(V_{j_0}/V'_{j_0})
=
\dim(V_{j_0}/V_{j_0-1})
=
1.
\]

Hence, by \eqref{eq:log3-entropy},
\[
\Hh_{\mu^{(0)}_{j_0}}(V'_{j_0})
=
\Hh_{\mu^{(0)}_{j_0}}(V_{j_0-1})
=
\log3 .
\]
Therefore, the $j_0$-th term in \eqref{eq:seed-gap} is
\[
\lambda_{j_0}
\bigl(
\Hh_{\mu^{(0)}_{j_0}}(V'_{j_0})-q_{j_0}
\bigr)
=
\lambda_{j_0}(\log3-1).
\]

For $t>j_0$, since $0\leq \mu^{(0)}_t(V'_t+x)\leq1,$
the definition \eqref{eq:coset-entropy} of the Shannon entropy gives $\Hh_{\mu^{(0)}_t}(V'_t)\geq0.$
Moreover,
\[
q_t=\dim(V_t/V'_t)\leq\dim(V_t/V_0)=t\leq h.
\]
Hence, $\Hh_{\mu^{(0)}_t}(V'_t)-q_t\geq-h.$

Since $j_0$ is the smallest index such that
$V'_{j_0}\neq V_{j_0}$, we have
$V'_j=V_j$ for $0\leq j<j_0$.
Thus
$q_j=\dim(V_j/V'_j)=0$ for $0\leq j<j_0$,
and $\Hh_{\mu^{(0)}_j}(V'_j)=\Hh_{\mu^{(0)}_j}(V_j)=0$ for $1\leq j<j_0$, by \eqref{lemma3-1}.
Therefore, all terms with $j<j_0$ in \eqref{eq:seed-gap} vanish. Combining these estimates with \eqref{eq:seed-gap}, we obtain
\[
\begin{aligned}
\mathrm e(\cV',\mathbf c^0,\boldsymbol\mu^{(0)})
-\mathrm e(\cV,\mathbf c^0,\boldsymbol\mu^{(0)})
&=\sum_{j=1}^{h}
\lambda_j\bigl(\Hh_{\mu^{(0)}_j}(V'_j)-q_j\bigr)\\
&=\lambda_{j_0}
\bigl(
\Hh_{\mu^{(0)}_{j_0}}(V'_{j_0})-q_{j_0}
\bigr)
+\sum_{t=j_0+1}^{h}
\lambda_t
\bigl(
\Hh_{\mu^{(0)}_t}(V'_t)-q_t
\bigr)
\\
&\geq
\lambda_{j_0}(\log3-1)
-
h\sum_{t=j_0+1}^{h}\lambda_t .
\end{aligned}
\]

By the definition of $\lambda_j$, we have $\lambda_{j_0+m}=\lambda_{j_0}\vartheta^m $ for $1\leq m\leq h-j_0.$
Hence,
\[
\sum_{t=j_0+1}^{h}\lambda_t
=
\lambda_{j_0}\sum_{m=1}^{h-j_0}\vartheta^m
\leq
\lambda_{j_0}\frac{\vartheta}{1-\vartheta}.
\]
Therefore,
\[
\mathrm e(\cV',\mathbf c^0,\boldsymbol\mu^{(0)})
-
\mathrm e(\cV,\mathbf c^0,\boldsymbol\mu^{(0)})
\geq
\lambda_{j_0}
\left(
\log3-1-\frac{h\vartheta}{1-\vartheta}
\right).
\]

Finally, since $\cV$ is complete and
$\dim V_0=1$, we have $\dim V_h=1+h.$
As $V_h\leq\mathbb Q^k$, it follows that $h\leq k-1<k.$ Since $\vartheta=\frac{\log3-1}{4k}<\frac14,$
we obtain
\[
\frac{h\vartheta}{1-\vartheta}
<\frac{k \vartheta}{1-\vartheta}
=\frac{\log3-1}{4(1-\vartheta)}
<
\log3-1 .
\]

Consequently, $\mathrm e(\cV',\mathbf c^0,\boldsymbol\mu^{(0)})
>\mathrm e(\cV,\mathbf c^0,\boldsymbol\mu^{(0)})$
for every proper subflag $\cV'<\cV$.
Therefore, $(\cV,\mathbf c^0,\boldsymbol\mu^{(0)})$
satisfies the strict entropy condition.
\end{proof}

\subsection{Concavity and strict entropy inequalities}
To mix systems, we use the finite measures obtained by multiplying
each probability measure by the width of its layer.

Let $S\subseteq V$ be a finite subset of a vector space $V$, and let
$U\le V$ be a subspace. Let $\mathbf 0=(0)_{\omega\in S}$ denote the
zero measure on $S$. For a finite nonnegative measure
$\mu=(\mu(\omega))_{\omega\in S}$ on $S$, put
$\lambda=\sum_{\omega\in S}\mu(\omega)$. If $\lambda>0$, define
\begin{equation}\label{defi-U}
\mathfrak H_U(\mu)=\lambda\Hh_{\mu/\lambda}(U),
\end{equation}
where $\mu/\lambda$ is
the probability measure on $S$ given by
$\left(\frac{\mu}{\lambda}\right)(A)
=\frac{1}{\lambda}\sum_{\omega\in A}\mu(\omega)$ for $A\subseteq S.$
Define $\mathfrak H_U(\mathbf 0)=0$.

The following lemma establishes
the concavity and positive homogeneity of this extension to
finite nonnegative measures.

\begin{lemma}\label{lem:perspective}
For any finite nonnegative measures $\mu$ and $\mu'$ on $S$ and
every $t\in[0,1]$, we have
\begin{equation*}
\mathfrak H_U\bigl(t\mu+(1-t)\mu'\bigr)
\ge
t\mathfrak H_U(\mu)+(1-t)\mathfrak H_U(\mu').
\end{equation*}
Moreover, for every finite nonnegative measure $\mu$ on $S$ and every
real number $a\ge0$,
$\mathfrak H_U(a\mu)=a\mathfrak H_U(\mu).$
\end{lemma}

\begin{proof}
If $S=\varnothing$, every measure on $S$ is zero and both assertions
follow from the definition. Assume $S\neq\varnothing$.
Let $C_1,\ldots,C_m$ be the distinct cosets of $U$ that intersect
$S$. For each $1\le i\le m$, put $x_i=\sum_{\omega\in C_i\cap S}\mu(\omega) .$
Then $\lambda=\sum_{i=1}^{m}x_i.$
Suppose first that $\lambda>0$. If $\omega\in C_i\cap S$, then
$U+\omega=C_i$, and hence
$$
\left(\frac{\mu}{\lambda}\right)(U+\omega)
=\sum_{\eta\in C_i\cap S}\frac{\mu(\eta)}{\lambda}
=\frac{x_i}{\lambda}.
$$
Applying \eqref{eq:coset-entropy} to $\mu/\lambda$ and grouping
the terms according to the cosets $C_1,\ldots,C_m$, we obtain
\begin{align*}
\Hh_{\mu/\lambda}(U)
&=
-\sum_{i=1}^{m}
\sum_{\omega\in C_i\cap S}
\frac{\mu(\omega)}{\lambda}
\log
\left(
\sum_{\eta\in C_i\cap S}
\frac{\mu(\eta)}{\lambda}
\right)\\
&=
-\sum_{i=1}^{m}
\sum_{\omega\in C_i\cap S}
\frac{\mu(\omega)}{\lambda}
\log\frac{x_i}{\lambda}\\
&=
-\sum_{i=1}^{m}
\left(
\sum_{\omega\in C_i\cap S}
\frac{\mu(\omega)}{\lambda}
\right)
\log\frac{x_i}{\lambda}\\
&=
-\sum_{i=1}^{m}
\frac{x_i}{\lambda}
\log\frac{x_i}{\lambda}.
\end{align*}

By the definition of $\mathfrak H_U$,
\begin{align}
\mathfrak H_U(\mu)
=\lambda\Hh_{\mu/\lambda}(U)
=-\sum_{i=1}^{m}x_i\bigl(\log x_i-\log\lambda\bigr)
=\lambda\log\lambda-\sum_{i=1}^{m}x_i\log x_i.
\label{eq:perspective-explicit}
\end{align}

The last expression in \eqref{eq:perspective-explicit} is also zero
when $\mu=\mathbf0$, by the convention $0\log0=0$. It therefore represents
$\mathfrak H_U$ on the whole nonnegative orthant and is continuous
there, since $u\log u$ tends to zero as $u$ tends to zero from above.

We next prove the concavity inequality. Let
$\mu'=(\mu'(\omega))_{\omega\in S}$ be another finite nonnegative measure on
$S$, and define $y_i=\sum_{\omega\in C_i\cap S}\mu'(\omega)$ for
$1\leq i\leq m$. First assume that $x_i>0$ and $y_i>0$ for every $i$.

For $s\in[0,1]$, set $d_i=x_i-y_i$ and
$$
r_i(s)=sx_i+(1-s)y_i=y_i+sd_i
$$
for $1\le i\le m$, and define
$R(s)=\sum_{i=1}^{m}r_i(s)$. Then $r_i(s)>0$,
$r_i'(s)=d_i$, and $R'(s)=\sum_{i=1}^{m}d_i$.

By the definitions of $x_i$, $y_i$, and $r_i(s)$, we have, for
$1\le i\le m$,
\begin{align}
\bigl(s\mu+(1-s)\mu'\bigr)(C_i\cap S)
&=
\sum_{\omega\in C_i\cap S}
\bigl(s\mu(\omega)+(1-s)\mu'(\omega)\bigr)
=sx_i+(1-s)y_i
=r_i(s).
\label{eq:mixed-coset-sum}
\end{align}
Since $C_1\cap S,\ldots,C_m\cap S$ form a partition of $S$,
\begin{align}
\sum_{\omega\in S}
\bigl(s\mu(\omega)+(1-s)\mu'(\omega)\bigr)
=\sum_{i=1}^{m}
\sum_{\omega\in C_i\cap S}
\bigl(s\mu(\omega)+(1-s)\mu'(\omega)\bigr)
=\sum_{i=1}^{m}r_i(s)
=R(s).
\label{eq:mixed-total-sum}
\end{align}

Define
$g(s)=\mathfrak H_U\bigl(s\mu+(1-s)\mu'\bigr)$ for $s\in[0,1]$.
Applying \eqref{eq:perspective-explicit} to the measure
$s\mu+(1-s)\mu'$, and using \eqref{eq:mixed-coset-sum} and
\eqref{eq:mixed-total-sum}, we obtain
\begin{align*}
g(s)
&=
\mathfrak H_U\bigl(s\mu+(1-s)\mu'\bigr) \\
&=
\left(
\sum_{\omega\in S}
\bigl(s\mu(\omega)+(1-s)\mu'(\omega)\bigr)
\right)
\log
\left(
\sum_{\omega\in S}
\bigl(s\mu(\omega)+(1-s)\mu'(\omega)\bigr)
\right)
\\
&\quad-
\sum_{i=1}^{m}
\bigl(s\mu+(1-s)\mu'\bigr)(C_i\cap S)
\log
\bigl(s\mu+(1-s)\mu'\bigr)(C_i\cap S)
\\
&=R(s)\log R(s)-\sum_{i=1}^{m}r_i(s)\log r_i(s).
\end{align*}
Differentiating with respect to $s$ gives
\begin{align*}
g'(s)
&=R'(s)\bigl(\log R(s)+1\bigr)
-\sum_{i=1}^{m}
r_i'(s)\bigl(\log r_i(s)+1\bigr)\\
&=\left(\sum_{i=1}^{m}d_i\right)
\bigl(\log R(s)+1\bigr)
-\sum_{i=1}^{m}
d_i\bigl(\log r_i(s)+1\bigr)\\
&=\left(\sum_{i=1}^{m}d_i\right)\log R(s)
-\sum_{i=1}^{m}d_i\log r_i(s).
\end{align*}
Differentiating once more, we obtain
\begin{align}
g''(s)=\left(\sum_{i=1}^{m}d_i\right)
\frac{R'(s)}{R(s)}-\sum_{i=1}^{m}
d_i\frac{r_i'(s)}{r_i(s)}
=\frac{\left(\sum_{i=1}^{m}d_i\right)^2}{R(s)}
-\sum_{i=1}^{m}\frac{d_i^2}{r_i(s)}.
\label{eq:line-second-derivative}
\end{align}

Since $r_i(s)>0$ for $1\le i\le m$, the Cauchy--Schwarz inequality
gives
\begin{align}
\left(\sum_{i=1}^{m}d_i\right)^2
&=
\left(
\sum_{i=1}^{m}
\sqrt{r_i(s)}
\frac{d_i}{\sqrt{r_i(s)}}
\right)^2 \notag\\
&\le
\left(\sum_{i=1}^{m}r_i(s)\right)
\left(\sum_{i=1}^{m}\frac{d_i^2}{r_i(s)}\right)=R(s)\sum_{i=1}^{m}\frac{d_i^2}{r_i(s)}.
\label{eq:line-cauchy}
\end{align}
Since $R(s)>0$, it follows from \eqref{eq:line-cauchy} that 
$$
\frac{\left(\sum_{i=1}^{m}d_i\right)^2}{R(s)}
\le
\sum_{i=1}^{m}\frac{d_i^2}{r_i(s)}.
$$
Combining this inequality with \eqref{eq:line-second-derivative}, we
obtain $g''(s)\le0$ for $0\le s\le1.$
Therefore, $g$ is concave on $[0,1]$, and hence
$$
g(t)\ge tg(1)+(1-t)g(0)
$$
for $0\le t\le1.$
Moreover,
$g(t)=\mathfrak H_U\bigl(t\mu+(1-t)\mu'\bigr)$,
$g(1)=\mathfrak H_U(\mu)$, and $g(0)=\mathfrak H_U(\mu')$. Therefore,
\[
\mathfrak H_U\bigl(t\mu+(1-t)\mu'\bigr)\ge
t\mathfrak H_U(\mu)+(1-t)\mathfrak H_U(\mu')
\]
whenever $x_i>0$ and
$y_i>0$ for every $1\le i\le m$.

For arbitrary nonnegative measures $\mu,\mu'$, choose the measure $\mu_*$ on $S$
with $\mu_*(\omega)=1$ for every $\omega\in S$. For $\xi>0$, put
$\mu^{(\xi)}=\mu+\xi\mu_*$ and $(\mu')^{(\xi)}=\mu'+\xi\mu_*$.
All their coset totals are positive, so the preceding argument gives
\[
\mathfrak H_U\bigl(t\mu^{(\xi)}+(1-t)(\mu')^{(\xi)}\bigr)
\geq t\mathfrak H_U(\mu^{(\xi)})+(1-t)\mathfrak H_U((\mu')^{(\xi)}).
\]
Letting $\xi\to0$ and using continuity proves the concavity inequality
for every pair of nonnegative measures, including zero measures.

It remains to prove the homogeneity property. Let $a\ge0$. If $a=0$
or $\mu=\mathbf 0$, then $a\mu=\mathbf 0$, and hence
$$
\mathfrak H_U(a\mu)=\mathfrak H_U(\mathbf 0)=0=a\mathfrak H_U(\mu).
$$

Assume now that $a>0$ and $\mu\ne\mathbf 0$. Since $\mu(\omega)\ge0$ for
every $\omega\in S$, we have
$\lambda=\sum_{\omega\in S}\mu(\omega)>0$. The measure
$a\mu=(a\mu(\omega))_{\omega\in S}$ satisfies
$$
\sum_{\omega\in S}a\mu(\omega)
=a\sum_{\omega\in S}\mu(\omega)=a\lambda.
$$
Moreover, since $a>0$ and
$\lambda>0$, we have $\frac{a\mu}{a\lambda}=\frac{\mu}{\lambda}.$
Consequently, by the definition of $\mathfrak H_U$,
\[
\mathfrak H_U(a\mu)
=a\lambda\Hh_{a\mu/(a\lambda)}(U)
=a\lambda\Hh_{\mu/\lambda}(U)
=a\bigl(\lambda\Hh_{\mu/\lambda}(U)\bigr)
=a\mathfrak H_U(\mu).
\]
Therefore, $\mathfrak H_U(a\mu)=a\mathfrak H_U(\mu)$ for every finite
nonnegative measure $\mu$ on $S$ and every $a\ge0$.
\end{proof}

We now combine the preceding lemmas to obtain strict entropy
inequalities from the entropy condition. The construction refines
the flag and allows the final threshold to be replaced by any
prescribed smaller positive value.

\begin{theorem}\label{thm:strictification}
Let $(\cV,\mathbf c,\boldsymbol\mu)$ satisfy the entropy condition,
where $\cV=(V_0,\ldots,V_r)$,
$\mathbf c=(c_1,\ldots,c_{r+1})$, $c_1=1$, and
$b=c_{r+1}>0$. Then, for every $c^\ast\in(0,b)$, there exist
a complete flag
$\cW:V_0=W_0<W_1<\cdots<W_h=V_r$
with $\{V_0,\ldots,V_r\}\subseteq\{W_0,\ldots,W_h\}$,
a threshold sequence
$\widetilde{\mathbf c}
=(\widetilde c_1,\ldots,\widetilde c_{h+1})$, and a family of
probability measures
$\widetilde{\boldsymbol\mu}
=(\widetilde\mu_1,\ldots,\widetilde\mu_h)$ such that
$\widetilde c_1=1$, $\widetilde c_{h+1}=c^\ast$, and
$(\cW,\widetilde{\mathbf c},\widetilde{\boldsymbol\mu})$
satisfies the strict entropy condition.
\end{theorem}

\begin{proof}
By Lemma~\ref{lem:completion}, there exist a complete flag
\[
\cW:\quad
W_0<W_1<\cdots<W_h
\]
of $\cV$, a threshold sequence $\mathbf c^{(1)}
=(c^{(1)}_1,\ldots,c^{(1)}_{h+1}),$
and probability measures
$\boldsymbol\mu^{(1)}=(\mu^{(1)}_1,\ldots,\mu^{(1)}_h)$
such that $(\cW,\mathbf c^{(1)},\boldsymbol\mu^{(1)})$
satisfies the entropy condition with $c^{(1)}_1=1$
and $c^{(1)}_{h+1}=b.$ Therefore, by \eqref{e1}, for every subflag $\cW'\leq \cW$,
\begin{equation}\label{eq:entropy-gap-nonstrict}
\mathrm e(\cW',\mathbf c^{(1)},\boldsymbol\mu^{(1)})-\mathrm e(\cW,\mathbf c^{(1)},\boldsymbol\mu^{(1)})\geq 0 .
\end{equation}

On the other hand, by Lemma~\ref{lem:strict-seed}, the same complete
flag $\cW$ admits a strict entropy system
$(\cW,\mathbf c^{(0)},\boldsymbol\mu^{(0)})$,
where
$\mathbf c^{(0)}
=
(c^{(0)}_1,\ldots,c^{(0)}_{h+1})$
and
$\boldsymbol\mu^{(0)}
=
(\mu^{(0)}_1,\ldots,\mu^{(0)}_h)$,
such that
$c^{(0)}_1=1$,
$c^{(0)}_{h+1}=0$,
and
$c^{(0)}_j-c^{(0)}_{j+1} > 0$
for every $1\leq j\leq h$.
Since this system satisfies the strict entropy condition, every subflag $\cW'<\cW$ satisfies
\begin{equation}\label{eq:entropy-gap-strict}
\mathrm e(\cW',\mathbf c^{(0)},\boldsymbol\mu^{(0)})
-\mathrm e(\cW,\mathbf c^{(0)},\boldsymbol\mu^{(0)})
>0 .
\end{equation}

Now define $t=\frac{c^\ast}{b}$. Since $0<c^\ast<b$, we have $0<t<1$. For $i\in\{0,1\}$ and $1\leq j\leq h$, put
$\lambda^{(i)}_j=c^{(i)}_j-c^{(i)}_{j+1}$.
Define the sequence
$\widetilde{\mathbf c}
=(\widetilde c_1,\ldots,\widetilde c_{h+1})$ by
\begin{equation}\label{eq:threshold-mixture}
\widetilde c_j=
t c^{(1)}_j+(1-t)c^{(0)}_j
\qquad(1\leq j\leq h+1).
\end{equation}

Let $\widetilde{\lambda}_j=\widetilde c_j-\widetilde c_{j+1}.$
Then
\begin{align}
\widetilde{\lambda}_j
=t(c^{(1)}_j-c^{(1)}_{j+1})+(1-t)(c^{(0)}_j-c^{(0)}_{j+1})
=t\lambda^{(1)}_j+(1-t)\lambda^{(0)}_j .
\label{eq:lambda-mixture}
\end{align}
The entropy system $(\cW,\mathbf c^{(1)},\boldsymbol\mu^{(1)})$
has non-increasing thresholds, so
$\lambda^{(1)}_j\geq0$.
Moreover, $\lambda^{(0)}_j>0$ and $1-t>0$. Hence
$\widetilde{\lambda}_j>0.$
Consequently,
$\widetilde c_1>\widetilde c_2>\cdots>\widetilde c_{h+1}$.
By \eqref{eq:threshold-mixture}, we have
$$
\widetilde c_1=tc^{(1)}_1+(1-t)c^{(0)}_1=1
$$
and 
$$
\widetilde c_{h+1}=tc^{(1)}_{h+1}+(1-t)c^{(0)}_{h+1}=
tb=c^\ast.
$$
Therefore,
$\widetilde{\mathbf c}=(\widetilde c_1,\ldots,\widetilde c_{h+1})$ has final threshold $c^\ast$.

We next define the measures. For every $1\leq j\leq h$ and every
$\omega\in W_j\cap\{0,1\}^k$, set
\begin{equation}\label{eq:measure-mixture}
\widetilde\mu_j(\omega)
=
\frac{
t\lambda^{(1)}_j\mu^{(1)}_j(\omega)
+
(1-t)\lambda^{(0)}_j\mu^{(0)}_j(\omega)
}
{\widetilde\lambda_j}.
\end{equation}

We have
$\lambda^{(1)}_j\geq0$,
$\lambda^{(0)}_j>0$,
and
$\mu^{(i)}_j(\omega)\geq0$ for $i\in\{0,1\}$.
Together with $0<t<1$, this implies
$$
t\lambda^{(1)}_j\mu^{(1)}_j(\omega)
+(1-t)\lambda^{(0)}_j\mu^{(0)}_j(\omega)
\geq0 .
$$
Therefore, $\widetilde\mu_j(\omega)\geq0.$
Furthermore, since $\mu^{(0)}_j$ and $\mu^{(1)}_j$ are probability
measures on $W_j\cap\{0,1\}^k$, we have
\[
\sum_{\omega\in W_j\cap\{0,1\}^k}\mu^{(0)}_j(\omega)
=\sum_{\omega\in W_j\cap\{0,1\}^k}\mu^{(1)}_j(\omega)
=1 .
\]
Therefore,
\begin{align*}
\sum_{\omega\in W_j\cap\{0,1\}^k}
\widetilde\mu_j(\omega)
&=\frac{
t\lambda^{(1)}_j
\sum_{\omega\in W_j\cap\{0,1\}^k}
\mu^{(1)}_j(\omega)
+(1-t)\lambda^{(0)}_j
\sum_{\omega\in W_j\cap\{0,1\}^k}
\mu^{(0)}_j(\omega)
}{\widetilde\lambda_j}\\
&=\frac{t\lambda^{(1)}_j+(1-t)\lambda^{(0)}_j
}{\widetilde\lambda_j}=1,
\end{align*}
where the last equality follows from \eqref{eq:lambda-mixture}. Thus each $\widetilde\mu_j$ is a probability measure.

It remains to prove the strict entropy condition. Let $\cW'\leq\cW$ be an arbitrary
subflag. From \eqref{eq:measure-mixture}, we have
$$
\widetilde\lambda_j\widetilde\mu_j=t\lambda^{(1)}_j\mu^{(1)}_j
+(1-t)\lambda^{(0)}_j\mu^{(0)}_j
$$
as finite nonnegative measures.

Applying Lemma~\ref{lem:perspective} with
$\mu=\lambda^{(1)}_j\mu^{(1)}_j$,
$\mu'=\lambda^{(0)}_j\mu^{(0)}_j$, and $U=W'_j$, and using
\[
\mathfrak H_{W'_j}(\lambda^{(i)}_j\mu^{(i)}_j)
=\lambda^{(i)}_j\Hh_{\mu^{(i)}_j}(W'_j)
\]
for $i\in\{0,1\}$, together with
\[
\mathfrak H_{W'_j}(\widetilde\lambda_j\widetilde\mu_j)
=\widetilde\lambda_j\Hh_{\widetilde\mu_j}(W'_j),
\]
we obtain
\begin{align}
\widetilde\lambda_j
\Hh_{\widetilde\mu_j}(W'_j)
\geq&
t\lambda^{(1)}_j
\Hh_{\mu^{(1)}_j}(W'_j)
+
(1-t)\lambda^{(0)}_j
\Hh_{\mu^{(0)}_j}(W'_j).
\label{eq:entropy-concavity}
\end{align}

Using the definition \eqref{eq:e-value} of the $\mathrm e$-value and
\eqref{eq:threshold-mixture},
\eqref{eq:lambda-mixture}, and
\eqref{eq:entropy-concavity}, we obtain

{\small
\begin{align}
\mathrm e(\cW',\widetilde{\mathbf c},
\widetilde{\boldsymbol\mu})
&=\sum_{j=1}^{h}\widetilde\lambda_j
\Hh_{\widetilde\mu_j}(W'_j)
+\sum_{j=1}^{h}\widetilde c_j
\dim(W'_j/W'_{j-1})
\nonumber\\
&\geq \sum_{j=1}^{h}
\left(t\lambda^{(1)}_j\Hh_{\mu^{(1)}_j}(W'_j)
+(1-t)\lambda^{(0)}_j\Hh_{\mu^{(0)}_j}(W'_j)\right)\nonumber+
\sum_{j=1}^{h}
\left(t c^{(1)}_j+(1-t)c^{(0)}_j\right)\dim(W'_j/W'_{j-1}) \nonumber\\
&=
t\,\mathrm e(\cW',\mathbf c^{(1)},
\boldsymbol\mu^{(1)})
+
(1-t)\,\mathrm e(\cW',\mathbf c^{(0)},
\boldsymbol\mu^{(0)}).
\label{eq:e-mixture-lower}
\end{align}
}

For the full flag $\cW$, each of the measures
$\mu_j^{(i)}$ and $\widetilde\mu_j$ is supported on
$W_j\cap\{0,1\}^k$. Therefore, $\Hh_{\mu_j^{(i)}}(W_j)=0$ and $\Hh_{\widetilde\mu_j}(W_j)=0 .$
Hence
\begin{align}
\mathrm e(\cW,\widetilde{\mathbf c},
\widetilde{\boldsymbol\mu})
=t\,
\mathrm e(\cW,\mathbf c^{(1)},
\boldsymbol\mu^{(1)})+(1-t)\,
\mathrm e(\cW,\mathbf c^{(0)},
\boldsymbol\mu^{(0)}).
\label{eq:e-full-mixture}
\end{align}

Subtracting \eqref{eq:e-full-mixture} from
\eqref{eq:e-mixture-lower}, we obtain
\begin{align}
&
\mathrm e(\cW',\widetilde{\mathbf c},
\widetilde{\boldsymbol\mu})
-\mathrm e(\cW,\widetilde{\mathbf c},
\widetilde{\boldsymbol\mu})
\nonumber \geq t
\Big(
\mathrm e(\cW',\mathbf c^{(1)},
\boldsymbol\mu^{(1)})
-\mathrm e(\cW,\mathbf c^{(1)},
\boldsymbol\mu^{(1)})
\Big)\nonumber\\
&\quad+
(1-t)
\Big(
\mathrm e(\cW',\mathbf c^{(0)},
\boldsymbol\mu^{(0)})
-\mathrm e(\cW,\mathbf c^{(0)},
\boldsymbol\mu^{(0)})
\Big).
\label{eq:gap-mixture}
\end{align}
If $\cW'=\cW$, the difference is zero. Suppose now that
$\cW'<\cW$. Since $1-t>0$,
\eqref{eq:entropy-gap-nonstrict} and
\eqref{eq:entropy-gap-strict} give $\mathrm e(\cW',\widetilde{\mathbf c},
\widetilde{\boldsymbol\mu})
>\mathrm e(\cW,\widetilde{\mathbf c},
\widetilde{\boldsymbol\mu}).$
Therefore, the minimum of the $\mathrm e$-value is attained uniquely at
the full flag $\cW$.

Hence $(\cW,\widetilde{\mathbf c},
\widetilde{\boldsymbol\mu})$
satisfies the strict entropy condition, $\cW$ is a complete refinement
of $\cV$, and $\widetilde c_1=1,\widetilde c_{h+1}=c^\ast.$
This completes the proof.
\end{proof}

\subsection{Entropy thresholds and a uniform gap}

We now use Theorem~\ref{thm:strictification} to show that the
entropy condition and the strict entropy condition define the
same threshold. Proposition~\ref{prop:threshold-comparison}
then identifies this common value with $\beta_k$.

\begin{proof}[Proof of Theorem~\ref{thm:entropy-equality}]
Since the strict entropy condition implies the entropy condition, it follows that
$\widetilde\gamma_k\leq\gamma_k$.
Let $\mathbf e_1,\ldots,\mathbf e_k$ be the standard basis of
$\Q^k$, and put
\[
U_j=\Span_{\Q}\{\one,\mathbf e_1,\ldots,\mathbf e_j\}
\qquad(0\leq j\leq k-1).
\]
Applying Lemma~\ref{lem:strict-seed} to the complete flag
$U_0<\cdots<U_{k-1}=\Q^k$ gives
$\widetilde\gamma_k\geq0$.
Thus the case $\gamma_k=0$ is immediate.

Suppose that $\gamma_k>0$, and fix $c^\ast\in(0,\gamma_k)$.
By the definition of $\gamma_k$, there exists an entropy system
$(\cV,\mathbf c,\boldsymbol\mu)$ satisfying the entropy condition
with $c_{r+1}>c^\ast$.
Define
\[
\widehat{\mathbf c}=(\widehat c_1,\ldots,\widehat c_{r+1}),
\qquad
\widehat c_j=\frac{c_j}{c_1}.
\]
Since $1\geq c_1\geq c_{r+1}>0$,
Lemma~\ref{lem:threshold-normalization} implies that
$(\cV,\widehat{\mathbf c},\boldsymbol\mu)$ satisfies the entropy
condition, with
\[
\widehat c_1=1,
\qquad
\widehat c_{r+1}
=\frac{c_{r+1}}{c_1}
\geq c_{r+1}>c^\ast.
\]
Theorem~\ref{thm:strictification} therefore yields a system satisfying
the strict entropy condition with final threshold $c^\ast$.
Hence $c^\ast\leq\widetilde\gamma_k$.
Taking the supremum over $c^\ast\in(0,\gamma_k)$ gives
$\gamma_k\leq\widetilde\gamma_k$, and consequently
$\widetilde\gamma_k=\gamma_k$.

Finally, Proposition~\ref{prop:threshold-comparison} gives
$\widetilde\gamma_k\leq\beta_k\leq\gamma_k$.
Therefore, we have 
\[
\widetilde\gamma_k=\beta_k=\gamma_k,
\]
as required.
\end{proof}

For later use, if $c\in(0,1]$, define
\begin{equation}\label{eq:Theta-def}
\Theta_k(c)=
\sup_{\substack{(\cV,\mathbf c,\boldsymbol\mu)\ \mathrm{complete}\\c_{r+1}=c}}\ \min_{\cV'\le\cV}\left(\mathrm e(\cV',\mathbf c,\boldsymbol\mu)-\mathrm e(\cV,\mathbf c,\boldsymbol\mu)\right).
\end{equation}

\begin{corollary}\label{cor:uniform-gap}
Fix $k\geq2$ and $\delta>0$. There exists $\eta=\eta(k,\delta)>0$
such that $\Theta_k(c)\leq-\eta$
for every $c\in[\beta_k+\delta,1].$
Equivalently, every complete system
$(\cV,\mathbf c,\boldsymbol\mu)$ satisfying
$c_{r+1}\in[\beta_k+\delta,1]$ admits a proper subflag
$\cV'<\cV$ such that
\[
\mathrm e(\cV',\mathbf c,\boldsymbol\mu)
\leq
\mathrm e(\cV,\mathbf c,\boldsymbol\mu)-\eta.
\]
\end{corollary}

\begin{proof}
We may assume that $\beta_k+\delta\leq1$, since otherwise the interval
$[\beta_k+\delta,1]$ is empty.
For a complete system $(\cV,\mathbf c,\boldsymbol\mu)$, define
\[
G(\cV,\mathbf c,\boldsymbol\mu)
=
\min_{\cV'\leq\cV}
\left(
\mathrm e(\cV',\mathbf c,\boldsymbol\mu)
-
\mathrm e(\cV,\mathbf c,\boldsymbol\mu)
\right).
\]
By Lemma~\ref{lem:finite-subflags}, the minimum defining $G$
is attained. Taking $\cV'=\cV$ gives
$G(\cV,\mathbf c,\boldsymbol\mu)\leq0$.
Moreover, $G(\cV,\mathbf c,\boldsymbol\mu)=0$ if and only if
\eqref{e1} holds for every subflag $\cV'\leq\cV$,
which is precisely the entropy condition.

By Theorem~\ref{thm:entropy-equality}, $\gamma_k=\beta_k$. Hence every
complete system with $c_{r+1}\geq\beta_k+\delta$ satisfies $G(\cV,\mathbf c,\boldsymbol\mu)<0.$
Indeed, equality would imply that the system satisfies the entropy
condition, and therefore
\[
\gamma_k\geq c_{r+1}\geq\beta_k+\delta>\beta_k,
\]
a contradiction.

Since $\{0,1\}^k$ is finite, only finitely many complete cube-spanned
flags can occur. For each fixed complete flag $\cV$, the admissible
parameters $(\mathbf c,\boldsymbol\mu)$ with
$c_{r+1}\in[\beta_k+\delta,1]$ form a closed and bounded subset of a
finite-dimensional Euclidean space, and hence a compact set. By
Lemma~\ref{lem:finite-subflags}, the minimum defining $G$ may be taken
over the fixed finite family $\mathscr F(\cV)$. Since the $\mathrm e$-value is continuous
in $\mathbf c$ and $\boldsymbol\mu$, the function $G$ is continuous.

By compactness and continuity, the quantity
\[
M_{k,\delta}=
\max_{\substack{
(\cV,\mathbf c,\boldsymbol\mu)\ \mathrm{complete}\\
c_{r+1}\in[\beta_k+\delta,1]
}}
G(\cV,\mathbf c,\boldsymbol\mu)
\]
is well defined. Since every complete system occurring in this maximum
satisfies $G(\cV,\mathbf c,\boldsymbol\mu)<0,$ we have $M_{k,\delta}<0$. Set
$\eta=-M_{k,\delta}>0$. Then $G(\cV,\mathbf c,\boldsymbol\mu)\leq-\eta$
for every complete system with
$c_{r+1}\in[\beta_k+\delta,1]$.

\end{proof}

\subsection{Affine subset-sum estimates}
\label{subsec:affine-window-stability}
Throughout this subsection, all sets under consideration are sets of
positive integers, and we retain the notation $\Sigma(B)=\sum_{b\in B}b$.

Fix an integer \(k\geq 2\), a real number \(D>1\), and a real number
\(c\). Recall that \(\mathcal A\) is the logarithmic random set defined in
Subsection~\ref{subsec:close-divisors}.

Let $\tau=(\tau_2,\ldots,\tau_k)\in\mathbb R^{k-1},$
where each \(\tau_s\) is a given real number, and let \(W\geq0\)
be a real number. We define
\(\mathcal N_{k,c,D}(\tau,W)\) to be the event that there exist
\(k\) pairwise distinct subsets $A_1,\ldots,A_k\subseteq\mathcal A\cap(D^c,D]$
such that
\begin{equation*}
\left|
\Sigma(A_s)-\Sigma(A_1)-\tau_s
\right|
\leq W
\qquad(2\leq s\leq k).
\end{equation*}

We now apply the uniform entropy gap to subset sums with prescribed
differences. The following theorem bounds the probability of
$\mathcal N_{k,c,D}(\tau,W)$ uniformly in $c\geq\beta_k+\delta$
and $\tau$, provided that $1+W=D^{o(1)}$.

\begin{theorem}\label{thm:affine}
Fix an integer $k\geq2$ and a real number $\delta>0$. Then there exists a constant $\eta=\eta(k,\delta)>0$
such that the following estimate holds as $D\to\infty$. For all real numbers $c$ satisfying $c\geq\beta_k+\delta,$
all vectors $\tau=(\tau_2,\ldots,\tau_k)\in\R^{k-1},$
and every nonnegative function $W=W(D)$ satisfying $1+W=D^{o(1)}$, one has
\begin{equation*}
\Pp\bigl(\mathcal N_{k,c,D}(\tau,W)\bigr)
\leq
D^{-\eta+o(1)}
+
O_k\!\left(
\exp\left\{-\frac14\sqrt{\log D}\right\}
\right).
\end{equation*}
For each fixed function $W=W(D)$ satisfying the stated condition,
the $o(1)$ term is uniform in $c$ and $\tau$. More precisely, the exponent denoted by $o(1)$ may be taken to be
\[
O_{k,\delta}\!\left(
(\log D)^{-1/4}+\frac{\log(2+W)}{\log D}
\right),
\]
with an implied constant independent of $c$, $\tau$, and $W$.
\end{theorem}

\begin{proof}
If $c\geq1$, then $(D^c,D]=\varnothing$ and the event is empty.
It is therefore enough to consider
$\beta_k+\delta\leq c<1$.
Put
\[
\mathcal I_{c,D}=(D^c,D]\cap\Z,
\qquad
\mathcal A_D=\mathcal A\cap\mathcal I_{c,D}.
\]
For $\mathbf r=(r_2,\ldots,r_k)\in\Z^{k-1}$, let
$\mathcal F_{k,c,D}(\mathbf r)$ be the event that there exist
pairwise distinct subsets $A_1,\ldots,A_k\subseteq\mathcal A_D$
satisfying
\begin{equation}\label{eq:exact-affine-defects}
\Sigma(A_s)-\Sigma(A_1)=r_s
\qquad(2\leq s\leq k).
\end{equation}
Put $\lambda_{\mathbf r}=(0,r_2,\ldots,r_k)\in\Z^k$.

Let $\mathcal R_{c,D}$ be the event that
\[
\left|
\#\bigl(\mathcal A_D\cap(D^\alpha,D^\beta]\bigr)
-(\beta-\alpha)\log D
\right|
\leq(\log D)^{3/4}
\qquad(c\leq\alpha\leq\beta\leq1).
\]
By \cite[Lemma~A.5]{FGK},
\begin{equation}\label{eq:affine-regularity-error}
\Pp(\mathcal R_{c,D}^{\complement})
\ll_k\exp\left\{-\frac14\sqrt{\log D}\right\},
\end{equation}
where $\mathcal R_{c,D}^{\complement}$ denotes the complement
of the event $\mathcal R_{c,D}$.

This estimate is uniform for $c\geq\beta_k+\delta$, since this
lower bound is a fixed positive number.

Fix a realization $B=\mathcal A_D$ in $\mathcal R_{c,D}$ and a
corresponding tuple $A_1,\ldots,A_k$ satisfying
\eqref{eq:exact-affine-defects}. For every $\omega\in\{0,1\}^k$,
define
\[
B_\omega
=
\left\{b\in B:
\bigl(1_{b\in A_1},\ldots,1_{b\in A_k}\bigr)=\omega
\right\},
\]
where $1_{b\in A_s}$ equals $1$ when $b\in A_s$ and $0$ otherwise.
In particular,
$B_{\mathbf0_k}=B\setminus(A_1\cup\cdots\cup A_k)$, and the
sets $B_\omega$ form a partition of the whole set $B$.
By the definitions,
\begin{equation}\label{eq:affine-venn-identity}
\sum_{\omega\in\{0,1\}^k}\omega\,\Sigma(B_\omega)
=
\bigl(\Sigma(A_1),\ldots,\Sigma(A_k)\bigr),
\end{equation}
and combining \eqref{eq:affine-venn-identity} with
\eqref{eq:exact-affine-defects} gives
\begin{equation}\label{eq:affine-quotient}
\sum_{\omega\in\{0,1\}^k}\omega\,\Sigma(B_\omega)
+\langle\one\rangle
=
\lambda_{\mathbf r}+\langle\one\rangle.
\end{equation}

Use the construction of \cite[Section~4.1]{FGK}.
Starting with $V_0=\langle\one\rangle$, at step $j$ choose
$\omega^j\notin V_{j-1}$ with $B_{\omega^j}\neq\varnothing$
for which $B_{\omega^j}$ is largest, and put
\[
V_j=V_{j-1}+\Span_{\Q}\{\omega^j\},
\qquad K_j=\max B_{\omega^j}.
\]
Stop when every $\omega$ with $B_\omega\neq\varnothing$ belongs
to the last space. The resulting flag
$\cV:V_0<\cdots<V_h$ is complete, with $1\leq h\leq k-1$,
and $\one,\omega^1,\ldots,\omega^h$ are linearly independent.
Its terminal space is non-degenerate, since for each $s\neq t$
the inequality $A_s\neq A_t$ supplies an element in a set
$B_\omega$ with $\omega_s\neq\omega_t$.

Set
\[
c_j=1+\frac{\lceil\log K_j-\log D\rceil}{\log D}
\quad(1\leq j\leq h),
\qquad c_{h+1}=c.
\]
Then
\[
1\geq c_1\geq\cdots\geq c_h>c,
\qquad
\ee^{-1}D^{c_j}<K_j\leq D^{c_j}.
\]
Delete $K_j$ from $B_{\omega^j}$ for each $j$, leaving sets
$B'_\omega$, and put
\[
B'=B\setminus\{K_1,\ldots,K_h\}.
\]
Define $\mu_j(\omega)$ as the proportion of
$B'\cap(D^{c_{j+1}},D^{c_j}]$ lying in $B'_\omega$.
If this intersection is empty, take $\mu_j$ to be concentrated
at $\mathbf0_k$.
For $j<h$, the choice of $\omega^{j+1}$ implies that
$B'_\omega\cap(D^{c_{j+1}},D^{c_j}]\neq\varnothing$
only if $\omega\in V_j$.
For $j=h$, the same conclusion follows from the stopping condition,
since every vector indexing a nonempty set $B_\omega$ belongs to $V_h$.
The same construction gives
$B'_\omega\cap(D^{c_1},D]=\varnothing$ when $\omega\notin V_0$.
Thus $(\cV,\mathbf c,\boldsymbol\mu)$ is a complete system, and
the partition $(B'_\omega)_\omega$ satisfies
\eqref{eq:partition-compatibility}.

Write
\[
\Lambda=(\cV,\mathbf c,\boldsymbol\mu,\omega^1,\ldots,\omega^h),
\qquad
\cL_\Lambda(B')=\cL_{\cV,\mathbf c,\boldsymbol\mu}(B').
\]
Since $B\in\mathcal R_{c,D}$, we have $\#B\leq2\log D$ for
large $D$. Each $c_j$ is drawn from a set of $O(\log D)$ values,
and every numerator and denominator in the empirical measures is
at most $2\log D$. The number of possible choices of $\Lambda$
is therefore $(\log D)^{O_k(1)}$, as in
\cite[Lemma~4.2]{FGK}. This bound is uniform in $c$ and
$\mathbf r$.

For fixed $\Lambda$, let $\mathscr F_\Lambda(\mathbf r)$ be the
family of sets $B\subseteq\mathcal I_{c,D}$ in
$\mathcal R_{c,D}$ having a tuple with these data and defects
$\mathbf r$. Choose one such tuple for each $B$ in this family.
The residual class
\[
L=\sum_{\omega\in\{0,1\}^k}\omega\,\Sigma(B'_\omega)
+\langle\one\rangle
\]
belongs to $\cL_\Lambda(B')$, and \eqref{eq:affine-quotient} gives
\begin{equation}\label{eq:affine-reconstruction}
\sum_{j=1}^hK_j\omega^j+\langle\one\rangle
=
\lambda_{\mathbf r}+\langle\one\rangle-L.
\end{equation}
For fixed $B'$, $L$, and $\mathbf r$,
\eqref{eq:affine-reconstruction} determines at
most one tuple $(K_1,\ldots,K_h)$.
Indeed, two such tuples would satisfy
\[
\sum_{j=1}^h(K_j-\widetilde K_j)\omega^j
\in\langle\one\rangle,
\]
and linear independence gives $K_j=\widetilde K_j$ for every $j$.
Thus the number of sets $B\in\mathscr F_\Lambda(\mathbf r)$
with a fixed residual set $B'$ is at most $\#\cL_\Lambda(B')$.
This bound counts classes, not the partitions representing them.

For every $T\subseteq\mathcal I_{c,D}$, independence gives
\[
\Pp(\mathcal A_D=T)
=
\prod_{i\in\mathcal I_{c,D}}\left(1-\frac1i\right)
\prod_{i\in T}\frac1{i-1}.
\]
For large $D$, all indices in $\mathcal I_{c,D}$ exceed $1$.
Since $B=B'\cup\{K_1,\ldots,K_h\}$ is a disjoint union,
\[
\begin{aligned}
\Pp(\mathcal A_D=B)
&=\Pp(\mathcal A_D=B')\prod_{j=1}^h\frac1{K_j-1}\\
&\leq(2\ee)^hD^{-\sum_{j=1}^hc_j}
\Pp(\mathcal A_D=B').
\end{aligned}
\]
The constant is uniform because $K_j>D^c\geq D^{\beta_k+\delta}$.

Deleting at most $h$ elements from $B$ gives
\[
\left|
\#\bigl(B'\cap(D^\alpha,D^\beta]\bigr)
-(\beta-\alpha)\log D
\right|
\leq(\log D)^{3/4}+h\leq2(\log D)^{3/4}.
\]
Let $\widetilde{\mathcal R}_{c,D}$ denote the family of sets
satisfying the last bound for all $c\leq\alpha\leq\beta\leq1$.
It follows that
\[
\begin{aligned}
\sum_{B\in\mathscr F_\Lambda(\mathbf r)}\Pp(\mathcal A_D=B)
&\ll_k D^{-\sum_{j=1}^hc_j}
\sum_{B'\in\widetilde{\mathcal R}_{c,D}}
\#\cL_\Lambda(B')\Pp(\mathcal A_D=B')\\
&\leq
\exp\bigl(O_k((\log D)^{3/4})\bigr)
D^{e_*(\Lambda)-\sum_{j=1}^hc_j},
\end{aligned}
\]
where
\[
e_*(\Lambda)=\min_{\cV'\leq\cV}
\mathrm e(\cV',\mathbf c,\boldsymbol\mu).
\]
Here the last inequality uses Lemma~\ref{lem:residual-sum-bound}
and $\sum_{B'\in\widetilde{\mathcal R}_{c,D}}
\Pp(\mathcal A_D=B')\leq1$.
In particular, no conditional distribution for the residual set
is being assumed.

By Corollary~\ref{cor:uniform-gap}, there is a fixed
$\eta=\eta(k,\delta)>0$ such that
\[
e_*(\Lambda)
\leq\mathrm e(\cV,\mathbf c,\boldsymbol\mu)-\eta
=\sum_{j=1}^hc_j-\eta.
\]
Summing over the $(\log D)^{O_k(1)}$ possible data gives
\begin{equation}\label{eq:exact-affine-fiber-bound}
\Pp\bigl(\mathcal F_{k,c,D}(\mathbf r)\cap\mathcal R_{c,D}\bigr)
\leq D^{-\eta+o(1)},
\end{equation}
where $o(1)=O_{k,\delta}((\log D)^{-1/4})$ is uniform in $c$ and
$\mathbf r$.

For $\tau\in\R^{k-1}$ and $W\geq0$, put
\[
\mathcal D(\tau,W)
=
\left\{\mathbf r=(r_2,\ldots,r_k)\in\Z^{k-1}:
|r_s-\tau_s|\leq W\ (2\leq s\leq k)\right\}.
\]
An interval of length $2W$ contains at most $2W+1$ integers, so
\begin{equation}\label{eq:number-affine-defects}
\#\mathcal D(\tau,W)\leq(2W+1)^{k-1}.
\end{equation}
By the definitions,
\[
\mathcal N_{k,c,D}(\tau,W)
=
\bigcup_{\mathbf r\in\mathcal D(\tau,W)}
\mathcal F_{k,c,D}(\mathbf r).
\]
Consequently, \eqref{eq:affine-regularity-error},
\eqref{eq:exact-affine-fiber-bound}, and
\eqref{eq:number-affine-defects} give
\[
\begin{aligned}
\Pp\bigl(\mathcal N_{k,c,D}(\tau,W)\bigr)
&\leq\Pp(\mathcal R_{c,D}^{\complement})
+\sum_{\mathbf r\in\mathcal D(\tau,W)}
\Pp\bigl(\mathcal F_{k,c,D}(\mathbf r)\cap\mathcal R_{c,D}\bigr)\\
&\leq(2W+1)^{k-1}D^{-\eta+o(1)}
+O_k\left(\exp\left\{-\frac14\sqrt{\log D}\right\}\right)\\
&=D^{-\eta+o(1)}
+O_k\left(\exp\left\{-\frac14\sqrt{\log D}\right\}\right),
\end{aligned}
\]
since $1+W=D^{o(1)}$ and $k$ is fixed.
More precisely, the $o(1)$ term in the last line may be taken to be
\[
O_{k,\delta}\!\left(
(\log D)^{-1/4}
+
\frac{\log(2+W)}{\log D}
\right),
\]
where the implied constant is independent of $c$, $\tau$, and $W$.
\end{proof}

\section{Arithmetic setup and affine reduction}
\label{sec:arithmetic-setup}

Fix an integer $k\geq2$. Unless otherwise stated, all implied
constants may depend on $k$, on the fixed exponent $a$, and on fixed
auxiliary parameters. We prove the upper bound in
Theorem~\ref{thm:main} for every fixed exponent satisfying
\begin{equation}\label{eq:supercritical-a}
a>\frac{\beta_k}{1-\beta_k}.
\end{equation}

By the definition of $\beta_k$, we have $\beta_k\leq\beta_2$.
Moreover, \cite[Section~3.4]{FGK} gives $\beta_2=1-1/\log3$, and hence
$\frac{\beta_2}{1-\beta_2}=\log3-1<\frac12.$
Suppose that $a\geq1/2$. Choose $a_0$ such that
$\frac{\beta_k}{1-\beta_k}<a_0<\frac12.$
For all sufficiently large $n$, $(\log n)^{-a}\leq(\log n)^{-a_0}.$
Hence, if divisors $d_1<\cdots<d_k\mid n$ satisfy
$$
d_k\leq d_1\bigl(1+(\log n)^{-a}\bigr),
$$
then they also satisfy 
$$
d_k\leq d_1\bigl(1+(\log n)^{-a_0}\bigr).
$$
Consequently, the density-zero assertion for
$a_0$ implies the corresponding assertion for $a$. It is therefore
enough to assume
\begin{equation}\label{eq:a-less-half}
0<a<\frac12.
\end{equation}

Throughout the arithmetic arguments, $X$ is a sufficiently large positive integer. We first work on $(X/2,X]$. Define
\begin{equation}\label{eq:bin-parameters}
K=(\log X)^a,
\qquad
i_1=\left\lfloor K(\log\log X)^3\right\rfloor,
\qquad
D=\left\lfloor\frac{K\log X}{2\log\log\log X}\right\rfloor.
\end{equation}
As $X\to\infty$, $i_1=(\log X)^{a+o(1)}$ and $D=(\log X)^{1+a+o(1)}.$
In particular, $2\leq i_1<D$ for all sufficiently large $X$.

For every integer $i$ with $i_1\leq i\leq D$, define
$\cP_i=\bigl(\ee^{i/K},\ee^{(i+1)/K}\bigr].$
Also put $y=\ee^{i_1/K},z=\ee^{(D+1)/K}.$
The intervals $\cP_i$ are pairwise disjoint and consecutive, and therefore $\bigcup_{i=i_1}^{D}\cP_i=(y,z].$ Moreover, 
\begin{equation}\label{equ-1}
\frac{i_1}{K}=(\log\log X)^3+O(K^{-1})    
\end{equation}
and 
$$
\frac{D+1}{K}=\frac{\log X}{2\log\log\log X}+O(K^{-1}).
$$
Hence
\begin{equation}\label{eq:y-z-asymp}
y=\exp\bigl((\log\log X)^3+o(1)\bigr),\qquad
z=X^{1/(2\log\log\log X)+o(1)}.
\end{equation}

A prime $p$ is called \emph{medium} if $p\in(y,z]$, and
\emph{external} otherwise. For a prime $p$ and a positive integer $d$,
put
\[
v_p(d)=\max\{\nu\in\Z_{\geq0}:p^\nu\mid d\}.
\]
Thus $p^{v_p(d)}\mid d$ and $p^{v_p(d)+1}\nmid d$.
For every positive integer $d$, define
\[
m(d)=\prod_{\substack{p\mid d\\p\ \mathrm{medium}}}p^{v_p(d)},
\qquad
r(d)=\prod_{\substack{p\mid d\\p\ \mathrm{external}}}p^{v_p(d)}.
\]
Unique prime factorization gives
\begin{equation}\label{eq:external-medium-decomp}
d=r(d)m(d),
\end{equation}
where $\gcd(r(d),m(d))=1$, and this decomposition is unique.
All products and sums indexed by $p$ below are over primes.
For a medium prime $p\in\cP_i$, put $\iota(p)=i$.
For a positive integer $n$ whose prime divisors all lie in $(y,z]$,
define
\[
S_K(n)=\sum_{p\mid n}v_p(n)\iota(p).
\]
For every positive integer $d$, write
$\Omega(d)=\sum_{p\mid d}v_p(d)$.
The empty sums are zero and the empty products are one.

For a positive integer $n$, define its \emph{occupied-bin set} by
\[
I(n)
=\left\{
i\in[i_1,D]\cap\Z:
\text{ there exists }p\in\cP_i\text{ with }p\mid n
\right\}.
\]

For every $i_1\leq i\leq D$, define
\begin{equation}\label{eq:R-i-def}
R_i=\sum_{p\in\cP_i}\frac1p.
\end{equation}
For an absolute constant $c_0>0$, we have
\[
\sum_{u<p\leq v}\frac1p
=\log\frac{\log v}{\log u}
+O\bigl(\exp(-c_0\sqrt{\log u})\bigr)
\qquad(2\leq u<v),
\]
by the quantitative prime number theorem in
\cite[Chapter~6]{MontgomeryVaughan} and partial summation.
Applying this with $u=\ee^{i/K}$ and $v=\ee^{(i+1)/K}$ yields
\[
R_i=\log\left(1+\frac1i\right)
+O\bigl(\exp(-c_0\sqrt{i/K})\bigr).
\]
By \eqref{equ-1} and
$\log D=(1+a+o(1))\log\log X$, we have
\[
\sup_{i_1\leq i\leq D}
 i^2\exp(-c_0\sqrt{i/K})
\leq D^2\exp(-c_0\sqrt{i_1/K})=o(1).
\]
Therefore, uniformly for $i_1\leq i\leq D$,
\begin{equation}\label{eq:R-i-asymp}
R_i=\frac1i+O(i^{-2}).
\end{equation}

\subsection{The affine medium-part equation}

The next proposition converts closeness of divisors into an affine
relation among the weighted bin sums of their medium parts. The
translation in this affine relation is determined entirely by the
external parts.

\begin{proposition}\label{prop:AMR}
Let $d_1,\ldots,d_k$ be positive integers, and write
$d_s=r_sm_s$, where $r_s=r(d_s)$ and $m_s=m(d_s)$ for
$1\leq s\leq k$. Suppose that, for some $\Delta\geq0$,
\begin{equation*}
\max_{1\leq r,s\leq k}
|\log d_r-\log d_s|
\leq\Delta.
\end{equation*}
Then, for every $2\leq s\leq k$,
\begin{equation}\label{eq:AMR}
\left|
S_K(m_s)-S_K(m_1)
+
K\log\frac{r_s}{r_1}
\right|
\leq
K\Delta+\Omega(m_s)+\Omega(m_1).
\end{equation}
\end{proposition}

\begin{proof}
Let $p$ be a medium prime. Since the prime bins $\cP_i$ are pairwise
disjoint, there is a unique index $i$ such that $p\in\cP_i$, and by
definition $\iota(p)=i$. From $\cP_i=\bigl(\ee^{i/K},\ee^{(i+1)/K}\bigr],$
we have 
$$
\ee^{i/K}<p\leq\ee^{(i+1)/K}.
$$
Taking logarithms and multiplying by $K>0$ gives
$$
i<K\log p\leq i+1.
$$
Since $\iota(p)=i$, it follows that
\[
0<K\log p-\iota(p)\leq1.
\]
We therefore define $$\rho(p)=K\log p-\iota(p).$$
Then $0<\rho(p)\leq1,$ and $$K\log p=\iota(p)+\rho(p).$$

Let $m$ be a positive integer whose prime divisors all lie in $(y,z]$.
By unique prime factorization, we have $m=\prod_{p\mid m}p^{v_p(m)}.$
Taking logarithms, we obtain
$$\log m=\sum_{p\mid m}v_p(m)\log p.$$ Multiplying both sides by $K$ gives
\begin{align}
K\log m
&=
\sum_{p\mid m}v_p(m)\bigl(\iota(p)+\rho(p)\bigr) \notag \\
&=\sum_{p\mid m}v_p(m)\iota(p)+\sum_{p\mid m}v_p(m)\rho(p) \notag \\
&=S_K(m)+\rho_K(m),
\label{equ-2}
\end{align}
where $\rho_K(m)=\sum_{p\mid m}v_p(m)\rho(p).$
Since $0<\rho(p)\leq1$, it follows that
\begin{equation*}
0\leq\rho_K(m)
\leq
\sum_{p\mid m}v_p(m)
=
\Omega(m).
\end{equation*}

Fix $s\in\{2,\ldots,k\}$. Combining the preceding identities, we have
\begin{align*}
S_K(m_s)-S_K(m_1)
+
K\log\frac{r_s}{r_1}
&=
K\log\frac{m_s}{m_1}
-
\bigl(\rho_K(m_s)-\rho_K(m_1)\bigr)
+
K\log\frac{r_s}{r_1}\\
&=
K\left(
\log\frac{d_s}{d_1}
-
\log\frac{r_s}{r_1}
\right)
-
\bigl(\rho_K(m_s)-\rho_K(m_1)\bigr)
+
K\log\frac{r_s}{r_1}\\
&=
K\log\frac{d_s}{d_1}
-
\bigl(\rho_K(m_s)-\rho_K(m_1)\bigr).
\end{align*}
Here, the first equality follows from \eqref{equ-2}. The second equality follows from $m_s=d_s/r_s$ and $m_1=d_1/r_1$.

Taking absolute values and applying the triangle inequality gives
\begin{align*}
\left|
S_K(m_s)-S_K(m_1)
+
K\log\frac{r_s}{r_1}
\right|
&\leq
K\left|\log\frac{d_s}{d_1}\right|
+
\rho_K(m_s)+\rho_K(m_1)\\
&\leq
K\Delta+\Omega(m_s)+\Omega(m_1).
\end{align*}
This proves \eqref{eq:AMR}.
\end{proof}

Let $X$ be a sufficiently large positive integer, and let
$\mathbf n$ be chosen uniformly from $\{1,\ldots,X\}$. Thus, $\Pp(\mathbf n=n)=\frac1X$ for $1\leq n\leq X.$
For every positive integer $n$,
\[
\Omega(n)=\sum_p v_p(n)=\sum_p\sum_{\substack{j\geq1\\p^j\mid n}}1.
\]
Indeed, if $v_p(n)=\nu$, then $p^j\mid n$ holds precisely for $1\leq j\leq\nu$, so the prime $p$ contributes $\nu$ terms to the last sum.

By linearity of expectation, 
$$
\E\Omega(\mathbf n)=\sum_{p^j\leq X}
\Pp\bigl(p^j\mid\mathbf n\bigr).
$$
Since exactly $\lfloor X/p^j\rfloor$ integers in $\{1,\ldots,X\}$
are divisible by $p^j$, it follows that
$$
\Pp\bigl(p^j\mid\mathbf n\bigr)
=\frac{\lfloor X/p^j\rfloor}{X}
\leq \frac1{p^j}.
$$
Consequently, $\E\Omega(\mathbf n) \leq \sum_{p^j\leq X}\frac1{p^j}.$
By Mertens' theorem, we have
\[
\sum_{p\leq X}\frac1p
=
\log\log X+O(1),
\]
while $\sum_p\sum_{j\geq2}\frac1{p^j}=\sum_p\frac1{p(p-1)}=O(1).$
It follows that 
$$
\E\Omega(\mathbf n)
\leq \log\log X+O(1).
$$
Markov's inequality therefore gives
\begin{equation}\label{eq:Omega-good}
\begin{aligned}
\Pp\bigl(
\Omega(\mathbf n)>(\log\log X)^2
\bigr) \leq \frac{\E\Omega(\mathbf n)}{(\log\log X)^2}
\leq \frac{\log\log X+O(1)}{(\log\log X)^2}
\ll \frac1{\log\log X} =o(1).
\end{aligned}
\end{equation}
Thus, all but $o(X)$ integers $n\leq X$ satisfy
$\Omega(n)\leq(\log\log X)^2$.

Fix such an integer $n\in(X/2,X]$, and suppose that
$d_1<\cdots<d_k\mid n$ satisfy $d_k
\leq d_1\bigl(1+(\log n)^{-a}\bigr).$
Set 
$$
\Delta=\max_{1\leq r,s\leq k}|\log d_r-\log d_s|.
$$
Since the divisors are ordered, it follows that
\[
\Delta=\log\frac{d_k}{d_1} \leq
\log\bigl(1+(\log n)^{-a}\bigr)
\ll(\log X)^{-a}.
\]
Since $K=(\log X)^a$ in \eqref{eq:bin-parameters}, this gives $K\Delta\ll1.$

Write $d_s=r_sm_s$ according to
\eqref{eq:external-medium-decomp}. Since $m_s\mid d_s\mid n$,
\[
\Omega(m_s)
\leq
\Omega(n)
\leq
(\log\log X)^2
\qquad
(1\leq s\leq k).
\]
Define the affine target determined by the external parts by
\begin{equation}\label{eq:external-target}
\tau(r_1,\ldots,r_k)=\left(
-K\log\frac{r_2}{r_1}, \ldots,-K\log\frac{r_k}{r_1}\right) \in\R^{k-1}.
\end{equation}
Choose a sufficiently large constant $C_{a,k}>0$ and set
$W=
C_{a,k}\bigl(1+(\log\log X)^2\bigr).$
Proposition~\ref{prop:AMR} then gives, for every $2\leq s\leq k$,
\[
\left|
S_K(m_s)-S_K(m_1)
-
\tau_s(r_1,\ldots,r_k)
\right|
\leq W.
\]
Equivalently,
\[
\bigl(
S_K(m_2)-S_K(m_1),
\ldots,
S_K(m_k)-S_K(m_1)
\bigr)
\in
\tau(r_1,\ldots,r_k)+[-W,W]^{k-1}.
\]
Since $D=(\log X)^{1+a+o(1)},$ we have $\log W=o(\log D).$
Therefore, $W=D^{o(1)}.$

We finally relate the lower endpoint $i_1$ of the medium-bin interval
to the parameter $D$. For sufficiently large $X$, define
\begin{equation}\label{eq:cX-definitions}
c_X=\frac{\log i_1}{\log D},
\qquad
\widehat c_X=\frac{\log(i_1-1)}{\log D}.
\end{equation}

Thus, $D^{c_X}=i_1$ and $D^{\widehat c_X}=i_1-1.$
From \eqref{eq:bin-parameters}, $i_1=(\log X)^{a+o(1)}$ and $D=(\log X)^{1+a+o(1)}.$
Consequently,
\[
c_X
=
\frac{\log i_1}{\log D}
=
\frac{a}{1+a}+o(1).
\]
Moreover, since $i_1\to\infty$,
\[
\log(i_1-1)
=
\log i_1+\log\left(1-\frac1{i_1}\right)
=
\log i_1+o(1),
\]
and therefore $\widehat c_X=\frac{a}{1+a}+o(1).$

Since $\beta_k<1$, the assumption
\eqref{eq:supercritical-a} is equivalent to
\begin{equation}\label{eq:conversion}
\frac{a}{1+a}>\beta_k.
\end{equation}
Define $\delta=\frac14\left(\frac{a}{1+a}-\beta_k\right)>0.$
Since both $c_X$ and $\widehat c_X$ converge to
$a/(1+a)=\beta_k+4\delta$, for all sufficiently large $X$, we have
\begin{equation}\label{eq:cX-supercritical}
c_X,\widehat c_X
\geq
\beta_k+2\delta.
\end{equation}

Finally, $D^{\widehat c_X}=i_1-1$, so
$(D^{\widehat c_X},D]\cap\Z=[i_1,D]\cap\Z.$

\subsection{Squarefree medium parts and bin occupancy}

We now show that, outside a negligible set of integers, the medium part
$m(n)$ is squarefree and no two medium prime divisors of $n$ lie in the same
bin.

\begin{lemma}\label{lem:MOP}
For all but $o(X)$ integers $n\leq X$, the following properties hold:
\begin{enumerate}[label=\textup{(\roman*)},leftmargin=2.2em]
\item for every $i\in[i_1,D]\cap\Z$, at most one prime
$p\in\cP_i$ divides $n$;
\item if $p$ is a medium prime dividing $n$, then $p^2\nmid n$.
\end{enumerate}
\end{lemma}

\begin{proof}
Let $\mathbf n$ be chosen uniformly from $\{1,\ldots,X\}$. We first estimate the probability that property \textup{(i)} fails.
For a fixed $i\in[i_1,D]\cap\Z$, we have
\begin{align*}
\Pp\bigl(
\exists\,p,q\in\cP_i,\ p<q,\ pq\mid\mathbf n
\bigr)
&\leq \sum_{\substack{p,q\in\cP_i\\p<q}}
\Pp(pq\mid\mathbf n)=\sum_{\substack{p,q\in\cP_i\\p<q}}
\frac{\lfloor X/(pq)\rfloor}{X}\\
&\leq
\sum_{\substack{p,q\in\cP_i\\p<q}}
\frac1{pq}\\
&\leq
\frac12
\left(
\sum_{p\in\cP_i}\frac1p
\right)^2
=
\frac12R_i^2.
\end{align*}
The equality in the first line holds
because there are exactly $\lfloor X/(pq)\rfloor$ multiples of $pq$
in $\{1,\ldots,X\}$. The next inequality follows from
$\lfloor X/(pq)\rfloor\leq X/(pq)$. Finally,
\[
\sum_{\substack{p,q\in\cP_i\\p<q}}\frac1{pq}
=
\frac12
\left[
\left(\sum_{p\in\cP_i}\frac1p\right)^2
-
\sum_{p\in\cP_i}\frac1{p^2}
\right]
\leq
\frac12
\left(\sum_{p\in\cP_i}\frac1p\right)^2,
\]
and the last equality uses the definition
$R_i=\sum_{p\in\cP_i}p^{-1}$.

Summing over all medium bins and using
$R_i\ll i^{-1}$ from \eqref{eq:R-i-asymp}, we obtain
\begin{align*}
\Pp\bigl(\text{\textup{(i)} fails}\bigr)
&\leq
\frac12\sum_{i=i_1}^{D}R_i^2\ll
\sum_{i=i_1}^{\infty}\frac1{i^2}\ll \frac1{i_1}=o(1),
\end{align*}
because $i_1\to\infty$ as $X\to\infty$.

We next consider property \textup{(ii)}. If it fails, then
$p^2\mid\mathbf n$ for some medium prime $p$. Hence, again by the
union bound,
\begin{align*}
\Pp\bigl(\text{\textup{(ii)} fails}\bigr)
&\leq \sum_{p\ \mathrm{medium}}
\Pp(p^2\mid\mathbf n)
\leq \sum_{p\ \mathrm{medium}}\frac1{p^2}\\
&\leq \sum_{m>y}\frac1{m^2} \ll
\frac1y
=o(1),
\end{align*}
where $y\to\infty$ by \eqref{eq:y-z-asymp}.
Thus the probability that either property fails is $o(1)$.

Let $\mathcal M_X$ denote the set of integers $n\leq X$ for which at
least one of \textup{(i)} and \textup{(ii)} fails. Since $\mathbf n$
is chosen uniformly from $\{1,\ldots,X\}$, we have
$$
\Pp(\mathbf n\in\mathcal M_X)=\frac{|\# \mathcal M_X|}{X}.
$$
The preceding estimates show that
$\Pp(\mathbf n\in\mathcal M_X)=o(1)$. Therefore,
\[
|\# \mathcal M_X|
=
X\,\Pp(\mathbf n\in\mathcal M_X)
=
o(X).
\]
Hence both \textup{(i)} and \textup{(ii)} hold for all but $o(X)$
integers $n\leq X$.
\end{proof}

Assume that $n$ satisfies the conclusions of
Lemma~\ref{lem:MOP}. For each $i\in I(n)$, let $p_i$ denote the unique
prime in $\cP_i$ dividing $n$. By Lemma~\ref{lem:MOP}\textup{(ii)},
each $p_i$ occurs in the prime factorization of $n$ with exponent one.

For every subset $A\subseteq I(n)$, define $m_A=\prod_{i\in A}p_i,$
where the empty product is understood to be $1$.  Conversely, every divisor $m\mid n$ supported on medium primes is of this form for a unique subset
\begin{equation}\label{eq:A-m-definition}
A(m)=\{i\in I(n):p_i\mid m\}.
\end{equation}

Since $m=\prod_{i\in A(m)}p_i$, it follows that each prime $p_i$ with
$i\in A(m)$ occurs in $m$ with exponent one. Using
$\iota(p_i)=i$, we obtain
\begin{equation}\label{eq:medium-subset-sum}
S_K(m)
=
\sum_{i\in A(m)}\iota(p_i)
=
\sum_{i\in A(m)}i
=
\Sigma(A(m)).
\end{equation}

\subsection{Repeated medium parts}

A close-divisor witness with fewer than $k$ distinct medium parts
necessarily contains two divisors with the same medium part. We show
that this can occur only on a density-zero set, using only the
sparsity of pairs of close external divisors.

\begin{lemma}\label{lem:EPS}
Let $\delta_X=(\log X)^{-a}.$
For each fixed constant $C>0$, let $\mathcal B_X(C)$ denote the set of
integers $n\leq X$ having two distinct divisors $r<t$ such that both
$r$ and $t$ are supported entirely on external primes and
$t\leq(1+C\delta_X)r$. Then
\[
\#\mathcal B_X(C)=o(X)
\qquad\text{as }X\to\infty.
\]
\end{lemma}

\begin{proof}
Let
\[
\mathcal S_X=
\left\{
m\in\N:
\text{every prime divisor }p\text{ of }m
\text{ satisfies }p\leq y\text{ or }z<p\leq X
\right\}.
\]
Thus $\mathcal S_X$ is the set of positive integers supported on
external primes not exceeding $X$. In particular, $1\in\mathcal S_X$.
Define $Z_X=\sum_{m\in\mathcal S_X}\frac1m.$
Let
\[
\mathcal Q_X=
\left\{
p\leq X:
p\text{ is an external prime}
\right\}.
\]
By the definition of $\mathcal S_X$, every $m\in\mathcal S_X$ has a
unique representation $m=\prod_{p\in\mathcal Q_X}p^{\nu_p}$ with $\nu_p\in\Z_{\geq0}.$
Therefore, by unique prime factorization and the geometric-series identity, we have 
\begin{align*}
Z_X
&=\sum_{(\nu_p)_{p\in\mathcal Q_X}}
\frac{1}{\displaystyle\prod_{p\in\mathcal Q_X}p^{\nu_p}}
=\sum_{(\nu_p)_{p\in\mathcal Q_X}}
\prod_{p\in\mathcal Q_X}\frac1{p^{\nu_p}}\\
&=\prod_{p\in\mathcal Q_X}\left(\sum_{\nu=0}^{\infty}\frac1{p^\nu}\right)
=\prod_{p\in\mathcal Q_X}
\left(1-\frac1p\right)^{-1}\\
&=\prod_{\substack{p\leq X\\p\ \mathrm{external}}}
\left(1-\frac1p\right)^{-1}.
\end{align*}

The external primes not exceeding $X$ are precisely the primes
satisfying $p\leq y$ or $z<p\leq X$. Hence
\[
Z_X
=
\prod_{p\leq y}
\left(1-\frac1p\right)^{-1}
\prod_{z<p\leq X}
\left(1-\frac1p\right)^{-1}.
\]

By Mertens' product theorem, as $T\to\infty$, we have
$$
\prod_{p\leq T}
\left(1-\frac1p\right)^{-1}
=\ee^\gamma\log T\bigl(1+o(1)\bigr),
$$
where $\gamma$ is the Euler--Mascheroni constant. Hence
$\prod_{p\leq y}
\left(1-\frac1p\right)^{-1} \ll \log y.$
Moreover,
\begin{align*}
\prod_{z<p\leq X}
\left(1-\frac1p\right)^{-1}
=\frac{\displaystyle\prod_{p\leq X}\left(1-\frac1p\right)^{-1}}{\displaystyle\prod_{p\leq z}\left(1-\frac1p\right)^{-1}}
\ll\frac{\log X}{\log z}.
\end{align*}
Therefore, $Z_X \ll (\log y)\frac{\log X}{\log z}.$

By the definitions of $y$, $z$, $i_1$, and $D$ in \eqref{eq:bin-parameters} and \eqref{eq:y-z-asymp}, we have $\log y \ll(\log\log X)^3,$
while
\[
\log z
=
\frac{D+1}{K}
=
\frac{\log X}{2\log\log\log X}
+
O(K^{-1}).
\]
In particular, for all sufficiently large $X$,
$\log z \gg \frac{\log X}{\log\log\log X},$
and hence 
$$
\frac{\log X}{\log z} \ll \log\log\log X.
$$
Combining these estimates gives
\[
Z_X
\ll
(\log\log X)^3\log\log\log X.
\]

Let $\mathbf n$ be chosen uniformly from $\{1,\ldots,X\}$. Since 
$$
\Pp\bigl(\mathbf n\in\mathcal B_X(C)\bigr)=\frac{|\# \mathcal B_X(C)|}{X},
$$
it is enough to prove that
$\Pp(\mathbf n\in\mathcal B_X(C))=o(1)$.
Suppose that $\mathbf n\in\mathcal B_X(C)$, and choose corresponding external divisors $r<t$. Put
\[
g=\gcd(r,t),\qquad r=gu,\qquad t=gv.
\]
Thus $g,u,v\in\mathcal S_X$. Moreover,
$(u,v)=1$ and $u<v\leq(1+C\delta_X)u$.
Since $r,t\mid\mathbf n$, we have $guv\mid\mathbf n$.
Hence
\[
\Pp(guv\mid\mathbf n)
=
\frac{\lfloor X/(guv)\rfloor}{X}
\leq
\frac1{guv}.
\]

By the union bound,
\begin{align*}
\Pp\bigl(\mathbf n\in\mathcal B_X(C)\bigr)
&\leq
\sum_{\substack{
g,u,v\in\mathcal S_X\\
(u,v)=1,\ u<v\leq(1+C\delta_X)u
}}\frac1{guv}\\
&\leq \sum_{g\in\mathcal S_X}\frac1g \sum_{u\in\mathcal S_X}\frac1u \sum_{\substack{
v\in\mathcal S_X\\u<v\leq(1+C\delta_X)u}}
\frac1v.
\end{align*}

For each fixed $u\in\mathcal S_X$, we have
$$
\sum_{\substack{
v\in\mathcal S_X\\u<v\leq(1+C\delta_X)u
}}\frac1v \leq \frac{\lfloor C\delta_Xu\rfloor}{u}
\leq C\delta_X.
$$
Therefore,
\begin{align}
\Pp\bigl(\mathbf n\in\mathcal B_X(C)\bigr)
& \leq C\delta_X
\left(\sum_{m\in\mathcal S_X}\frac1m\right)^2
=C\delta_XZ_X^2 \notag\\
& \ll_C
(\log X)^{-a}
(\log\log X)^6
(\log\log\log X)^2
=o(1) \notag.
\end{align}
Since $\mathbf n$ is uniformly distributed on $\{1,\ldots,X\}$, it follows that
\[
\frac{|\#\mathcal B_X(C)|}{X}
=
\Pp\bigl(\mathbf n\in\mathcal B_X(C)\bigr)
=
o(1),
\]
and hence $|\#\mathcal B_X(C)|=o(X)$.
\end{proof}

\begin{corollary}\label{cor:DMC}
Fix $C>0$. For all but $o(X)$ integers $n\leq X$, every collection of divisors $d_1<\cdots<d_k\mid n$ satisfying
$d_k\leq(1+C\delta_X)d_1$
has pairwise distinct medium parts. More precisely, write
$d_s=r_sm_s(1\leq s\leq k),$
where $r_s$ is supported on external primes and $m_s$ is supported on
medium primes. Then $m_1,\ldots,m_k$ are pairwise distinct.
\end{corollary}

\begin{proof}
Let $n\leq X$ satisfy $n\notin\mathcal B_X(C)$. Suppose, to the
contrary, that $m_i=m_j$ for some $1\leq i<j\leq k$. Since $d_i<d_j$, we have $r_i<r_j$.
Moreover,
\[
\frac{r_j}{r_i}
=
\frac{d_j/m_j}{d_i/m_i}
=
\frac{d_j}{d_i}
\leq
\frac{d_k}{d_1}
\leq
1+C\delta_X.
\]
Thus $r_i$ and $r_j$ are two distinct divisors of $n$, both supported
on external primes, such that $r_j\leq(1+C\delta_X)r_i.$
By the definition of $\mathcal B_X(C)$, this implies
$n\in\mathcal B_X(C)$, a contradiction. Hence
$m_1,\ldots,m_k$ are pairwise distinct. Since
$|\mathcal B_X(C)|=o(X)$ by Lemma~\ref{lem:EPS}, the result follows.
\end{proof}

\section{The range \texorpdfstring{$Q\leq X^{1-\eps}$}{Q at most X to the power 1-epsilon}}\label{sec:nonsat}

Let $n\in(X/2,X]$ admit divisors $d_1<\cdots<d_k\mid n$
satisfying $d_k\leq d_1\bigl(1+(\log n)^{-a}\bigr).$
For $1\leq s\leq k$, write $d_s=r_sm_s,$
where $r_s$ is supported on external primes and $m_s$ is supported on medium primes.
In this section, we consider the case in which
$m_1,\ldots,m_k$ are pairwise distinct.

Let $Q$ be the least positive integer divisible by each of
$r_1,\ldots,r_k$; that is,
\begin{equation}\label{eq:Q-def}
Q=
\min\left\{
t \in\N:
r_s\mid t
\text{ for every }1\leq s\leq k
\right\}.
\end{equation}
Since $r_s\mid d_s\mid n$ for every $1\leq s\leq k$, we have
$Q\mid n$. Moreover, every prime divisor of $Q$ is external.

For fixed external parts, the affine reduction allows us to apply
Theorem~\ref{thm:affine}. The following lemma is obtained by
summing the resulting estimates over the external tuples satisfying $Q\leq X^{1-\eps}$.

\begin{lemma}\label{lem:nonsat}
Let $\eps>0$ be a sufficiently small constant depending only on $k$.
The number of integers $n\in(X/2,X]$ for which there exist divisors
$d_1<\cdots<d_k\mid n$ satisfying
$d_k\leq d_1\bigl(1+(\log n)^{-a}\bigr),$
whose medium parts $m_1,\ldots,m_k$ are pairwise distinct and for which
the integer $Q$ defined in \eqref{eq:Q-def} satisfies $Q\leq X^{1-\eps},$
is $o(X)$ as $X\to\infty$.
\end{lemma}

\begin{proof}
\setcounter{claim}{0}
By Lemma~\ref{lem:MOP} and \eqref{eq:Omega-good}, it suffices to
count those integers in the statement that satisfy both conclusions
of Lemma~\ref{lem:MOP} and
\begin{equation}\label{eq:nonsat-good-integers}
\Omega(n)\leq(\log\log X)^2,
\end{equation}
since the remaining integers number $o(X)$.

Fix an ordered tuple $(r_1,\ldots,r_k)$ of external parts, and put
$Q=\lcm(r_1,\ldots,r_k)$. Assume that $Q\leq X^{1-\eps}$.
Every integer $n$ under consideration is uniquely expressible as
$n=Q\ell$, where
$\ell\leq N=\lfloor X/Q\rfloor$. For sufficiently large $X$,
\begin{equation}\label{eq:N}
N\geq\frac XQ-1\geq X^\eps-1\geq\frac12X^\eps.
\end{equation}
Since every prime divisor of $Q$ is external, it follows that 
\begin{equation}\label{eq:nonsat-occupancy-equality}
I(n)=I(\ell).
\end{equation}

Put
\[
\tau=(\tau_2,\ldots,\tau_k)
=
\left(-K\log\frac{r_2}{r_1},\ldots,-K\log\frac{r_k}{r_1}\right),
\qquad
W=C_{a,k}\bigl(1+(\log\log X)^2\bigr),
\]
where $C_{a,k}>0$ is sufficiently large. In particular,
$W=D^{o(1)}$.
Let $\mathfrak G_\tau$ be the family of sets
$B\subseteq[i_1,D]\cap\Z$ containing pairwise distinct subsets
$A_1,\ldots,A_k$ such that
\begin{equation}\label{eq:nonsat-affine-condition}
\left|\Sigma(A_s)-\Sigma(A_1)-\tau_s\right|
\leq W
\qquad(2\leq s\leq k).
\end{equation}

\begin{claim}\label{clm:nonsat-affine}
For every $n=Q\ell$ under consideration,
$I(\ell)\in\mathfrak G_\tau$.
\end{claim}

\begin{proof}
For $1\leq s\leq k$, let $A_s=A(m_s)$ be the subset corresponding
to $m_s$ in \eqref{eq:A-m-definition}.
Lemma~\ref{lem:MOP} and \eqref{eq:medium-subset-sum} give
\[
A_s\subseteq I(n)=I(\ell),
\qquad
S_K(m_s)=\Sigma(A_s).
\]
Since $m_1,\ldots,m_k$ are pairwise distinct, it follows that $A_1,\ldots,A_k$ are  pairwise distinct.

The closeness condition implies
\[
\Delta=\max_{1\leq r,s\leq k}|\log d_r-\log d_s|
=\log\frac{d_k}{d_1}
\leq\log\bigl(1+(\log n)^{-a}\bigr)
\ll(\log X)^{-a}.
\]
Thus $K\Delta\ll1$, since $K=(\log X)^a$.
Moreover, \eqref{eq:nonsat-good-integers} gives
$\Omega(m_s)\leq\Omega(n)\leq(\log\log X)^2$.
By Proposition~\ref{prop:AMR},
\[
\left|
S_K(m_s)-S_K(m_1)+K\log\frac{r_s}{r_1}
\right|
\leq K\Delta+\Omega(m_s)+\Omega(m_1)
\leq W.
\]
Substituting $S_K(m_s)=\Sigma(A_s)$ and
$\tau_s=-K\log(r_s/r_1)$ proves
\eqref{eq:nonsat-affine-condition}.
\end{proof}

For a sufficiently small constant $c_\eps>0$, define
\begin{equation}\label{eq:nonsat-error}
E_X=
\sum_{p>y}\frac1{p^2}
+\frac1{i_1}
+\exp\left\{
-c_\eps(\log\log\log X)\log\log\log\log X
\right\}.
\end{equation}

\begin{claim}\label{clm:nonsat-fixed-tuple}
There exists $\eta=\eta(k,\delta)>0$ such that the number of
integers $n$ under consideration is at most
\begin{equation}\label{eq:fixed-external-count}
\frac XQ
\left(
D^{-\eta+o(1)}
+O_k\left(\exp\left\{-\frac14\sqrt{\log D}\right\}\right)
+O(E_X)
\right),
\end{equation}
uniformly in $(r_1,\ldots,r_k)$ with $Q\leq X^{1-\eps}$.
\end{claim}

\begin{proof}
Let $\mathbf t$ be uniformly distributed on $\{1,\ldots,N\}$.
By \eqref{eq:bin-parameters} and $z=\ee^{(D+1)/K}$,
\[
\frac{\log X}{2\log\log\log X}
<\log z
\leq\frac{\log X}{2\log\log\log X}+\frac1K.
\]
Together with \eqref{eq:N}, this gives $z<N$ for sufficiently
large $X$ and
\[
u=\frac{\log N}{\log z}
\geq
\frac{\eps\log X-\log2}
{\dfrac{\log X}{2\log\log\log X}+\dfrac1K}
=(2\eps+o(1))\log\log\log X.
\]
Consequently, choosing $c_\eps$ sufficiently small, we obtain
\begin{equation}\label{eq:u-error-bound}
u^{-u}\leq
\exp\left\{
-c_\eps(\log\log\log X)\log\log\log\log X
\right\}.
\end{equation}
These estimates are uniform in the fixed external parts.

For $i_1\leq i\leq D$, put
$X_i=\#\{p\in\cP_i:p\mid\mathbf t\}$,
and let $Y_i$ be independent Poisson random variables with means
$R_i$. Thus
\[
\Pp(Y_i=m)=\exp(-R_i)\frac{R_i^m}{m!}
\qquad(m\in\Z_{\geq0}).
\]
The prime sets in distinct bins are disjoint, and all their primes
are at most $z$. Empty bins may be omitted.
Ford's joint Poisson approximation
\cite[Theorem~1]{FordPoisson} therefore gives
\begin{align}
&\left|
\Pp\bigl(I(\mathbf t)\in\mathfrak G_\tau\bigr)
-\Pp\left(
\{i\in[i_1,D]\cap\Z:Y_i>0\}\in\mathfrak G_\tau
\right)
\right| \notag\\
&\qquad\ll
\sum_{i=i_1}^{D}
\frac{\displaystyle\sum_{p\in\cP_i}p^{-2}}{1+R_i}
+u^{-u}
\ll\sum_{p>y}\frac1{p^2}+u^{-u}.
\label{eq:nonsat-poisson-comparison}
\end{align}

By \eqref{eq:R-i-asymp},
\[
\Pp(Y_i>0)
=1-\exp(-R_i)
=R_i+O(R_i^2)
=\frac1i+O(i^{-2}).
\]
The indicators $\one_{\{Y_i>0\}}$ are independent, as are
$\one_{\{i\in\mathcal A\}}$, and
$\Pp(i\in\mathcal A)=1/i$.
Coupling the indicators coordinate by coordinate yields
\begin{align}
&\left|
\Pp\left(
\{i\in[i_1,D]\cap\Z:Y_i>0\}\in\mathfrak G_\tau
\right)
-\Pp\left(
\mathcal A\cap[i_1,D]\in\mathfrak G_\tau
\right)
\right| \notag\\
&\qquad\leq
\sum_{i=i_1}^{D}
\left|1-\exp(-R_i)-\frac1i\right|
\ll\sum_{i=i_1}^{\infty}\frac1{i^2}
\ll\frac1{i_1}.
\label{eq:nonsat-bernoulli-comparison}
\end{align}
Combining \eqref{eq:u-error-bound},
\eqref{eq:nonsat-poisson-comparison}, and
\eqref{eq:nonsat-bernoulli-comparison}, we obtain
\begin{equation}\label{eq:nonsat-model-comparison}
\Pp\bigl(I(\mathbf t)\in\mathfrak G_\tau\bigr)
\leq
\Pp\left(\mathcal A\cap[i_1,D]\in\mathfrak G_\tau\right)
+O(E_X).
\end{equation}

Since $D^{\widehat c_X}=i_1-1$, we have
\[
\mathcal A\cap[i_1,D]
=\mathcal A\cap(D^{\widehat c_X},D].
\]
By the definitions of $\mathfrak G_\tau$ and
$\mathcal N_{k,c,D}(\tau,W)$,
\[
\Pp\left(\mathcal A\cap[i_1,D]\in\mathfrak G_\tau\right)
=\Pp\bigl(\mathcal N_{k,\widehat c_X,D}(\tau,W)\bigr).
\]
Moreover, \eqref{eq:cX-supercritical} gives
$\widehat c_X\geq\beta_k+2\delta$ for sufficiently large $X$,
and $W=D^{o(1)}$.
Combining \eqref{eq:nonsat-model-comparison} with
Theorem~\ref{thm:affine}, applied with the fixed parameter
$\delta>0$, gives a constant $\eta=\eta(k,\delta)>0$ such that
\begin{equation}\label{eq:fixed-external-prob}
\Pp\bigl(I(\mathbf t)\in\mathfrak G_\tau\bigr)
\leq
D^{-\eta+o(1)}
+O_k\left(\exp\left\{-\frac14\sqrt{\log D}\right\}\right)
+O(E_X).
\end{equation}
All estimates are uniform in the fixed external parts, since
the comparison bounds hold for every event and the estimate in
Theorem~\ref{thm:affine} is uniform in $\tau$.

By Claim~\ref{clm:nonsat-affine}, the number of integers $n=Q\ell$
under consideration is at most
\[
\#\{t\in\N:t\leq N,\ I(t)\in\mathfrak G_\tau\}
=N\Pp\bigl(I(\mathbf t)\in\mathfrak G_\tau\bigr).
\]
Using $N\leq X/Q$ and \eqref{eq:fixed-external-prob} proves
\eqref{eq:fixed-external-count}.
\end{proof}

\begin{claim}\label{clm:nonsat-external-sum}
For each tuple $(r_1,\ldots,r_k)$, let $Q$ be defined by
\eqref{eq:Q-def}. Then
\begin{equation}\label{eq:external-budget}
\sum_{\substack{
r_1,\ldots,r_k\geq1\\
p\mid r_1\cdots r_k\ \Longrightarrow\
p\leq X,\ p\notin(y,z]
}}
\frac1Q
\leq
\exp\bigl(O_k(\log\log\log X)\bigr)
=D^{o(1)}.
\end{equation}
\end{claim}

\begin{proof}
By \eqref{eq:Q-def}, for every prime $p$, we have
$v_p(Q)=\max\{v_p(r_1),\ldots,v_p(r_k)\}.$ Therefore,
unique prime factorization gives
\begin{align}
&\sum_{\substack{
r_1,\ldots,r_k\geq1\\
p\mid r_1\cdots r_k\ \Longrightarrow\
p\leq X,\ p\notin(y,z]
}}\frac1Q=
\prod_{\substack{p\leq X\\p\notin(y,z]}}
\left(
\sum_{\alpha_1,\ldots,\alpha_k\geq0}
p^{-\max\{\alpha_1,\ldots,\alpha_k\}}
\right).
\label{eq:external-tuple-product}
\end{align}
For each integer $\nu\geq1$, there are $(\nu+1)^k$ tuples
$(\alpha_1,\ldots,\alpha_k)\in\Z_{\geq0}^k$ satisfying
$\max\{\alpha_1,\ldots,\alpha_k\}\leq\nu,$
and there are $\nu^k$ such tuples for which
$\max\{\alpha_1,\ldots,\alpha_k\}\leq\nu-1$.
Hence the number of tuples satisfying
$\max\{\alpha_1,\ldots,\alpha_k\}=\nu$
is $(\nu+1)^k-\nu^k$. Hence, for every prime $p\geq2$, we have
\begin{align*}
\sum_{\alpha_1,\ldots,\alpha_k\geq0}
p^{-\max\{\alpha_1,\ldots,\alpha_k\}}
&=1+\sum_{\nu\geq1}\frac{(\nu+1)^k-\nu^k}{p^\nu}\\
&\leq1+\frac1p
\sum_{\nu\geq1}\frac{(\nu+1)^k-\nu^k}{2^{\nu-1}}
=1+O_k\left(\frac1p\right).
\end{align*}
Thus \eqref{eq:external-tuple-product} and $1+x\leq\exp(x)$ give
\begin{align}
&\sum_{\substack{
r_1,\ldots,r_k\geq1\\
p\mid r_1\cdots r_k\ \Longrightarrow\
p\leq X,\ p\notin(y,z]
}}
\frac1Q \leq
\exp\left\{
O_k\left(
\sum_{p\leq y}\frac1p+\sum_{z<p\leq X}\frac1p
\right)
\right\}.
\label{eq:external-budget-first}
\end{align}
By Mertens' theorem, we have 
\[
\sum_{p\leq y}\frac1p=\log\log y+O(1),
\qquad
\sum_{z<p\leq X}\frac1p
=\log\log X-\log\log z+O(1).
\]
The definitions of $y$ and $z$ imply
\[
\log\log y=O(\log\log\log X),
\qquad
\log\log X-\log\log z
=O(\log\log\log\log X).
\]
Substituting these estimates into \eqref{eq:external-budget-first}
proves \eqref{eq:external-budget}, since
$\log\log\log X=o(\log D)$.
\end{proof}

Let $M_X$ denote the total number of integers under consideration.
Summing \eqref{eq:fixed-external-count} over all external tuples
with $Q\leq X^{1-\eps}$ and applying
Claim~\ref{clm:nonsat-external-sum}, we obtain
\begin{equation}\label{eq:nonsat-final-bound}
\frac{M_X}{X}
\ll
\exp\bigl(O_k(\log\log\log X)\bigr)
\left(
D^{-\eta+o(1)}
+\exp\left\{-\frac14\sqrt{\log D}\right\}
+E_X
\right).
\end{equation}

Since $\log D=(1+a+o(1))\log\log X$, we have
$\log\log\log X=o(\sqrt{\log D})$. Hence
\begin{align*}
\exp\bigl(O_k(\log\log\log X)\bigr)D^{-\eta+o(1)}
&=D^{-\eta+o(1)}=o(1),\\
\exp\left\{
O_k(\log\log\log X)-\frac14\sqrt{\log D}
\right\}
&=o(1).
\end{align*}
Moreover, \eqref{eq:nonsat-error}, \eqref{eq:bin-parameters},
and \eqref{eq:y-z-asymp} imply
\begin{align*}
\exp\bigl(O_k(\log\log\log X)\bigr)E_X
&\ll
\exp\left\{
-(\log\log X)^3+O_k(\log\log\log X)
\right\}
+(\log X)^{-a+o(1)}\\
&+
\exp\left\{
O_k(\log\log\log X)
-c_\eps(\log\log\log X)\log\log\log\log X
\right\}=o(1).
\end{align*}
Substituting these estimates into \eqref{eq:nonsat-final-bound}
gives $M_X=o(X)$.
Adding the $o(X)$ exceptional integers excluded at the beginning
completes the proof.
\end{proof}

\section{The range \texorpdfstring{$Q>X^{1-\eps}$}{Q greater than X to the power 1-epsilon}}\label{sec:sat}

Let $X\geq2$ and $S\subseteq\N$. Define
\[
\mathcal H_X(S)
=
\sum_{\substack{n\in S\\X/2<n\leq X}}\frac1n.
\]
Since $\frac1X\leq\frac1n\leq\frac2X$ for $X/2<n\leq X),$ it follows that
summing over $n\in S\cap(X/2,X]$ yields
\begin{equation}\label{eq:harmonic-density}
\frac{\#\bigl(S\cap(X/2,X]\bigr)}{X}
\leq
\mathcal H_X(S)
\leq
\frac{2\,\#\bigl(S\cap(X/2,X]\bigr)}{X}.
\end{equation}

Let $n\in(X/2,X]$ admit divisors $d_1<\cdots<d_k\mid n$
satisfying $d_k\leq d_1\bigl(1+(\log n)^{-a}\bigr).$
For $1\leq s\leq k$, write $d_s=r_sm_s$
as in \eqref{eq:external-medium-decomp}. Let $Q$ be the least positive
integer divisible by $r_1,\ldots,r_k$, as in \eqref{eq:Q-def}. We now consider the case
\begin{equation}\label{eq:saturated-range}
Q>X^{1-\eps}.
\end{equation}
All external prime powers occurring in any of
$r_1,\ldots,r_k$ are included in $Q$; in particular, common prime
factors are not removed.

We restrict attention to the case in which
$m_1,\ldots,m_k$ are pairwise distinct. The case of repeated medium
parts has already been treated in Corollary~\ref{cor:DMC}. By
\eqref{eq:harmonic-density}, it is enough to show that the integers
considered in this section have reciprocal sum $o(1)$ on $(X/2,X]$.

In Subsections~\ref{subsec:selected-tuple}--\ref{subsec:weighted-reconstruction},
we restrict to integers satisfying both conclusions of
Lemma~\ref{lem:MOP} and $\Omega(n)\leq(\log\log X)^2$.
The omitted integers form a set of cardinality $o(X)$ by
Lemma~\ref{lem:MOP} and \eqref{eq:Omega-good}, and contribute $o(1)$ to the reciprocal sum
on $(X/2,X]$. They are included again in Lemma~\ref{lem:saturated}.

\subsection{Divisor tuples, residual classes and regularity}
\label{subsec:selected-tuple}

For each integer $n$ under consideration, choose one tuple among those
satisfying the close-divisor condition, having pairwise distinct medium
parts, and satisfying $Q>X^{1-\eps}$. To make the choice unique, choose
successively the smallest possible values of $d_1,\ldots,d_k$ among
these tuples. For $1\leq t\leq k$, put
\[
A_t=A(m_t)\subseteq I(n).
\]
By \eqref{eq:medium-subset-sum},
$S_K(m_t)=\Sigma(A_t)$, and the sets $A_1,\ldots,A_k$ are
pairwise distinct. For every $\omega\in\{0,1\}^k$, define
\[
B_\omega
=
\left\{i\in I(n):
(1_{i\in A_1},\ldots,1_{i\in A_k})=\omega\right\}.
\]
These sets partition $I(n)$, including
$B_{\mathbf0_k}=I(n)\setminus(A_1\cup\cdots\cup A_k)$.

Use the construction described in the proof of Theorem~\ref{thm:affine},
which is the construction of \cite[Section~4.1]{FGK}.
Starting with $V_0=\langle\one\rangle$, select at step $j$ a vector
$\omega^j\notin V_{j-1}$ with $B_{\omega^j}\neq\varnothing$
for which $\max B_{\omega^j}$ is largest, and put
\[
J_j=\max B_{\omega^j},
\qquad
V_j=V_{j-1}+\Span_{\Q}\{\omega^j\}.
\]
The procedure stops when every vector indexing a nonempty set
$B_\omega$ belongs to the last space. It produces a complete flag
\[
\cV:\quad
\langle\one\rangle=V_0<V_1<\cdots<V_h,
\qquad 1\leq h\leq k-1.
\]
Its terminal space is non-degenerate, because $A_t\neq A_u$ for
$t\neq u$. Define
\[
c_j=1+\frac{\lceil\log J_j-\log D\rceil}{\log D}
\quad(1\leq j\leq h),
\qquad c_{h+1}=\widehat c_X.
\]
Then $1\geq c_1\geq\cdots\geq c_h>\widehat c_X$.

Remove $J_j$ from $B_{\omega^j}$ for each $j$, writing
\[
B'_\omega=
\begin{cases}
B_{\omega^j}\setminus\{J_j\},&\omega=\omega^j\ (1\leq j\leq h),\\
B_\omega,&\omega\notin\{\omega^1,\ldots,\omega^h\},
\end{cases}
\qquad
A'=I(n)\setminus\{J_1,\ldots,J_h\}.
\]
For $1\leq j\leq h$, put
$N_j=\#\bigl(A'\cap(D^{c_{j+1}},D^{c_j}]\bigr)$ and define
\begin{equation}\label{eq:empirical-measures}
\mu_j(\omega)=
\begin{cases}
\displaystyle
\frac{\#\bigl(B'_\omega\cap(D^{c_{j+1}},D^{c_j}]\bigr)}{N_j},
&N_j>0,\\[2mm]
1,&N_j=0,\ \omega=\mathbf0_k,\\
0,&N_j=0,\ \omega\neq\mathbf0_k.
\end{cases}
\end{equation}
As in the proof of Theorem~\ref{thm:affine}, every $\mu_j$ is
supported on $V_j\cap\{0,1\}^k$, and the partition
$(B'_\omega)_\omega$ satisfies \eqref{eq:partition-compatibility}.
Put
\begin{equation}\label{eq:Lambda-definition}
\Lambda=(\cV,\mathbf c,\boldsymbol\mu,\omega^1,\ldots,\omega^h).
\end{equation}
Here $\cV:V_0<V_1<\cdots<V_h$ is the complete flag constructed from
the selected divisor tuple,
$\mathbf c=(c_1,\ldots,c_{h+1})$ is the threshold sequence defined above
with $c_{h+1}=\widehat c_X$, and
$\boldsymbol\mu=(\mu_1,\ldots,\mu_h)$ is the family of probability
measures in \eqref{eq:empirical-measures}.
The selected vectors satisfy
$\omega^j\in(V_j\setminus V_{j-1})\cap\{0,1\}^k$ and
$V_j=V_{j-1}+\Span_{\Q}\{\omega^j\}$ for $1\leq j\leq h$.
Fixing $\Lambda$ means fixing all these objects.

For a set $A\subseteq[i_1,D]\cap\Z$, write $A\in\mathcal R$ if
\[
\left|
\#\bigl(A\cap(D^\alpha,D^\beta]\bigr)
-(\beta-\alpha)\log D
\right|\leq(\log D)^{3/4}
\qquad(\widehat c_X\leq\alpha\leq\beta\leq1).
\]
Write $A\in\widetilde{\mathcal R}$ if the same inequalities hold
with $2(\log D)^{3/4}$ in place of $(\log D)^{3/4}$.
If $I(n)\in\mathcal R$, then $\#I(n)\leq2\log D$ for large $D$.
The thresholds are chosen from $O(\log D)$ values, and all numerators
and denominators in \eqref{eq:empirical-measures} are at most
$2\log D$. Consequently, in this range there are at most
$(\log D)^{O_k(1)}$ possible choices of $\Lambda$, as in
\cite[Lemma~4.2]{FGK}.

Since $(D^{\widehat c_X},D]\cap\Z=[i_1,D]\cap\Z$,
it follows that \cite[Lemma~A.5]{FGK} gives
\begin{equation}\label{eq:regularity-failure}
\Pp\bigl(\mathcal A\cap[i_1,D]\notin\mathcal R\bigr)
\ll_k\exp\left\{-\frac14\sqrt{\log D}\right\}.
\end{equation}

Define the residual class by
\begin{equation}\label{eq:canonical-L}
L=
\left(\sum_{\omega\in\{0,1\}^k}\omega\,\Sigma(B'_\omega)\right)
+\langle\one\rangle
\in\Q^k/\langle\one\rangle.
\end{equation}
Its $t$-th coordinate before taking the quotient is
$\sum_{\omega_t=1}\Sigma(B'_\omega)$.

The following lemma establishes the properties of the selected
indices and the residual partition. It expresses $L$ in terms
of the medium-part sums and shows that $I(n)\in\mathcal R$
implies $A'\in\widetilde{\mathcal R}$.

\begin{lemma}\label{lem:residual-interface}
With the notation above, the following statements hold.

\begin{enumerate}[label=\textup{(\roman*)},leftmargin=2.3em]
\item For every $1\leq j\leq h$,
\begin{equation}\label{eq:residual-pivot-scale}
\ee^{-1}D^{c_j}<J_j\leq D^{c_j}.
\end{equation}

\item The sets $B'_\omega$, $\omega\in\{0,1\}^k$, are pairwise
disjoint and have union $A'$. They satisfy the compatibility
conditions in \eqref{eq:partition-compatibility}, with $B=A'$.

\item In the quotient space $\Q^k/\langle\one\rangle$, one has
\begin{equation}\label{eq:residual-identity}
L
=
\left(
\bigl(S_K(m_1),\ldots,S_K(m_k)\bigr)
-
\sum_{j=1}^hJ_j\omega^j
\right)
+
\langle\one\rangle.
\end{equation}

\item If $I(n)\in\mathcal R$, then
$A'\in\widetilde{\mathcal R}$ for all sufficiently large $D$.
\end{enumerate}
\end{lemma}

\begin{proof}
Note that
$$
c_j=1+
\frac{\lceil\log J_j-\log D\rceil}{\log D}.
$$

Since
\[
\log J_j-\log D
\leq
\lceil\log J_j-\log D\rceil
<
\log J_j-\log D+1,
\]
we have
\[
\log J_j
\leq
c_j\log D
<
\log J_j+1.
\]
Exponentiating these inequalities gives $J_j\leq D^{c_j}<\ee J_j,$
which is equivalent to \eqref{eq:residual-pivot-scale}.

By their definition, the sets $B_\omega$,
$\omega\in\{0,1\}^k$, are pairwise disjoint and have union $I(n)$.
For each selected vector $\omega^j$, only the element $J_j$ is
removed from $B_{\omega^j}$, while all other sets remain unchanged.
Therefore, the sets $B'_\omega$ are pairwise disjoint and
\[
\bigcup_{\omega\in\{0,1\}^k}B'_\omega
=
I(n)\setminus\{J_1,\ldots,J_h\}
=
A'.
\]
For each $j$ with $N_j>0$, multiplying
\eqref{eq:empirical-measures} by $N_j$ gives the first condition in
\eqref{eq:partition-compatibility}. If $N_j=0$, both sides vanish.
For $j<h$, every $\omega\notin V_j$ with $B_\omega\neq\varnothing$
satisfies 
$$
\max B_\omega\leq J_{j+1}\leq D^{c_{j+1}}
$$
by the
maximal choice at the next step. For $j=h$, every vector indexing a
nonempty set belongs to $V_h$. These facts prove that $\mu_j$ is
supported on $V_j\cap\{0,1\}^k$.
Similarly, if $\omega\notin V_0$ and $B_\omega\neq\varnothing$, then
\[
\max B_\omega\leq J_1\leq D^{c_1},
\]
which gives the second
condition in \eqref{eq:partition-compatibility}. This proves
\textup{(ii)}.

We next prove \textup{(iii)}. By the definition of $B_\omega$, for
each $1\leq s\leq k$, we have
$$
A_s=\bigcup_{\substack{\omega\in\{0,1\}^k\\\omega_s=1}}
B_\omega,
$$
where the union is disjoint. Hence
\begin{equation}\label{eq:S-before-removal}
S_K(m_s)
=
\Sigma(A_s)
=
\sum_{\substack{\omega\in\{0,1\}^k\\\omega_s=1}}
\sum_{i\in B_\omega}i
\qquad
(1\leq s\leq k).
\end{equation}

For every $1\leq j\leq h$, we have
$B_{\omega^j}=B'_{\omega^j}\cup\{J_j\},$
where the union is disjoint. If
$\omega\notin\{\omega^1,\ldots,\omega^h\}$, then
$B_\omega=B'_\omega$. It follows from
\eqref{eq:S-before-removal} that
\[
\bigl(S_K(m_1),\ldots,S_K(m_k)\bigr)
=
\sum_{\omega\in\{0,1\}^k}
\omega\sum_{i\in B'_\omega}i
+
\sum_{j=1}^hJ_j\omega^j.
\]

Passing to $\Q^k/\langle\one\rangle$ and using
\eqref{eq:canonical-L}, we obtain
\[
L
=
\left(
\bigl(S_K(m_1),\ldots,S_K(m_k)\bigr)
-
\sum_{j=1}^hJ_j\omega^j
\right)
+
\langle\one\rangle,
\]
which proves \eqref{eq:residual-identity}.

Finally, suppose that $I(n)\in\mathcal R$. For every
$\widehat c_X\leq\alpha\leq\beta\leq1$, deleting
$J_1,\ldots,J_h$ changes the number of elements in
$(D^\alpha,D^\beta]$ by at most $h$. Therefore, for every
$\widehat c_X\leq\alpha\leq\beta\leq1$ and all sufficiently large
$D$,
\begin{align*}
\left|
\#\bigl(A'\cap(D^\alpha,D^\beta]\bigr)
-
(\beta-\alpha)\log D
\right|
&\leq
\left|
\#\bigl(I(n)\cap(D^\alpha,D^\beta]\bigr)
-
(\beta-\alpha)\log D
\right|
+h\\
&\leq
(\log D)^{3/4}+k-1\\
&\leq
2(\log D)^{3/4},
\end{align*}
where the last inequality follows since $k$ is fixed. Hence
$A'\in\widetilde{\mathcal R}$, proving \textup{(iv)}.
\end{proof}

\subsection{Harmonic normalization}
We next estimate the sums of reciprocal weights arising from
medium and external prime factors. The following lemma gives
the required normalization factors and bounds the weighted
contribution of sets $A\notin\mathcal R$.

\begin{lemma}\label{lem:harmonic-normalization}
Let $R_i$ be defined by \eqref{eq:R-i-def}, let $\mathcal R$
be the family defined above, and put
$Z_R=\prod_{i=i_1}^D(1+R_i).$
Then the following statements hold as $X\to\infty$.

\begin{enumerate}[label=\textup{(\roman*)},leftmargin=2.3em]
\item We have
\[
\sum_{A\subseteq[i_1,D]\cap\Z}\prod_{i\in A}R_i
=Z_R
=(1+o(1))\frac D{i_1}.
\]

\item We have
\begin{equation}\label{eq:irregular-medium-weight}
\sum_{\substack{
A\subseteq[i_1,D]\cap\Z\\
A\notin\mathcal R
}}
\prod_{i\in A}R_i
\ll
Z_R\left(
\exp\left\{-\frac14\sqrt{\log D}\right\}
+\frac1{i_1}
\right).
\end{equation}

\item The quantity
\begin{equation}\label{eq:Zext-harmonic}
\mathcal Z_{\rm ext}^{(k)}
=
\prod_{\substack{
p\leq X\\
p\notin(y,z]
}}
\left(
\sum_{\nu=0}^{\infty}\frac{(\nu+1)^k}{p^\nu}
\right)
\end{equation}
satisfies
\[
\mathcal Z_{\rm ext}^{(k)}
\leq
\exp\bigl(O_k(\log\log\log X)\bigr)
=D^{o(1)}.
\]
\end{enumerate}
\end{lemma}

\begin{proof}
\textup{(i)}
Expanding the finite product gives
\[
Z_R
=\prod_{i=i_1}^D(1+R_i)
=\sum_{A\subseteq[i_1,D]\cap\Z}\prod_{i\in A}R_i.
\]
By \eqref{eq:R-i-asymp},
\[
\log(1+R_i)
=R_i+O(R_i^2)
=\frac1i+O(i^{-2})
\qquad(i_1\leq i\leq D).
\]
Consequently,
\begin{align*}
\log Z_R
=\sum_{i=i_1}^D\log(1+R_i)
=\sum_{i=i_1}^D\frac1i
+O\left(\sum_{i=i_1}^D\frac1{i^2}\right)
=\log\frac D{i_1}+O(i_1^{-1}).
\end{align*}
Since $i_1\to\infty$, it follows that
\[
Z_R
=\frac D{i_1}\exp\bigl(O(i_1^{-1})\bigr)
=(1+o(1))\frac D{i_1}.
\]

\medskip
\textup{(ii)}
Let $\mathcal B\subseteq[i_1,D]\cap\Z$ be a random set such that
the events $i\in\mathcal B$ are independent and
$\Pp(i\in\mathcal B)=\frac{R_i}{1+R_i}.$
For every $A\subseteq[i_1,D]\cap\Z$, independence gives
\begin{align*}
\Pp(\mathcal B=A)
&=
\prod_{i\in A}\frac{R_i}{1+R_i}
\prod_{\substack{i_1\leq i\leq D\\i\notin A}}
\frac1{1+R_i}=\frac1{Z_R}\prod_{i\in A}R_i.
\end{align*}
Hence
\[
\sum_{\substack{
A\subseteq[i_1,D]\cap\Z\\
A\notin\mathcal R
}}
\prod_{i\in A}R_i
=
Z_R\Pp(\mathcal B\notin\mathcal R).
\]

By \eqref{eq:R-i-asymp},
\[
\frac{R_i}{1+R_i}
=R_i+O(R_i^2)
=\frac1i+O(i^{-2}).
\]
The membership events of $\mathcal A$ are independent and
$\Pp(i\in\mathcal A)=1/i$.
A coordinatewise coupling therefore gives
\begin{align*}
&\left|
\Pp(\mathcal B\notin\mathcal R)
-\Pp\bigl(\mathcal A\cap[i_1,D]\notin\mathcal R\bigr)
\right|
\leq
\sum_{i=i_1}^D
\left|\frac{R_i}{1+R_i}-\frac1i\right|
\ll\sum_{i=i_1}^{\infty}\frac1{i^2}
\ll\frac1{i_1}.
\end{align*}
Combining this estimate with \eqref{eq:regularity-failure},
we obtain
\[
\Pp(\mathcal B\notin\mathcal R)
\ll
\exp\left\{-\frac14\sqrt{\log D}\right\}
+\frac1{i_1}.
\]
Multiplication by $Z_R$ proves
\eqref{eq:irregular-medium-weight}.

\medskip
\textup{(iii)}
For every prime $p\geq2$,
\begin{align*}
1\leq
\sum_{\nu=0}^{\infty}\frac{(\nu+1)^k}{p^\nu}
&=1+\sum_{\nu=1}^{\infty}\frac{(\nu+1)^k}{p^\nu}\\
&\leq
1+\frac1p
\sum_{\nu=1}^{\infty}\frac{(\nu+1)^k}{2^{\nu-1}}
=1+O_k\left(\frac1p\right).
\end{align*}
Therefore,
\begin{align*}
0\leq\log\mathcal Z_{\rm ext}^{(k)}
&=
\sum_{\substack{p\leq X\\p\notin(y,z]}}
\log\left(
\sum_{\nu=0}^{\infty}\frac{(\nu+1)^k}{p^\nu}
\right)\ll_k
\sum_{p\leq y}\frac1p+\sum_{z<p\leq X}\frac1p.
\end{align*}
By Mertens' theorem,
\[
\sum_{p\leq y}\frac1p+\sum_{z<p\leq X}\frac1p
=
\log\log y+\log\log X-\log\log z+O(1).
\]
The definitions in \eqref{eq:bin-parameters} give
\[
\log\log y=O(\log\log\log X),
\qquad
\log\log X-\log\log z
=O(\log\log\log\log X).
\]
Thus
\[
0\leq\log\mathcal Z_{\rm ext}^{(k)}
=O_k(\log\log\log X)
=o(\log D),
\]
where the last equality follows from
$\log D=(1+a+o(1))\log\log X$.
Exponentiating completes the proof.
\end{proof}

\subsection{Control of external prime powers}

We next remove a further exceptional set involving the external prime
factors of $n$.  Recall that
$y=\ee^{i_1/K}$ and $z=\ee^{(D+1)/K}$,
so that the medium primes are those in $(y,z]$, while the external
primes are those not belonging to this interval.

\begin{lemma}\label{lem:external-good}
Let $\eps>0$ be fixed.  For all but $o(X)$ integers $n\leq X$, we have
\begin{equation}\label{eq:small-external-good}
\prod_{p\leq y}p^{v_p(n)}
\leq
X^{\eps/4}
\end{equation}
and
\begin{equation}\label{eq:large-squarefree}
p^2\nmid n
\qquad
\text{for every prime $p>z$}.
\end{equation}
\end{lemma}

\begin{proof}
We first consider the prime factors of $n$ not exceeding $y$.  For
every positive integer $n$, we have
$$
\log\left(\prod_{p\leq y}p^{v_p(n)}\right)
=\sum_{p\leq y}v_p(n)\log p.
$$
For each fixed prime $p$, we have
\[
\begin{aligned}
\sum_{n\leq X}v_p(n)
=\sum_{j\geq1}
\left\lfloor\frac{X}{p^j}\right\rfloor.
\end{aligned}
\]
Therefore,
\[
\begin{aligned}
\frac1X
\sum_{n\leq X}
\log\left(\prod_{p\leq y}p^{v_p(n)}\right)
&=\frac1X
\sum_{p\leq y}\log p
\sum_{n\leq X}v_p(n)
=\frac1X
\sum_{p\leq y}\log p
\sum_{j\geq1}
\left\lfloor\frac{X}{p^j}\right\rfloor\\
&\leq
\sum_{p\leq y}\log p
\sum_{j\geq1}\frac1{p^j}
=
\sum_{p\leq y}\frac{\log p}{p-1}.
\end{aligned}
\]

Since $p\geq2$, we have $\frac1{p-1}\leq\frac2p.$
Hence Mertens' estimate gives
\[
\sum_{p\leq y}\frac{\log p}{p-1}
\ll
\sum_{p\leq y}\frac{\log p}{p}
\ll
\log y.
\]
It follows from the preceding calculation that
\[
\sum_{n\leq X}
\log\left(\prod_{p\leq y}p^{v_p(n)}\right)
\ll
X\log y.
\]

Let $\mathcal E_y=\left\{
n\leq X:
\prod_{p\leq y}p^{v_p(n)}>X^{\eps/4}
\right\}.$
For every $n\in\mathcal E_y$, taking logarithms gives
\[
\log\left(\prod_{p\leq y}p^{v_p(n)}\right)
>\frac{\eps}{4}\log X.
\]
Therefore,
\[
\begin{aligned}
\frac{\eps}{4}\log X\,|\# \mathcal E_y|
\leq\sum_{n\in\mathcal E_y}
\log\left(\prod_{p\leq y}p^{v_p(n)}\right)
\leq
\sum_{n\leq X}
\log\left(\prod_{p\leq y}p^{v_p(n)}\right)
\ll X\log y.
\end{aligned}
\]
Consequently, $\frac{|\# \mathcal E_y|}{X}
\ll_{\eps} \frac{\log y}{\log X}.$
By \eqref{eq:y-z-asymp},
\[
\log y
=
(\log\log X)^3+o(1)
=
o(\log X).
\]
Thus $\frac{|\#  \mathcal E_y|}{X}=o(1),$
and hence $|\#  \mathcal E_y|=o(X).$
This proves that \eqref{eq:small-external-good} holds for all but
$o(X)$ integers $n\leq X$.

We next consider primes exceeding $z$. If
\eqref{eq:large-squarefree} fails, then there exists a prime $p>z$
such that $p^2\mid n$.  By the union bound,
\[
\begin{aligned}
\frac1X
\#\left\{
n\leq X:
p^2\mid n
\text{ for some prime $p>z$}
\right\}
\leq \sum_{p>z} \frac1X
\left\lfloor\frac{X}{p^2}\right\rfloor
\leq \sum_{p>z}\frac1{p^2}
\ll \frac1z.
\end{aligned}
\]
Since $z\to\infty$ by \eqref{eq:y-z-asymp}, the last quantity is
$o(1)$.  Thus \eqref{eq:large-squarefree} also fails for only $o(X)$
integers.  Combining the two exceptional sets proves the lemma.
\end{proof}

For the rest of this section, we assume that $n$ satisfies both
conclusions of Lemma~\ref{lem:external-good}.  This excludes only
$o(X)$ integers.  Since every $n\in(X/2,X]$ satisfies $1/n\leq2/X$,
the excluded integers contribute at most
\[
\frac{2}{X}\,o(X)=o(1).
\]

Let $\lcm(m_1,\ldots,m_k)$ denote the least common multiple of
$m_1,\ldots,m_k$.
Recall from \eqref{eq:Q-def} that
$Q=\lcm(r_1,\ldots,r_k)$, where $r_t$ is the external part of $d_t$.
In particular, $Q\mid n$ and every prime divisor of $Q$ is external.
Therefore,
$\gcd\bigl(Q,\lcm(m_1,\ldots,m_k)\bigr)=1.$

Since $Q\mid n$ and $m_t\mid n$ for every $1\leq t\leq k$, we also
have $\lcm(m_1,\ldots,m_k)\mid n.$
It follows that $Q\lcm(m_1,\ldots,m_k)\mid n.$
Using $n\leq X$ and $Q>X^{1-\eps}$, we obtain
\begin{equation}\label{eq:medium-lcm-small}
\lcm(m_1,\ldots,m_k)
\leq
\frac{n}{Q}
<
X^\eps.
\end{equation}

Define the linear map $f:\R^k\to\R^{k-1}$ by
\[
f(x_1,\ldots,x_k)
=
(x_2-x_1,\ldots,x_k-x_1).
\]
For every prime $p>z$ dividing $Q$, define
\begin{equation}\label{eq:u-p-def}
u(p)=f\bigl(v_p(r_1),\ldots,v_p(r_k)\bigr).
\end{equation}
Since $Q\mid n$ and $p^2\nmid n$ by
\eqref{eq:large-squarefree}, we have $v_p(Q)=1$.  Moreover,
$r_t\mid Q$ for every $1\leq t\leq k$.  Hence $v_p(r_t)\in\{0,1\}
(1\leq t\leq k),$
and therefore $u(p)\in f \bigl(\{0,1\}^k\bigr).$
Let
\[
E=\Span\{u(p):p>z,\ p\mid Q\}
\subseteq
\R^{k-1},
\qquad s=\dim E.
\]
We use the same symbol $f$ for the induced map on
$\Q^k/\langle\one\rangle$, defined by
$f(x+\langle\one\rangle)=f(x)$.
This is well defined because $f(t\one)=\mathbf0_{k-1}$ for every
$t\in\Q$.
We also define
\begin{equation}\label{eq:SL-TL}
S_L
=
\sum_{\substack{p>z\\p\mid Q}}\log p
\qquad\text{and}\qquad
T_L
=
\sum_{\substack{p>z\\p\mid Q}}u(p)\log p.
\end{equation}

We first estimate $S_L$.  Every prime divisor of $Q$ is either at
most $y$ or greater than $z$.  In addition, every prime $p>z$
dividing $Q$ occurs in $Q$ with exponent one.  Consequently,
\[
\log Q
=
\sum_{p\leq y}v_p(Q)\log p
+
S_L.
\]
Since $Q\mid n$, we have $v_p(Q)\leq v_p(n)$ for every prime $p$.
It follows from \eqref{eq:small-external-good} that
\[
\sum_{p\leq y}v_p(Q)\log p
\leq
\sum_{p\leq y}v_p(n)\log p
\leq
\frac{\eps}{4}\log X.
\]
Together with $Q>X^{1-\eps}$, this gives
\begin{equation}\label{eq:SL-lower}
S_L
\geq
\left(1-\frac{5\eps}{4}\right)\log X.
\end{equation}

We next estimate $T_L$.  Since each $r_t$ is supported on external
primes, the linearity of $f$ gives
\[
\begin{aligned}
T_L=f(\log r_1,\ldots,\log r_k)
-\sum_{p\leq y}
f\bigl(v_p(r_1),\ldots,v_p(r_k)\bigr)\log p.
\end{aligned}
\]
Using $d_t=r_tm_t$, we obtain
\[
\begin{aligned}
T_L
=f(\log d_1,\ldots,\log d_k)
-f(\log m_1,\ldots,\log m_k)
-\sum_{p\leq y}
f\bigl(v_p(r_1),\ldots,v_p(r_k)\bigr)\log p.
\end{aligned}
\]

We estimate the three terms on the right.  Since
$d_1<\cdots<d_k$ and $d_k\leq d_1\bigl(1+(\log n)^{-a}\bigr),$
it follows that each coordinate of $f(\log d_1,\ldots,\log d_k)$ satisfies
\[
0
\leq
\log\frac{d_t}{d_1}
\leq
\log\bigl(1+(\log n)^{-a}\bigr)
\ll
(\log X)^{-a}.
\]

Throughout this section, for
$\mathbf x=(x_1,\ldots,x_{k-1})\in\R^{k-1}$, we write
$\|\mathbf x\|
=\left(
\sum_{j=1}^{k-1}x_j^2
\right)^{1/2}.$
Thus $\|\cdot\|$ denotes the Euclidean norm on $\R^{k-1}$.
We use the same notation for the Euclidean norm on other real coordinate
spaces.
Hence
\begin{equation}\label{eq-1}
\bigl\| f(\log d_1,\ldots,\log d_k)\bigr\|
\ll_k
(\log X)^{-a}.
\end{equation}

By \eqref{eq:medium-lcm-small}, we have 
\[
1
\leq
m_t
\leq
\lcm(m_1,\ldots,m_k)
<
X^\eps
\qquad
(1\leq t\leq k).
\]
Thus $|\log\frac{m_t}{m_1}|<\eps\log X$ for $2\leq t\leq k,$
and consequently,
\begin{equation}\label{eq-2}
\bigl\| f(\log m_1,\ldots,\log m_k)\bigr\|
\ll_k
\eps\log X.
\end{equation}

Finally, since $r_t\mid Q\mid n$, condition
\eqref{eq:small-external-good} gives
\[
0
\leq
\sum_{p\leq y}v_p(r_t)\log p
\leq
\frac{\eps}{4}\log X
\qquad
(1\leq t\leq k).
\]
For every $2\leq t\leq k$, the corresponding coordinate of
$\sum_{p\leq y}
f\bigl(v_p(r_1),\ldots,v_p(r_k)\bigr)\log p$
is $\sum_{p\leq y}
\bigl(v_p(r_t)-v_p(r_1)\bigr)\log p.$
This is the difference of two numbers belonging to
$[0,\frac{\eps}{4}\log X]$.  Hence
\begin{equation}\label{eq-3}
\left\|
\sum_{p\leq y}
f\bigl(v_p(r_1),\ldots,v_p(r_k)\bigr)\log p
\right\|
\ll_k
\eps\log X.
\end{equation}
Combining \eqref{eq-1}, \eqref{eq-2}, and \eqref{eq-3} yields
\[
\|T_L\|
\ll_k
(\log X)^{-a}+\eps\log X.
\]
Since $\eps>0$ is fixed, it follows that $(\log X)^{-a}\leq\eps\log X$ for all
sufficiently large $X$, and hence 
\begin{equation}\label{eq:TL-small}
\|T_L\|
\ll_k
\eps\log X.
\end{equation}

\subsection{Rank of the augmented external vectors}

All spans and dimensions in this and the following subsections are
taken over $\R$. Rational spaces generated by cube points are identified
with their real linear spans when used with $f$. Their dimensions do
not change under this identification.
For $x=(x_1,\ldots,x_{k-1})\in E$, the notation $\binom{1}{x}$ means the block column vector
\[
\binom{1}{x}
=
(1,x_1,\ldots,x_{k-1})^{\mathsf T}
\in
\R\oplus E,
\]
where $\R\oplus E
=\left\{
(\xi,x_1,\ldots,x_{k-1})^{\mathsf T}:
\xi\in\R,\
(x_1,\ldots,x_{k-1})\in E
\right\}.$

By definition, the vectors $u(p)$ with $p>z$ and $p\mid Q$
span $E$. The following lemma shows that the columns
$\binom{1}{u(p)}$ span $\R\oplus E$, so their span has dimension $s+1$.

\begin{lemma}\label{lem:augmented-rank}
If $\eps>0$ is sufficiently small in terms of $k$, then
\begin{equation*}
\dim\Span
\left\{
\binom{1}{u(p)}:
p>z,\ p\mid Q
\right\}
=
s+1.
\end{equation*}
\end{lemma}

\begin{proof}
Put
\[
H=\Span\left\{
\binom{1}{u(p)}:p>z,\ p\mid Q \right\} \subseteq \R\oplus E.
\]
Note that $\dim E=s$. Then $\dim(\R\oplus E)=s+1.$
Therefore, $\dim H\leq s+1.$
We first note that there is at least one prime $p>z$ dividing $Q$.
Indeed, by \eqref{eq:SL-lower},
$S_L \geq \left(1-\frac{5\eps}{4}\right)\log X.$
For $\eps<4/5$ and $X>1$, the right-hand side is positive.  Since
$S_L
=\sum_{\substack{p>z\\p\mid Q}}\log p,$
the set of primes $p>z$ dividing $Q$ cannot be empty.

Suppose first that $s=0$.  Then $E=\{\mathbf0_{k-1}\}$, and hence $u(p)=\mathbf0_{k-1}$
for every prime $p>z$ dividing $Q$.  Since at least one such prime
exists, it follows that
$H=\Span\left\{
\binom{1}{\mathbf0_{k-1}}
\right\}.$
Thus $\dim H=1=s+1.$

Assume that $s\geq1$.  Suppose, to the contrary, that $\dim H\leq s.$
Define $\phi:H\longrightarrow E$
by
\[
\phi\binom{\xi}{x}=x
\qquad
\left(
\binom{\xi}{x}\in H
\right).
\]
Thus $\phi$ maps a vector in $\R\oplus E$ to its component in $E$.
Since the vectors $u(p)$ span $E$ by the definition of $E$, the map
$\phi$ is surjective.  Therefore,
\[
\dim H\geq\dim E=s.
\]
Together with the assumed inequality $\dim H\leq s$, this gives $\dim H=s.$
Hence $\phi$ is a surjective linear map between two
finite-dimensional vector spaces of the same dimension.  It follows
that $\phi$ is an isomorphism.

For every $x\in E$, define $\psi(x)$ by
\[
\phi^{-1}(x)
=
\binom{\psi(x)}{x}.
\]
 Since $\phi^{-1}$ is
linear, the map $\psi:E \longrightarrow \R$
is a linear functional.

For every prime $p>z$ dividing $Q$, the vector
$\binom{1}{u(p)}$ belongs to $H$ and satisfies $\phi\binom{1}{u(p)}=u(p).$
Since $\phi^{-1}(u(p))$ is uniquely determined, we have
\[
\phi^{-1}(u(p))
=
\binom{1}{u(p)}.
\]
Comparing the first coordinates gives $\psi(u(p))=1$ for every prime $p>z$ dividing $Q$.

We next bound the operator norm of $\psi$ by a constant depending
only on $k$.  By \eqref{eq:large-squarefree}, if $p>z$ divides $Q$, then
$p^2\nmid n$.  Since $r_t\mid Q\mid n$, it follows that $v_p(r_t)\in\{0,1\}$ for
$1\leq t\leq k.$
By the definition of $u(p)$,
\[
u(p)
=
\bigl(
v_p(r_2)-v_p(r_1),
\ldots,
v_p(r_k)-v_p(r_1)
\bigr),
\]
and hence $u(p)\in\{-1,0,1\}^{k-1}.$

Choose $w_1,\ldots,w_s
\in \{u(p):p>z,\ p\mid Q\}$
that form a basis of $E$.  Since $\psi(u(p))=1$ for every prime
 $p>z$ dividing $Q$, we have
$\psi(w_j)=1(1\leq j\leq s).$
For $x\in E$, write uniquely
$x=\sum_{j=1}^s c_jw_j.$
By the linearity of $\psi$,
$\psi(x)=\sum_{j=1}^s c_j.$

Let $M$ be the $(k-1)\times s$ matrix whose columns are
$w_1,\ldots,w_s$.  Since $\rank M=s$, there is a nonsingular
$s\times s$ submatrix $M_0$ obtained by selecting $s$ rows of $M$.
Let $x_0\in\R^s$ consist of the corresponding coordinates of $x$.
Then
\[
M_0
\begin{pmatrix}
c_1\\
\vdots\\
c_s
\end{pmatrix}
=x_0.
\]
For $1\leq j\leq s$, let $M_{0,j}$ be obtained from $M_0$ by
replacing its $j$-th column by $x_0$.  Cramer's rule gives
$c_j=\frac{\det M_{0,j}}{\det M_0}.$ Since the entries of $M_0$ belong to $\{-1,0,1\}$ and $M_0$ is
nonsingular, $\det M_0$ is a nonzero integer.  Hence
$|\det M_0|\geq1.$

Let the columns of $M_{0,j}$ be
$a_1,\ldots,a_s$.  By construction, the $j$-th column is
$a_j=x_0$.  For every $i\neq j$, the column $a_i$ is a column of
$M_0$.  Since all entries of $M_0$ belong to $\{-1,0,1\}$, we have
\[
\|a_i\|^2
=
\sum_{j=1}^s |(a_i)_j|^2
\leq
s,
\]
and hence $\|a_i\|\leq\sqrt{s}.$

Hadamard's inequality states that the absolute value of the
determinant of a square matrix is at most the product of the
Euclidean norms of its columns.  Therefore,
\[
\begin{aligned}
|\det M_{0,j}|
&\leq
\prod_{i=1}^s\|a_i\|=
\|x_0\| \prod_{\substack{1\leq i\leq s\\i\neq j}}\|a_i\|\\
&\leq
\|x_0\|(\sqrt{s})^{s-1}.
\end{aligned}
\]

Since $\|x_0\|\leq\|x\|$, it follows that $|c_j|
\leq s^{(s-1)/2}\|x\| $ for $1\leq j\leq s.$
Therefore,
\[
|\psi(x)|
\leq
\sum_{j=1}^s|c_j|
\leq
s^{(s+1)/2}\|x\|.
\]
Since $s\leq k-1$, there exists a constant $C_k>0$, depending only
on $k$, such that
\[
|\psi(x)|
\leq
C_k\|x\|
\qquad
(x\in E).
\]

Using \eqref{eq:SL-TL}, the linearity of $\psi$, and
$\psi(u(p))=1$, we obtain
\[
\begin{aligned}
\psi(T_L)
=\psi\left(
\sum_{\substack{p>z\\p\mid Q}}
u(p)\log p
\right)=
\sum_{\substack{p>z\\p\mid Q}}
\psi(u(p))\log p
=\sum_{\substack{p>z\\p\mid Q}}
\log p=S_L.
\end{aligned}
\]
Since $S_L>0$, it follows that
\[
S_L
=
|\psi(T_L)|
\leq
C_k\|T_L\|.
\]
By \eqref{eq:TL-small}, there is a constant $C'_k>0$, depending only
on $k$, such that $S_L
\leq
C'_k\eps\log X.$
On the other hand, \eqref{eq:SL-lower} gives $$
S_L \geq \left(1-\frac{5\eps}{4}
\right)\log X.
$$
Choose $\eps>0$ sufficiently small in terms of $k$ so that $C'_k\eps
<1-\frac{5\eps}{4}.$
The last two inequalities contradict each other.  Hence the
assumption $\dim H\leq s$ is false.

Therefore, $\dim H>s.$
Since $\dim H\leq s+1$, we conclude that $\dim H=s+1.$
\end{proof}

\subsection{A basis adapted to the target spaces}
For $1\leq j\leq h \leq k-1$, write $\omega^j
=(\omega^j_1,\ldots,\omega^j_k)
\in \{0,1\}^k.$
Put
\[
v_j
=
f(\omega^j)
=
\bigl(
\omega^j_2-\omega^j_1,
\ldots,
\omega^j_k-\omega^j_1
\bigr)
\in
\{-1,0,1\}^{k-1},
\]
and let $U=\Span\{v_1,\ldots,v_h\}
\subseteq \R^{k-1}.$

The vectors $v_1,\ldots,v_h$ are linearly independent.  Suppose that
$b_1,\ldots,b_h\in\R$ satisfy
$\sum_{j=1}^h b_jv_j=\mathbf 0_{k-1},$
where $\mathbf 0_{k-1}$ denotes the zero vector in $\R^{k-1}$.
Since $v_j=f(\omega^j)$ and $f$ is linear, we have
$$f\left(\sum_{j=1}^h b_j\omega^j\right)
=\mathbf 0_{k-1}.$$

Write $\sum_{j=1}^h b_j\omega^j
=(x_1,\ldots,x_k)\in\R^k.$
By the definition of $f$, the preceding equality is equivalent to
\[
(x_2-x_1,\ldots,x_k-x_1)
=\mathbf 0_{k-1}
=(0,\ldots,0).
\]
Hence $x_1=x_2=\cdots=x_k.$ Therefore,
$\sum_{j=1}^h b_j\omega^j
=x_1\one,$
and therefore
\[
-x_1\one+\sum_{j=1}^h b_j\omega^j
=\mathbf 0_k.
\]
The vectors $\one,\omega^1,\ldots,\omega^h$ are linearly
independent. Consequently,
$x_1=b_1=\cdots=b_h=0.$
Thus $v_1,\ldots,v_h$ are linearly independent.  Since they span
$U$, they form a basis of $U$, and hence $\dim U=h.$

Set $s=\dim E$ and $q=\dim(E\cap U).$
The dimension formula gives
\begin{equation}\label{dim-E+U}
\dim(E+U)
=\dim E+\dim U-\dim(E\cap U)
=s+h-q.
\end{equation}

We next choose a basis of $E$ from the vectors $u(p)$ and
extend it to a basis of $E+U$ using suitable $v_j$.
The following lemma also gives a basis of $\R\oplus(E+U)$
by adjoining a first coordinate equal to $1$ and adding
one further column.

\begin{lemma}\label{lem:pivot-basis}
There exist vectors $u_0,u_1,\ldots,u_s
\in
\{u(p):p>z,\ p\mid Q\}$
and a set $J\subseteq\{1,\ldots,h\}$ satisfying the following
properties.
\begin{enumerate}[label=\textup{(\roman*)},leftmargin=2.3em]
\item
The vectors $u_1,\ldots,u_s$ form a basis of $E$, and the columns
$\binom{1}{u_0},
\binom{1}{u_1},
\ldots,
\binom{1}{u_s}$
form a basis of $\R\oplus E$.

\item
The cosets $v_j+E$ for $j\in J$
form a basis of $(E+U)/E$.
\end{enumerate}

For every such choice,
$|J|=h-q.$
Moreover, the columns
\begin{equation}\label{eq:joint-columns}
\binom{1}{u_0},
\binom{1}{u_1},
\ldots,
\binom{1}{u_s},
\binom{1}{v_j}
\quad
(j\in J)
\end{equation}
form a basis of $\R\oplus(E+U)$.  Consequently, if
$F=\{1,\ldots,h\}\setminus J,$
then $|F|=q.$
\end{lemma}

\begin{proof}
By the definition of $E$, choose
$u_1,\ldots,u_s
\in
\{u(p):p>z,\ p\mid Q\}$
that form a basis of $E$. Then the columns $\binom{1}{u_1},
\ldots,
\binom{1}{u_s}$
are linearly independent.
By Lemma~\ref{lem:augmented-rank}, there exists 
$$u_0\in\{u(p):p>z,\ p\mid Q\}$$
such that $\binom{1}{u_0},
\binom{1}{u_1},
\ldots,
\binom{1}{u_s}$
are linearly independent.  Since
$\dim(\R\oplus E)=s+1,$ it follows that
these columns form a basis of $\R\oplus E$.

Since $v_1,\ldots,v_h$ form a basis of $U$, the cosets
$v_1+E,\ldots,v_h+E$ span $(E+U)/E$.  Choose
$J\subseteq\{1,\ldots,h\}$ such that $v_j+E\; (j\in J)$ form a basis of $(E+U)/E$.  By \eqref{dim-E+U}, we have
\[
\begin{aligned}
|J|=\dim\bigl((E+U)/E\bigr)=
\dim(E+U)-\dim E=h-q.
\end{aligned}
\]

We now prove that the columns in \eqref{eq:joint-columns} form a
basis of $\R\oplus(E+U)$.  Suppose that
\[
\sum_{i=0}^s a_i\binom{1}{u_i}
+
\sum_{j\in J}b_j\binom{1}{v_j}
=
\mathbf 0_k.
\]
Comparing the last $k-1$ coordinates, we obtain
\[
\sum_{i=0}^s a_iu_i
+
\sum_{j\in J}b_jv_j
=
\mathbf 0_{k-1}.
\]
Since $u_i\in E$ for every $0\leq i\leq s$, it follows that
$$\sum_{j\in J}b_jv_j=-\sum_{i=0}^s a_iu_i
\in E.$$

The cosets $v_j+E$, $j\in J$, are linearly independent in
$(E+U)/E$.  Therefore, the only linear combination of the vectors
$v_j$, $j\in J$, that belongs to $E$ is the trivial one.  Hence
$b_j=0$ for every $j\in J$. Therefore,
\[
\sum_{i=0}^s a_i\binom{1}{u_i}
=
\mathbf 0_k.
\]
Since $\binom{1}{u_0},
\binom{1}{u_1},
\ldots,
\binom{1}{u_s}$
form a basis of $\R\oplus E$, we have $a_0=a_1=\cdots=a_s=0.$
Thus the columns in \eqref{eq:joint-columns} are linearly
independent.
The number of these columns is
\[
s+1+|J|
=
s+1+h-q.
\]
On the other hand, by \eqref{dim-E+U},
\[
\dim\bigl(\R\oplus(E+U)\bigr)
=1+\dim(E+U)
=1+s+h-q.
\]
Hence the columns in \eqref{eq:joint-columns} form a basis of
$\R\oplus(E+U)$.

Finally, since $F=\{1,\ldots,h\}\setminus J$
and $|J|=h-q$, we obtain $|F|=q.$
\end{proof}

\subsection{Localization of primes and bin indices}

Retain $K$ and the intervals $\cP_i$ from
Section~\ref{sec:arithmetic-setup}, and retain
\[
W=C_{a,k}\bigl(1+(\log\log X)^2\bigr)
\]
from the definition following Proposition~\ref{prop:AMR}.  Thus
$W=D^{o(1)}$ and $W/K=o(1)$.

Retain $\Lambda$, $A'$, and $L$ from
Subsection~\ref{subsec:selected-tuple}, and the spaces $E,U$, the vectors
$u_0,\ldots,u_s$, and the sets $J,F$ from
Lemma~\ref{lem:pivot-basis}.  Recall that
$J_j=\max B_{\omega^j}$ for $1\leq j\leq h$, and that $p_{J_j}$ is
the unique prime in $\cP_{J_j}$ dividing $n$.

Fix the primes $p_i$ for $i\in A'$ and the pairs
$(J_j,p_{J_j})$ for $j\in F$. Also fix $b\in\mathcal S_X$ and
integers $\alpha_{t,p}$ with
$0\leq\alpha_{t,p}\leq v_p(b)$ for $p\mid b$ and $1\leq t\leq k$.
We consider representations in which the external part of $n$ is
$bq_0\cdots q_s$, where the $q_i$ are distinct primes not dividing $b$,
and in which the external divisor exponents at $p\mid b$ are $\alpha_{t,p}$.
Thus all external prime powers apart from the $q_i$ are fixed.
For $0\leq i\leq s$, the prime $q_i$ is required to satisfy
\[
q_i>z,\qquad q_i\mid Q,\qquad u(q_i)=u_i.
\]
The quantities that may vary are $q_0,\ldots,q_s$, the pairs
$(J_j,p_{J_j})$ for $j\in J$, and the families
$(B'_\omega)_{\omega\in\{0,1\}^k}$ representing the fixed class $L$
in \eqref{eq:canonical-L}.

The basis relations above allow us to compare two admissible
choices of primes and indices. The following lemma bounds the
variation of $\log q_0$ and, after $q_0$ is fixed, that of
$\log q_i$ and $J_j$.

\begin{lemma}\label{lem:joint-localization}
Suppose that
$\bigl(q_0,\ldots,q_s,(J_j,p_{J_j})_{j\in J}\bigr)$ and
$\bigl(\widetilde q_0,\ldots,\widetilde q_s,
(\widetilde J_j,\widetilde p_{\widetilde J_j})_{j\in J}\bigr)$
are obtained from two choices satisfying all the conditions above.  Then
\[
\left|
\log q_0-\log\widetilde q_0
\right|
\ll_k1.
\]
If $q_0=\widetilde q_0$, then
\[
\left|
\log q_i-\log\widetilde q_i
\right|
\ll_k\frac{W}{K}
\qquad
(1\leq i\leq s)
\]
and
\[
|J_j-\widetilde J_j|
\ll_k W
\qquad
(j\in J).
\]
Consequently, for fixed $q_0$, each $J_j$, $j\in J$, has
$O_k(W)$ possible integer values.
\end{lemma}

\begin{proof}
Let $n$ and $\widetilde n$ correspond to the two choices, and write
\[
d_t=r_tm_t,
\qquad
\widetilde d_t=\widetilde r_t\widetilde m_t
\qquad (1\leq t\leq k)
\]
as in \eqref{eq:external-medium-decomp}. Put
\[
\Delta_i=\log\frac{q_i}{\widetilde q_i}
\quad (0\leq i\leq s),
\qquad
\delta_j=\frac{J_j-\widetilde J_j}{K}
\quad (j\in J).
\]

By Lemma~\ref{lem:MOP}\textup{(ii)} and
\eqref{eq:large-squarefree}, the primes that vary between the two
choices occur with exponent one. All remaining prime powers are
fixed. Therefore,
\[
\frac{n}{\widetilde n}
=
\prod_{i=0}^s\frac{q_i}{\widetilde q_i}
\prod_{j\in J}\frac{p_{J_j}}{\widetilde p_{\widetilde J_j}}.
\]
Since $p_{J_j}\in\cP_{J_j}$ and
$\widetilde p_{\widetilde J_j}\in\cP_{\widetilde J_j}$, we have
\[
0<\log p_{J_j}-\frac{J_j}{K}\leq\frac1K,
\qquad
0<\log\widetilde p_{\widetilde J_j}
  -\frac{\widetilde J_j}{K}\leq\frac1K.
\]
Consequently,
\[
\left|
\log\frac{p_{J_j}}{\widetilde p_{\widetilde J_j}}-\delta_j
\right|\leq\frac1K.
\]
Taking logarithms in the expression for $n/\widetilde n$ and using
$n,\widetilde n\in(X/2,X]$, we obtain
\begin{equation}\label{eq:joint-size-difference}
\left|
\sum_{i=0}^s\Delta_i+\sum_{j\in J}\delta_j
\right|
\leq\log2+\frac{|J|}{K}
\ll_k1.
\end{equation}

Define
\[
\mathbf S=f\bigl(S_K(m_1),\ldots,S_K(m_k)\bigr),
\qquad
\widetilde{\mathbf S}
=f\bigl(S_K(\widetilde m_1),\ldots,S_K(\widetilde m_k)\bigr).
\]
The two choices represent the same class $L$, so
\eqref{eq:residual-identity} gives
\[
\mathbf S=f(L)+\sum_{j=1}^hJ_jv_j,
\qquad
\widetilde{\mathbf S}=f(L)+\sum_{j=1}^h\widetilde J_jv_j.
\]
Since $J_j=\widetilde J_j$ for $j\in F$ and
$\{1,\ldots,h\}=J\mathbin{\dot\cup}F$, it follows that
\begin{equation}\label{eq-4}
\mathbf S-\widetilde{\mathbf S}
=\sum_{j\in J}(J_j-\widetilde J_j)v_j.
\end{equation}

Put $\mathbf b_{\mathrm{ext}}
=\sum_{p\mid b}f(\alpha_{1,p},\ldots,\alpha_{k,p})\log p.$
The integer $b$ and the exponents $\alpha_{t,p}$ are fixed, and
\[
f\bigl(v_{q_i}(r_1),\ldots,v_{q_i}(r_k)\bigr)
=f\bigl(v_{\widetilde q_i}(\widetilde r_1),\ldots,
v_{\widetilde q_i}(\widetilde r_k)\bigr)=u_i
\quad(0\leq i\leq s).
\]
Unique prime factorization and the linearity of $f$ therefore give
\[
\begin{aligned}
f(\log r_1,\ldots,\log r_k)
=\mathbf b_{\mathrm{ext}}+\sum_{i=0}^su_i\log q_i,\qquad
f(\log\widetilde r_1,\ldots,\log\widetilde r_k)
=\mathbf b_{\mathrm{ext}}+\sum_{i=0}^su_i\log\widetilde q_i.
\end{aligned}
\]
Subtracting these identities yields
\begin{equation}\label{eq-5}
f(\log r_1,\ldots,\log r_k)
-f(\log\widetilde r_1,\ldots,\log\widetilde r_k)
=\sum_{i=0}^su_i\Delta_i.
\end{equation}

Define
\[
\begin{aligned}
\mathbf R
=\mathbf S+Kf(\log r_1,\ldots,\log r_k), \qquad
\widetilde{\mathbf R}
=\widetilde{\mathbf S}
  +Kf(\log\widetilde r_1,\ldots,\log\widetilde r_k).
\end{aligned}
\]
The inequalities
\[
d_1<\cdots<d_k,
\qquad
\frac{d_k}{d_1}\leq1+(\log n)^{-a}
\]
imply
\[
\max_{1\leq r,t\leq k}|\log d_r-\log d_t|
=\log\frac{d_k}{d_1}
\leq\log\bigl(1+(\log n)^{-a}\bigr)
\leq(\log n)^{-a}.
\]
For $X\geq4$, the inequalities $n>X/2$ and
$K=(\log X)^a$ give
\[
K\log\frac{d_k}{d_1}
\leq K(\log n)^{-a}\leq2^a.
\]
Thus Proposition~\ref{prop:AMR} and
$\Omega(m_t)\leq\Omega(n)\leq(\log\log X)^2$ yield
\[
\begin{aligned}
\left|
S_K(m_t)-S_K(m_1)+K\log\frac{r_t}{r_1}
\right|
&\leq K\log\frac{d_k}{d_1}+\Omega(m_t)+\Omega(m_1)\\
&\leq2^a+2(\log\log X)^2
\leq W
\end{aligned}
\]
for $2\leq t\leq k$, provided that $C_{a,k}$ is sufficiently
large. These are precisely the coordinate inequalities
$\mathbf R\in[-W,W]^{k-1}$. The same argument gives
$\widetilde{\mathbf R}\in[-W,W]^{k-1}$.

On the other hand, \eqref{eq-4} and \eqref{eq-5} imply
\[
\begin{aligned}
\mathbf R-\widetilde{\mathbf R}
=\sum_{j\in J}(J_j-\widetilde J_j)v_j
  +K\sum_{i=0}^su_i\Delta_i
=K\left(\sum_{i=0}^su_i\Delta_i
  +\sum_{j\in J}v_j\delta_j\right).
\end{aligned}
\]
Since $K>0$, we conclude that
\begin{equation}\label{eq:joint-target-difference}
\sum_{i=0}^su_i\Delta_i+\sum_{j\in J}v_j\delta_j
\in\left[-\frac{2W}{K},\frac{2W}{K}\right]^{k-1}.
\end{equation}

Write $J=\{j_1<\cdots<j_{h-q}\}$ and put
\[
d=s+h-q=\dim(E+U),
\]
where the last equality is \eqref{dim-E+U}. If $d=0$, then
$s=0$ and $J=\varnothing$. In this case,
\eqref{eq:joint-size-difference} gives $|\Delta_0|\ll_k1$,
and there are no further inequalities to prove. Assume henceforth
that $d\geq1$.

By Lemma~\ref{lem:pivot-basis}, the vectors
\[
u_1,\ldots,u_s,v_{j_1},\ldots,v_{j_{h-q}}
\]
form a basis of $E+U$. The matrix with these columns has rank $d$,
so it has a nonsingular $d\times d$ submatrix obtained by selecting
$d$ rows. Let
$\mathcal I=\{t_1<\cdots<t_d\}\subseteq\{1,\ldots,k-1\}$
be their indices. For $w\in\mathbb R^{k-1}$, write
\[
w_{\mathcal I}=(w_{t_1},\ldots,w_{t_d})^{\mathsf T},
\]
and define
\[
N=\bigl(
(u_1)_{\mathcal I},\ldots,(u_s)_{\mathcal I},
(v_{j_1})_{\mathcal I},\ldots,(v_{j_{h-q}})_{\mathcal I}
\bigr).
\]
By construction, $N$ is nonsingular. Moreover, the coordinate
projection $w\mapsto w_{\mathcal I}$ is injective on $E+U$,
if
\[
w=\sum_{i=1}^sa_iu_i+\sum_{\nu=1}^{h-q}b_\nu v_{j_\nu}
\quad\text{and}\quad w_{\mathcal I}=0,
\]
then
\[
N(a_1,\ldots,a_s,b_1,\ldots,b_{h-q})^{\mathsf T}=0,
\]
so all coefficients vanish and $w=0$.

Now define the $(d+1)\times(d+1)$ matrix
\[
M=\left(
\binom{1}{(u_0)_{\mathcal I}},\ldots,
\binom{1}{(u_s)_{\mathcal I}},
\binom{1}{(v_{j_1})_{\mathcal I}},\ldots,
\binom{1}{(v_{j_{h-q}})_{\mathcal I}}
\right).
\]
To prove that $M$ is nonsingular, suppose that
\[
\sum_{i=0}^sa_i\binom{1}{(u_i)_{\mathcal I}}
+\sum_{\nu=1}^{h-q}b_\nu\binom{1}{(v_{j_\nu})_{\mathcal I}}
=0.
\]
The first coordinate gives
$\sum_{i=0}^sa_i+\sum_{\nu=1}^{h-q}b_\nu=0$.
The remaining coordinates, together with the injectivity just
proved, give
\[
\sum_{i=0}^sa_iu_i+\sum_{\nu=1}^{h-q}b_\nu v_{j_\nu}=0.
\]
Hence
\[
\sum_{i=0}^sa_i\binom{1}{u_i}
+\sum_{\nu=1}^{h-q}b_\nu\binom{1}{v_{j_\nu}}=0.
\]
By \eqref{eq:joint-columns}, these columns are linearly
independent. Thus all coefficients vanish, proving that $M$ is
nonsingular.

The entries of $M$ and $N$ belong to $\{-1,0,1\}$, and their
orders are at most $k$. Their determinants are nonzero integers,
so their absolute values are at least one. A cofactor of a matrix
of order $r\leq k$ has absolute value at most $(r-1)!$, by the
determinant expansion. The adjugate formula therefore gives
\[
|(M^{-1})_{\ell m}|\leq(k-1)!,
\qquad
|(N^{-1})_{\ell m}|\leq(k-1)!
\]
for every entry of the respective matrices.

Put
\[
\mathbf x=
(\Delta_0,\ldots,\Delta_s,
\delta_{j_1},\ldots,\delta_{j_{h-q}})^{\mathsf T}.
\]
The first coordinate of $M\mathbf x$ is
$\sum_{i=0}^s\Delta_i+\sum_{j\in J}\delta_j$, and the
remaining coordinates form the vector
\[
\left(\sum_{i=0}^su_i\Delta_i
+\sum_{j\in J}v_j\delta_j\right)_{\mathcal I}.
\]
Thus \eqref{eq:joint-size-difference} and
\eqref{eq:joint-target-difference} imply
\[
M\mathbf x\in[-C_k,C_k]
\times\left[-\frac{2W}{K},\frac{2W}{K}\right]^d
\]
for some $C_k>0$ depending only on $k$. Applying $M^{-1}$ gives
\[
|\Delta_0|\ll_k1+\frac WK\ll_k1,
\]
where the last estimate uses $W/K=o(1)$. Equivalently,
\[
|\log q_0-\log\widetilde q_0|\ll_k1.
\]

If $q_0=\widetilde q_0$, then $\Delta_0=0$, and
\eqref{eq:joint-target-difference}, restricted to the coordinates
in $\mathcal I$, becomes
\[
N(\Delta_1,\ldots,\Delta_s,
\delta_{j_1},\ldots,\delta_{j_{h-q}})^{\mathsf T}
\in\left[-\frac{2W}{K},\frac{2W}{K}\right]^d.
\]
Applying $N^{-1}$ yields
\[
|\Delta_i|\ll_k\frac WK\quad(1\leq i\leq s),
\qquad
|\delta_j|\ll_k\frac WK\quad(j\in J).
\]
By the definitions of $\Delta_i$ and $\delta_j$, these estimates
are precisely
\[
|\log q_i-\log\widetilde q_i|\ll_k\frac WK
\quad(1\leq i\leq s),
\qquad
|J_j-\widetilde J_j|\ll_k W
\quad(j\in J).
\]
For fixed $q_0$, the latter inequality holds for every pair of
possible values of $J_j$. Those values therefore lie in an
interval of length $O_k(W)$, which contains $O_k(W)$ integers
because $W\geq1$.

Finally, in \eqref{eq:residual-identity}, the family
$(B'_\omega)_{\omega\in\{0,1\}^k}$ enters only through $L$.
All subsequent bounds are independent of the particular family
representing this class. The estimates are therefore uniform over
all such families.
\end{proof}

\subsection{Reciprocal sums over prime intervals}
Lemma~\ref{lem:joint-localization} bounds the variation of the
logarithms of the selected primes. We next estimate the sums
of reciprocals of primes in the corresponding intervals,
with bounds independent of their endpoints.

\begin{lemma}\label{lem:prime-windows}
Let $1<A<B$ and $I=(A,B]$.
\begin{enumerate}[label=\textup{(\roman*)},leftmargin=2.2em]
\item
For every fixed $C>0$, if
$\log\frac{B}{A}\leq C$, we have
$$\sum_{\substack{p\in I\cap(z,X]\\ p\ {\rm prime}}}\frac1p
\ll_C
\frac1{\log z}.$$

\item
If $z^{-1/2}\leq\Delta\leq1$ and
$\log\frac{B}{A}\leq\Delta$, then
$$\sum_{\substack{p\in I\cap(z,X]\\ p\ {\rm prime}}}\frac 1p
\ll \frac{\Delta}{\log z}.$$
\end{enumerate}

In both parts, the implied constants are independent of $A$ and
$B$.
\end{lemma}

\begin{proof}
If $I\cap(z,X]$ is empty, there is nothing to prove.  Put
$Y=\max\{A,z\}$.

For \textup{(i)}, the inequality $B\leq A\ee^C$ gives
$I\cap(z,X]\subseteq(Y,Y\ee^C]$.  By Mertens' theorem for the sum of reciprocals of primes,
\[
\sum_{p\leq x}\frac1p
=
\log\log x
+
O\left(\frac1{\log x}\right)
\qquad
(x\geq2).
\]
Hence
\[
\begin{aligned}
\sum_{Y<p\leq Y\ee^C}\frac1p
&=\sum_{p\leq Y\ee^C}\frac1p-\sum_{p\leq Y}\frac1p\\
&=
\log\log(Y\ee^C)-\log\log Y
+
O\left(\frac1{\log Y}\right)\\
&=
\log\left(1+\frac{C}{\log Y}\right)
+
O\left(\frac1{\log Y}\right)\\
&\ll_C
\frac1{\log Y}.
\end{aligned}
\]

For \textup{(ii)}, we have
$I\cap(z,X]\subseteq(Y,Y\ee^\Delta]$.  Put
$\eta=Y(\ee^\Delta-1)$, so that $Y\ee^\Delta=Y+\eta$.  Since
$0<\Delta\leq1$, the inequalities
$\ee^\Delta\geq1+\Delta$ and
$\ee^\Delta\leq1+(\ee-1)\Delta$ give
\[
Y\Delta
\leq
\eta
\leq
(\ee-1)Y\Delta.
\]
Moreover, since $Y\geq z$ and $\Delta\geq z^{-1/2}$,
\[
\eta
\geq
Y\Delta
\geq
Yz^{-1/2}
\geq
Y^{1/2}.
\]
Thus $\log\eta\geq\frac12\log Y$.
By the Brun--Titchmarsh inequality in its short-interval form
(see, for example, \cite[Theorem~4.2]{Buthe2014} with constant weight
and modulus $1$), we have
$$\#\{p\text{ prime}:Y<p\leq Y+\eta\}
\ll
\frac{\eta}{\log\eta}.$$
Using $\eta\leq(\ee-1)Y\Delta$ and
$\log\eta\geq\frac12\log Y$, we obtain
\[
\begin{aligned}
\#\{p\text{ prime}:Y<p\leq Y+\eta\}
\ll
\frac{\eta}{\log\eta}
\leq \frac{(\ee-1)Y\Delta}{\frac12\log Y}
\ll
\frac{Y\Delta}{\log Y}.
\end{aligned}
\]

Since $Y\ee^\Delta=Y+\eta$ and $p>Y$ for every prime occurring in
the sum, we obtain
\[
\begin{aligned}
\sum_{Y<p\leq Y\ee^\Delta}\frac1p
=\sum_{Y<p\leq Y+\eta}\frac1p
\leq
\frac1Y\#\{p:Y<p\leq Y+\eta\}
\ll \frac{\Delta}{\log Y}
\leq
\frac{\Delta}{\log z}.
\end{aligned}
\]
\end{proof}

By \eqref{eq:bin-parameters}, \eqref{eq:y-z-asymp}, and the choice
$W=C_{a,k}(1+(\log\log X)^2)$, for all sufficiently large $X$,
\[
z^{-1/2}
\leq
C_k\frac{W}{K}
\leq
1,
\qquad
K\log z=D+1,
\qquad
W=D^{o(1)}.
\]

Every selected prime $q_i$ lies in $(z,X]$.  By
Lemma~\ref{lem:joint-localization}, the possible values of
$\log q_0$ lie in an interval of length $O_k(1)$.  Hence
Lemma~\ref{lem:prime-windows}\textup{(i)} gives
\begin{equation}\label{eq:broad-weight}
\sum_{q_0}\frac1{q_0}
\ll_k
\frac1{\log z}.
\end{equation}

After $q_0$ is fixed, the possible values of $\log q_i$ lie in an
interval of length $O_k(W/K)$ for $1\leq i\leq s$.  Applying
Lemma~\ref{lem:prime-windows}\textup{(ii)} with
$\Delta=C_kW/K$, we obtain
\begin{equation}\label{eq:narrow-weight}
\sum_{q_i}\frac1{q_i}
\ll_k
\frac{W}{K\log z}
=
D^{-1+o(1)}
\qquad
(1\leq i\leq s),
\end{equation}
where $q_0$ is fixed in each sum.

Recall that $R_\ell=\sum_{p\in\cP_\ell}p^{-1}$ and
$R_\ell\ll\ell^{-1}$ by \eqref{eq:R-i-asymp}.  If $j\in J$, then
Lemma~\ref{lem:joint-localization} gives $O_k(W)$ possible values of
$J_j$, while \eqref{eq:residual-pivot-scale} gives
$\ee^{-1}D^{c_j}<J_j\leq D^{c_j}$.  Thus
$R_{J_j}\ll D^{-c_j}$ and
\begin{equation}\label{eq:constrained-medium-weight}
\sum_{J_j}R_{J_j}
\ll_k
WD^{-c_j}
=
D^{-c_j+o(1)}
\qquad
(j\in J).
\end{equation}

If $j\in F$, then
\begin{equation}\label{eq:free-medium-weight}
\sum_{\substack{i_1\leq\ell\leq D\\
\ee^{-1}D^{c_j}<\ell\leq D^{c_j}}}R_\ell
\ll
\sum_{\ee^{-1}D^{c_j}<\ell\leq D^{c_j}}\frac1\ell
\ll
1.
\end{equation}

We now combine these estimates with all dependencies specified.
Fix $\Lambda$ as in \eqref{eq:Lambda-definition}, together with
$E,\mathbf u,J$, a set $A'$, a class $L$, and primes
$p_i\in\cP_i$ for $i\in A'$. Also fix a positive integer
$b\in\mathcal S_X$ and integers
\[
0\leq\alpha_{t,p}\leq v_p(b)
\qquad(p\mid b,\ 1\leq t\leq k).
\]
The integer $b$ and these exponents specify all external prime powers
other than the selected primes $q_0,\ldots,q_s$.
Let $\mathcal T_L$ be the set of tuples
\[
\bigl(q_0,\ldots,q_s,(J_j,p_{J_j})_{j=1}^h\bigr)
\]
that occur for some integer and selected divisor tuple satisfying
the restrictions of the preceding subsection with these fixed data.
In particular, the $q_i$ are distinct, $q_i\nmid b$, and each such
tuple is included once, irrespective of how many partitions
$(B'_\omega)_\omega$ represent $L$.

First, also fix $(J_j,p_{J_j})$ for $j\in F$.
Lemma~\ref{lem:joint-localization} bounds the diameter of the possible
values of $\log q_0$ by $O_k(1)$, uniformly in all the fixed data,
and \eqref{eq:broad-weight} bounds the corresponding reciprocal sum.
For each fixed $q_0$, the tuples of
$q_1,\ldots,q_s$ and $(J_j,p_{J_j})$, $j\in J$, form a subset
of the product of their coordinate projections. Since all weights
are nonnegative, \eqref{eq:narrow-weight} and \eqref{eq:constrained-medium-weight} give an upper bound for their total reciprocal weight that is uniform in $q_0$.
Summing over $q_0$ with weight $1/q_0$ and applying \eqref{eq:broad-weight}, we obtain
\[
\sum
\left(\prod_{i=0}^s\frac1{q_i}\right)
\left(\prod_{j\in J}\frac1{p_{J_j}}\right)
\ll_k
\frac1{\log z}
\left(\frac W{K\log z}\right)^s
\prod_{j\in J}\bigl(WD^{-c_j}\bigr),
\]
where the sum is over the tuples compatible with the fixed free pairs.
If that set is empty, its sum is zero and the same bound holds.
The constant is independent of the locations of all the intervals.

Now sum over $(J_j,p_{J_j})$ for $j\in F$. By
\eqref{eq:free-medium-weight}, their total factor is at most
\[
\prod_{j\in F}
\left(\sum_{\substack{i_1\leq\ell\leq D\\
\ee^{-1}D^{c_j}<\ell\leq D^{c_j}}}R_\ell\right)
\ll_k1.
\]
Choose a constant $C_k^*>0$ large enough to dominate these uniform
bounds, and define
\begin{equation}\label{eq:C-pivot}
\begin{aligned}
C_{\Lambda,E,J}
&=
\frac{C_k^*}{\log z}
\left(\frac W{K\log z}\right)^s
\prod_{j\in J}\bigl(WD^{-c_j}\bigr)
\ll_k
\frac{D^{-s-\sum_{j\in J}c_j+o(1)}}{\log z}.
\end{aligned}
\end{equation}
Then, for every choice of the fixed residual data,
\begin{equation}\label{eq:uniform-prime-sum}
\sum_{\mathcal T_L}
\left(\prod_{i=0}^s\frac1{q_i}\right)
\left(\prod_{j=1}^h\frac1{p_{J_j}}\right)
\leq C_{\Lambda,E,J}.
\end{equation}
Thus $C_{\Lambda,E,J}$ is a uniform upper bound, not a sum depending
on $A'$, $L$, $b$, or the unselected prime powers.
The bound is also uniform in $\mathbf u$, because all its vectors
belong to the fixed finite set $f(\{0,1\}^k)$.
The last inequality in \eqref{eq:C-pivot} uses
$K\log z=D+1$, $W=D^{o(1)}$, and
$s+|J|=\dim(E+U)\leq k-1$.

\subsection{Weighted reconstruction}\label{subsec:weighted-reconstruction}

For fixed $\Lambda$ as in \eqref{eq:Lambda-definition} and
$A'\subseteq[i_1,D]\cap\Z$, put
\[
\cL_\Lambda(A')=\cL_{\cV,\mathbf c,\boldsymbol\mu}(A'),
\]
using the definition in \eqref{eq:partition-compatibility}.
By Lemma~\ref{lem:residual-interface}\textup{(ii)}, every residual
partition constructed above satisfies these conditions.
For $A'\in\widetilde{\mathcal R}$,
Lemma~\ref{lem:residual-sum-bound} gives
\begin{equation}\label{eq:L-entropy-bound}
|\cL_\Lambda(A')|
\leq D^{e_*(\Lambda)+o(1)},
\qquad
e_*(\Lambda)=\min_{\cV'\leq\cV}
\mathrm e(\cV',\mathbf c,\boldsymbol\mu),
\end{equation}
where $o(1)=O_k((\log D)^{-1/4})$ is uniform in $\Lambda$ and $A'$.

Fix $E,u_0,\ldots,u_s$ and $J$ as in
Lemma~\ref{lem:pivot-basis}, put
$\mathbf u=(u_0,\ldots,u_s)$, and let
$\mathscr S_{\Lambda,E,\mathbf u,J}$ be the set of integers
$n\in(X/2,X]$ satisfying both conclusions of
Lemmas~\ref{lem:MOP} and \ref{lem:external-good},
$\Omega(n)\leq(\log\log X)^2$, and $I(n)\in\mathcal R$, for which
the selected divisor tuple has pairwise distinct medium parts,
$Q>X^{1-\eps}$, and data $\Lambda,E,\mathbf u,J$ as above.

We now sum over the residual data using the uniform bound
\eqref{eq:uniform-prime-sum}. The following lemma bounds the
reciprocal sum over $\mathscr S_{\Lambda,E,\mathbf u,J}$
in terms of the residual classes, and then applies the entropy
estimate \eqref{eq:L-entropy-bound}.

\begin{lemma}
With the notation above,
\begin{align}
\sum_{n\in\mathscr S_{\Lambda,E,\mathbf u,J}}\frac1n
&\ll_{a,k}
\mathcal Z_{\rm ext}^{(k)}C_{\Lambda,E,J}
\sum_{A'\in\widetilde{\mathcal R}}
\left(\prod_{i\in A'}R_i\right)
|\cL_\Lambda(A')|,
\label{eq:master-one}\\
&\ll_{a,k}
D^{e_*(\Lambda)+o(1)}
\mathcal Z_{\rm ext}^{(k)}Z_R C_{\Lambda,E,J}.
\label{eq:master-two}
\end{align}
\end{lemma}

\begin{proof}
For each $n\in\mathscr S_{\Lambda,E,\mathbf u,J}$, retain its
selected divisor tuple. By Lemma~\ref{lem:MOP}, each occupied
interval $\cP_i$, $i\in I(n)$, contains a unique prime $p_i$
dividing $n$, and $v_{p_i}(n)=1$.
The selected vectors $\omega^1,\ldots,\omega^h$ are distinct,
and the sets $B_\omega$ are pairwise disjoint. Thus the indices
$J_j=\max B_{\omega^j}$ are pairwise distinct, and
\[
I(n)=A'\mathbin{\dot\cup}\{J_1,\ldots,J_h\}.
\]
Since $I(n)\in\mathcal R$, Lemma~\ref{lem:residual-interface}
gives
\[
A'\in\widetilde{\mathcal R},
\qquad
L\in\cL_\Lambda(A'),
\]
where $L$ is the class defined in \eqref{eq:canonical-L}.

For $0\leq i\leq s$, choose $q_i$ to be the least prime $q>z$
such that $q\mid Q$ and $u(q)=u_i$. These primes exist by
Lemma~\ref{lem:pivot-basis}. The linear independence of
$\binom{1}{u_0},\ldots,\binom{1}{u_s}$ implies that the vectors
$u_i$ are pairwise distinct, and hence so are the primes $q_i$.
Moreover, $q_i\mid Q\mid n$, so \eqref{eq:large-squarefree}
gives $v_{q_i}(n)=1$.

Put
\[
b=\prod_{\substack{p\mid n,\ p\ \mathrm{external}\\
p\notin\{q_0,\ldots,q_s\}}}p^{v_p(n)}.
\]
Then $b\in\mathcal S_X$, $q_i\nmid b$ for every $i$, and unique
prime factorization gives
\begin{equation}\label{eq:reconstruction-factorization}
n=b\left(\prod_{i\in A'}p_i\right)
\left(\prod_{j=1}^hp_{J_j}\right)
\left(\prod_{i=0}^sq_i\right).
\end{equation}
In particular, the quantities on the right determine $n$ uniquely.

For each $p\mid b$, put
\[
\alpha_{t,p}=v_p(r_t)
\qquad(1\leq t\leq k).
\]
Since $r_t\mid n$, these exponents satisfy
$0\leq\alpha_{t,p}\leq v_p(b)$. For fixed $b$, the number of
possible collections $(\alpha_{t,p})_{p\mid b,\,1\leq t\leq k}$
is therefore at most $\prod_{p\mid b}(v_p(b)+1)^k.$
At a selected prime $q_i$, the exponent vector
\[
\boldsymbol\varepsilon^{(i)}
=\bigl(v_{q_i}(r_1),\ldots,v_{q_i}(r_k)\bigr)
\]
belongs to $\{0,1\}^k$ and satisfies
$f(\boldsymbol\varepsilon^{(i)})=u_i$. Since
\[
\varepsilon_t^{(i)}
=\varepsilon_1^{(i)}+(u_i)_{t-1}
\qquad(2\leq t\leq k),
\]
the first coordinate determines all the others. Consequently,
\[
\#\{\boldsymbol\varepsilon\in\{0,1\}^k:
f(\boldsymbol\varepsilon)=u_i\}\leq2,
\]
and there are at most $2^{s+1}\leq2^k$ possible collections of
the selected exponent vectors.

Now fix $A'$, $L\in\cL_\Lambda(A')$, the primes $p_i$ for
$i\in A'$, the integer $b$, and the exponents $\alpha_{t,p}$.
Let $\mathcal T_L$ be the set of admissible tuples
\[
\bigl(q_0,\ldots,q_s,(J_j,p_{J_j})_{j=1}^h\bigr)
\]
with these fixed data, as in \eqref{eq:uniform-prime-sum}.
That estimate gives
\[
\sum_{\mathcal T_L}
\left(\prod_{i=0}^s\frac1{q_i}\right)
\left(\prod_{j=1}^h\frac1{p_{J_j}}\right)
\leq C_{\Lambda,E,J}.
\]
The same bound holds after restricting to any fixed collection
of selected exponent vectors. Summing over those collections
introduces at most the factor $2^k$.

The set $\mathcal T_L$ contains each tuple once, including when
the same tuple arises from several partitions representing $L$.
Lemma~\ref{lem:joint-localization} and the resulting estimate
\eqref{eq:uniform-prime-sum} hold uniformly over all such
partitions. Thus the summation over $L$ introduces the factor
$|\cL_\Lambda(A')|$, without any additional factor for its
representing partitions.

All admissibility conditions, including $X/2<n\leq X$ and
$q_i\nmid b$, are retained when applying
\eqref{eq:uniform-prime-sum}. After this uniform bound has been
applied, nonnegativity allows us to enlarge the outer sums to
all $p_i\in\cP_i$, all $b\in\mathcal S_X$, and all exponents
$0\leq\alpha_{t,p}\leq v_p(b)$.
Every integer on the left has a collection of the specified data,
and \eqref{eq:reconstruction-factorization} shows that a fixed
collection cannot represent two different integers. Therefore,
\[
\begin{aligned}
\sum_{n\in\mathscr S_{\Lambda,E,\mathbf u,J}}\frac1n
&\ll_{a,k} C_{\Lambda,E,J}
\sum_{A'\in\widetilde{\mathcal R}}
\sum_{L\in\cL_\Lambda(A')}
\left(
\sum_{\substack{p_i\in\cP_i\\i\in A'}}
\prod_{i\in A'}\frac1{p_i}
\right)\left(
\sum_{b\in\mathcal S_X}\frac1b
\prod_{p\mid b}(v_p(b)+1)^k
\right).
\end{aligned}
\]

By \eqref{eq:R-i-def},
\[
\sum_{\substack{p_i\in\cP_i\\i\in A'}}
\prod_{i\in A'}\frac1{p_i}
=\prod_{i\in A'}\left(\sum_{p\in\cP_i}\frac1p\right)
=\prod_{i\in A'}R_i.
\]
Also, unique prime factorization and nonnegativity give
\begin{equation}\label{eq:residual-external-sum}
\begin{aligned}
\sum_{b\in\mathcal S_X}\frac1b
\prod_{p\mid b}(v_p(b)+1)^k
&=\prod_{\substack{p\leq X\\p\ \mathrm{external}}}
\left(\sum_{\nu=0}^{\infty}\frac{(\nu+1)^k}{p^\nu}\right)
=\mathcal Z_{\rm ext}^{(k)},
\end{aligned}
\end{equation}
where the last equality is \eqref{eq:Zext-harmonic}.
Substitution yields
\[
\sum_{n\in\mathscr S_{\Lambda,E,\mathbf u,J}}\frac1n
\ll_{a,k}
\mathcal Z_{\rm ext}^{(k)}C_{\Lambda,E,J}
\sum_{A'\in\widetilde{\mathcal R}}
\left(\prod_{i\in A'}R_i\right)|\cL_\Lambda(A')|,
\]
which proves \eqref{eq:master-one}.
Finally, \eqref{eq:L-entropy-bound} gives
\[
|\cL_\Lambda(A')|\leq D^{e_*(\Lambda)+o(1)}
\qquad(A'\in\widetilde{\mathcal R}),
\]
with $o(1)$ uniform in $A'$. Hence
\[
\begin{aligned}
\sum_{A'\in\widetilde{\mathcal R}}
\left(\prod_{i\in A'}R_i\right)|\cL_\Lambda(A')|
&\leq D^{e_*(\Lambda)+o(1)}
\sum_{A'\subseteq[i_1,D]\cap\Z}\prod_{i\in A'}R_i\\
&=D^{e_*(\Lambda)+o(1)}\prod_{i=i_1}^D(1+R_i)\\
&=D^{e_*(\Lambda)+o(1)}Z_R.
\end{aligned}
\]
Combining this estimate with \eqref{eq:master-one} proves
\eqref{eq:master-two}.
\end{proof}

\subsection{Irregular medium occupancy}
The preceding estimate concerns integers with $I(n)\in\mathcal R$.
We next show that, among the integers under consideration,
those with $I(n)\notin\mathcal R$ contribute $o(1)$ to the
reciprocal sum.

\begin{lemma}\label{lem:irregular-saturated}
Let $\eps>0$ be sufficiently small in terms of $k$.
Let $\mathscr T$ be the set of integers $n\in(X/2,X]$
satisfying the conclusions of Lemmas~\ref{lem:MOP} and
\ref{lem:external-good}, with $I(n)\notin\mathcal R$, for which
there exist divisors $d_1<\cdots<d_k\mid n$ having pairwise
distinct medium parts and satisfying
$d_k\leq d_1\bigl(1+(\log n)^{-a}\bigr)$ and $Q>X^{1-\eps},$
where $Q$ is defined in \eqref{eq:Q-def}. Then
\[
\sum_{n\in\mathscr T}\frac1n=o(1)
\qquad
\text{as }X\to\infty.
\]
\end{lemma}

\begin{proof}
Fix $n\in\mathscr T$ and a divisor tuple satisfying the conditions in the statement.
We first show that $Q$ has a prime divisor exceeding $z$.
Otherwise, since every prime divisor of $Q$ is external,
every such prime would be at most $y$. As $Q\mid n$,
\eqref{eq:small-external-good} would give
\[
Q
\leq
\prod_{p\leq y}p^{v_p(n)}
\leq
X^{\eps/4}
<
X^{1-\eps},
\]
contrary to the assumption on $Q$. 

Choose a prime $q>z$ dividing $Q$. Since $q\mid Q\mid n$,
condition \eqref{eq:large-squarefree} gives $v_q(n)=1$.
Put $A=I(n)$. By Lemma~\ref{lem:MOP}, for each $i\in A$
there is a unique prime $p_i\in\cP_i$ dividing $n$, and
$v_{p_i}(n)=1$. Define
$m=\prod_{i\in A}p_i$ and $b=\frac{n}{qm}.$
Then $b$ is a positive integer and
$q\nmid b.$
The integer $m$ contains all medium prime factors of $n$.
Consequently, every prime divisor of $b$ is external and
does not exceed $X$.

Recall from the proof of Lemma~\ref{lem:EPS} that
$\mathcal S_X$ is the set of positive integers supported on
external primes not exceeding $X$, and that
\[
Z_X
=
\sum_{b\in\mathcal S_X}\frac1b
=
\prod_{\substack{p\leq X\\p\ \mathrm{external}}}
\left(1-\frac1p\right)^{-1}.
\]
In particular, the integer $b$ defined above belongs to
$\mathcal S_X$.

For fixed $A$, primes $p_i\in\cP_i$ with $i\in A$, and
$b\in\mathcal S_X$, the condition $X/2<bqm\leq X$ gives
$\frac{X}{2bm}<q\leq\frac{X}{bm}.$
The interval $\left(\frac{X}{2bm},\frac{X}{bm}\right]\cap(z,X],$
if nonempty, has a lower endpoint at least $z$ and a ratio of
upper to lower endpoint at most $2$. Hence
Lemma~\ref{lem:prime-windows}\textup{(i)}, with $C=\log2$,
gives
\[
\sum_{\substack{
q>z,\ q\ \mathrm{prime}\\
X/(2bm)<q\leq X/(bm)
}}
\frac1q
\ll
\frac{\log2}{\log z}.
\]

With $m=\prod_{i\in A}p_i$, the preceding factorization gives
\[
\mathscr T
\subseteq
\bigcup_{\substack{
A\subseteq[i_1,D]\cap\Z\\
A\notin\mathcal R
}}
\bigcup_{\substack{p_i\in\cP_i\\i\in A}}
\left\{
bqm
\,\middle|\,
\begin{gathered}
b\in\mathcal S_X,\quad q>z,\quad q\ \mathrm{prime},
X/2<bqm\leq X
\end{gathered}
\right\},
\]
where each $p_i$ is prime. Hence, by
Lemma~\ref{lem:prime-windows}\textup{(i)},
\[
\begin{aligned}
\sum_{n\in\mathscr T}\frac1n
&\leq
\sum_{\substack{
A\subseteq[i_1,D]\cap\Z\\
A\notin\mathcal R
}}
\sum_{\substack{p_i\in\cP_i\\i\in A}}
\sum_{b\in\mathcal S_X}
\sum_{\substack{
q>z,\ q\ \mathrm{prime}\\
X/2<bqm\leq X
}}
\frac1{bqm}\\
&=
\sum_{\substack{
A\subseteq[i_1,D]\cap\Z\\
A\notin\mathcal R
}}
\sum_{\substack{p_i\in\cP_i\\i\in A}}
\left(\prod_{i\in A}\frac1{p_i}\right)
\sum_{b\in\mathcal S_X}\frac1b
\sum_{\substack{
q>z,\ q\ \mathrm{prime}\\
X/(2bm)<q\leq X/(bm)
}}
\frac1q\\
&\ll
\frac{\log2}{\log z}
\left(\sum_{b\in\mathcal S_X}\frac1b\right)
\sum_{\substack{
A\subseteq[i_1,D]\cap\Z\\
A\notin\mathcal R
}}
\prod_{i\in A}
\left(\sum_{\substack{p\in\cP_i\\p\ \mathrm{prime}}}\frac1p\right)\\
&=
\frac{\log2 \; Z_X}{\log z}
\sum_{\substack{
A\subseteq[i_1,D]\cap\Z\\
A\notin\mathcal R
}}
\prod_{i\in A}R_i.
\end{aligned}
\]
The last equality follows from
$$
\sum_{\substack{p_i\in\cP_i\\i\in A}}\prod_{i\in A}\frac1{p_i}
=\prod_{i\in A}\left(\sum_{p\in\cP_i}\frac1p\right)
=\prod_{i\in A}R_i.
$$
Applying \eqref{eq:irregular-medium-weight}, we obtain
\begin{equation}\label{eq:irregular-bound}
\sum_{n\in\mathscr T}\frac1n
\ll_{a,k}
\frac{Z_RZ_X}{\log z}
\left(
\exp\left\{-\frac14\sqrt{\log D}\right\}
+\frac1{i_1}
\right).
\end{equation}

As in the proof of Lemma~\ref{lem:EPS}, Mertens' product
theorem gives
\[
\begin{aligned}
Z_X
&=
\prod_{p\leq y}
\left(1-\frac1p\right)^{-1}
\prod_{z<p\leq X}
\left(1-\frac1p\right)^{-1}
\ll
(\log y)\frac{\log X}{\log z}.
\end{aligned}
\]
Moreover, Lemma~\ref{lem:harmonic-normalization} gives
$Z_R=(1+o(1))D/i_1$. Since
$\log y=\frac{i_1}{K}$ and $\log z=\frac{D+1}{K},$
we have
\[
\frac{Z_R\log y}{\log z}
=
(1+o(1))
\frac{D}{i_1}
\frac{i_1}{D+1}
=
1+o(1).
\]
Consequently,
\[
\frac{Z_RZ_X}{\log z}
\ll
\frac{Z_R\log y}{\log z}
\frac{\log X}{\log z}
\ll
\frac{\log X}{\log z}
\ll
\log\log\log X,
\]
where the last inequality follows from
\eqref{eq:bin-parameters} and $\log z=(D+1)/K$.

Substituting this bound into \eqref{eq:irregular-bound} yields
\[
\sum_{n\in\mathscr T}\frac1n
\ll_{a,k}
(\log\log\log X)
\left(
\exp\left\{-\frac14\sqrt{\log D}\right\}
+\frac1{i_1}
\right).
\]
Finally, \eqref{eq:bin-parameters} gives
$\log D=(1+a+o(1))\log\log X$ and $i_1^{-1}=D^{-a/(1+a)+o(1)}.$
Since $a>0$ is fixed, it follows that
\[
(\log\log\log X)
\exp\left\{-\frac14\sqrt{\log D}\right\}
=o(1),
\qquad
\frac{\log\log\log X}{i_1}=o(1).
\]
Therefore, $\sum_{n\in\mathscr T}\frac1n=o(1),$
as required.
\end{proof}

\subsection{The reciprocal-sum estimate}
We now combine the estimates for $I(n)\in\mathcal R$ and $I(n)\notin\mathcal R$, together with the bounds for the
excluded integers. This gives the required reciprocal-sum
estimate in the range $Q>X^{1-\eps}$ and hence an $o(X)$
bound for the number of such integers.

\begin{lemma}\label{lem:saturated}
Let $\eps>0$ be sufficiently small in terms of $k$.
Let $\mathscr S$ be the set of integers $n\in(X/2,X]$ for which
there exist divisors $d_1<\cdots<d_k\mid n$ having pairwise
distinct medium parts and satisfying
$d_k\leq d_1\bigl(1+(\log n)^{-a}\bigr)$ and $Q>X^{1-\eps},$
where $Q$ is defined in \eqref{eq:Q-def}. Then
\[
\sum_{n\in\mathscr S}\frac1n=o(1),
\qquad
\#\mathscr S=o(X)
\qquad
\text{as }X\to\infty.
\]
\end{lemma}

\begin{proof}
Let $\mathscr B$ be the set of integers $n\in(X/2,X]$ for
which at least one conclusion of Lemmas~\ref{lem:MOP} and
\ref{lem:external-good} fails, or
$\Omega(n)>(\log\log X)^2.$
These two lemmas and \eqref{eq:Omega-good} give
$\#\mathscr B=o(X)$.

Let $\mathscr T$ be the set defined in Lemma~\ref{lem:irregular-saturated}, and put
$$\mathscr S'=\left\{
n\in\mathscr S\setminus\mathscr B:
I(n)\in\mathcal R
\right\}.$$
If $n\in\mathscr S\setminus\mathscr B$ and
$I(n)\notin\mathcal R$, then $n\in\mathscr T$. Hence
$\mathscr S \subseteq
\mathscr S'\cup\mathscr T\cup\mathscr B.$
Since Lemma~\ref{lem:irregular-saturated} gives
$\sum_{n\in\mathscr T}n^{-1}=o(1)$, it remains to estimate
the sum over $\mathscr S'$.

For every $n\in\mathscr S'$, take the selected divisor tuple
among those satisfying the conditions in the statement.
The preceding construction determines $\Lambda$ and $E$,
and Lemma~\ref{lem:pivot-basis} provides $\mathbf u$ and $J$
such that $n\in\mathscr S_{\Lambda,E,\mathbf u,J}$. Therefore, we have 
\[
\mathscr S'
\subseteq
\bigcup_{\Lambda,E,\mathbf u,J}
\mathscr S_{\Lambda,E,\mathbf u,J},
\]
where the union ranges over all choices arising from this
construction.

Fix one such choice of $\Lambda,E,\mathbf u,J$.
The complete system $(\cV,\mathbf c,\boldsymbol\mu)$ contained
in $\Lambda$ has final threshold $c_{h+1}=\widehat c_X$.
By \eqref{eq:cX-supercritical},
\[
c_{h+1}=\widehat c_X\geq\beta_k+2\delta
\]
for all sufficiently large $X$.
Corollary~\ref{cor:uniform-gap} and the definition of
$e_*(\Lambda)$ in \eqref{eq:L-entropy-bound} give a constant
$\eta=\eta(a,k)>0$, independent of $X$ and $\Lambda$, such that
\begin{equation}\label{eq:e-star-gap}
e_*(\Lambda)
\leq
\mathrm e(\cV,\mathbf c,\boldsymbol\mu)-\eta
=
\sum_{j=1}^h c_j-\eta.
\end{equation}
The equality follows from \eqref{eq:full-e}, since
$\dim(V_j/V_{j-1})=1$ for every $1\leq j\leq h$.

By Lemma~\ref{lem:harmonic-normalization} and
\eqref{eq:bin-parameters},
\begin{align}
\frac{Z_R}{\log z}
&=
(1+o(1))
\frac{D}{i_1}
\frac{K}{D+1}
\notag
=(1+o(1))\frac{K}{i_1}
\notag\\
&=
(1+o(1))(\log\log X)^{-3}=
D^{o(1)}.
\label{eq:normalization-cancel}
\end{align}

Combining \eqref{eq:master-two} and \eqref{eq:C-pivot} with
Lemma~\ref{lem:harmonic-normalization} and
\eqref{eq:normalization-cancel}, we obtain
\[
\begin{aligned}
\sum_{n\in\mathscr S_{\Lambda,E,\mathbf u,J}}\frac1n
&\ll_{a,k}
D^{e_*(\Lambda)+o(1)}
\mathcal Z_{\rm ext}^{(k)}Z_R C_{\Lambda,E,J}\\
&\ll_{a,k}
\mathcal Z_{\rm ext}^{(k)}
\frac{Z_R}{\log z}
D^{e_*(\Lambda)-s-\sum_{j\in J}c_j+o(1)}\\
&=
D^{e_*(\Lambda)-s-\sum_{j\in J}c_j+o(1)},
\end{aligned}
\]
where the last equality uses
$\mathcal Z_{\rm ext}^{(k)}=D^{o(1)}$ and
$Z_R/\log z=D^{o(1)}$.

Recall that $s=\dim E,q=\dim(E\cap U)$, and $F=\{1,\ldots,h\}\setminus J.$ By Lemma~\ref{lem:pivot-basis},
$|J|=h-q$ and $|F|=q\leq s.$

By Definition~\ref{def:system}\textup{(b)}, $1\geq c_1\geq\cdots\geq c_{h+1}\geq0.$
Hence $c_j\leq1$ for every $1\leq j\leq h$, and
\eqref{eq:e-star-gap} implies
\begin{equation}\label{eq:dimension-exponent}
\begin{aligned}
e_*(\Lambda)-s-\sum_{j\in J}c_j
&\leq \sum_{j=1}^h c_j-s-\sum_{j\in J}c_j-\eta
=\sum_{j\in F}c_j-s-\eta\\
&\leq
|F|-s-\eta\\
&\leq-\eta.
\end{aligned}
\end{equation}
Thus
\[
\sum_{n\in\mathscr S_{\Lambda,E,\mathbf u,J}}\frac1n
\ll_{a,k}
D^{-\eta+o(1)}.
\]

The estimate following \eqref{eq:Lambda-definition} gives
at most $(\log D)^{O_k(1)}$ possible choices of $\Lambda$. For each fixed $\Lambda$, there are $O_k(1)$ possible choices
of $(E,\mathbf u,J)$. Indeed, Lemma~\ref{lem:pivot-basis} gives
\[
E=\Span\{u_1,\ldots,u_s\},
\qquad
u_0,\ldots,u_s\in f(\{0,1\}^k),
\qquad
J\subseteq\{1,\ldots,h\}.
\]
Since $s,h\leq k-1$ and $\#f(\{0,1\}^k)\leq2^k$,
the assertion follows.

All the preceding estimates have constants depending only on
$a,k,\eps$ and error terms uniform in $\Lambda,E,\mathbf u,J$.
Summing over these choices, we obtain
\[
\begin{aligned}
\sum_{n\in\mathscr S'}\frac1n
&\leq
\sum_{\Lambda,E,\mathbf u,J}
\sum_{n\in\mathscr S_{\Lambda,E,\mathbf u,J}}\frac1n\\
&\ll_{a,k}
(\log D)^{O_k(1)}D^{-\eta+o(1)}
=D^{-\eta+o(1)}
=o(1).
\end{aligned}
\]
Here we used
$(\log D)^{O_k(1)}=D^{o(1)}$.

Finally, using
$\mathscr S\subseteq\mathscr S'\cup\mathscr T\cup\mathscr B$
and $1/n\leq2/X$ on $(X/2,X]$, we obtain
\[
\begin{aligned}
\sum_{n\in\mathscr S}\frac1n
&\leq
\sum_{n\in\mathscr S'}\frac1n
+
\sum_{n\in\mathscr T}\frac1n
+
\sum_{n\in\mathscr B}\frac1n\\
&\leq
\sum_{n\in\mathscr S'}\frac1n
+
\sum_{n\in\mathscr T}\frac1n
+
\frac{2}{X}\#\mathscr B\\
&=
o(1)+o(1)+\frac{2}{X}\,o(X)
=o(1).
\end{aligned}
\]
Since $n\leq X$ for every $n\in\mathscr S$,
\[
\#\mathscr S
\leq
X\sum_{n\in\mathscr S}\frac1n
=o(X).
\]
This completes the proof.
\end{proof}

\section{Proof of the main theorem}\label{sec:main-proof}
The lower bound \eqref{eq:fgk-lower} reduces the desired equality
to proving $\alpha_k\leq\beta_k/(1-\beta_k)$.
We establish the stronger assertion that $\mathcal E_{a,k}$ has natural density zero for every
$a>\beta_k/(1-\beta_k)$.

\begin{proof}[Proof of Theorem~\ref{thm:main}]
Fix $k\geq2$ and $a$ satisfying \eqref{eq:supercritical-a}.
By the reduction preceding \eqref{eq:a-less-half}, it suffices
first to consider $0<a<1/2$. Choose $\eps>0$, depending only
on $k$, sufficiently small that Lemmas~\ref{lem:nonsat} and
\ref{lem:saturated} both apply. Keep $k$, $a$, and $\eps$ fixed.
Let $\mathcal E_{a,k}$ be the set defined in
\eqref{eq:E-ak-definition}.

\begin{claim}\label{clm:main-dyadic}
For $X\to\infty$, we have
\[
\#\bigl(\mathcal E_{a,k}\cap(X/2,X]\bigr)=o(X).
\]
\end{claim}

\begin{proof}
Use the parameters in \eqref{eq:bin-parameters} and the
corresponding medium and external parts. Put
\[
\delta_X=(\log X)^{-a},
\qquad
\mathscr B=\mathcal B_X(2^a)\cap(X/2,X],
\]
where $\mathcal B_X(2^a)$ is defined in
Lemma~\ref{lem:EPS}. Thus $n\in\mathcal B_X(2^a)$ if and only
if $n\leq X$ has distinct divisors $r<t$, both supported on
external primes, such that
$t\leq(1+2^a\delta_X)r.$
Since $a$ is fixed, it follows that Lemma~\ref{lem:EPS} applies with the fixed
constant $C=2^a$ and gives
\[
\#\mathscr B\leq\#\mathcal B_X(2^a)=o(X).
\]

For $X\geq4$ and $n\in(X/2,X]$, we have
\[
\log n>\log X-\log2\geq\frac12\log X.
\]
Since $a>0$, it follows that
\[
(\log n)^{-a}
\leq\left(\frac12\log X\right)^{-a}
=2^a\delta_X.
\]
Now fix
$n\in
\bigl(\mathcal E_{a,k}\cap(X/2,X]\bigr)\setminus\mathscr B.$
By the definition of $\mathcal E_{a,k}$, choose divisors
$d_1<\cdots<d_k\mid n$ satisfying $d_k\leq d_1\bigl(1+(\log n)^{-a}\bigr).$

Write
\[
d_s=r_sm_s,
\qquad r_s=r(d_s),
\qquad m_s=m(d_s)
\qquad(1\leq s\leq k),
\]
as in \eqref{eq:external-medium-decomp}. Then
\[
\frac{d_k}{d_1}
\leq1+(\log n)^{-a}
\leq1+2^a\delta_X.
\]

We show that $m_1,\ldots,m_k$ are pairwise distinct, as in
Corollary~\ref{cor:DMC}. Suppose that $m_i=m_j$ for some
$1\leq i<j\leq k$. Since $d_i<d_j$, we have $r_i<r_j$, and
\[
\frac{r_j}{r_i}
=\frac{d_j/m_j}{d_i/m_i}
=\frac{d_j}{d_i}
\leq\frac{d_k}{d_1}
\leq1+2^a\delta_X.
\]
Both $r_i$ and $r_j$ divide $n$ and are supported on external
primes. Hence $n\in\mathcal B_X(2^a)$. Together with
$n\in(X/2,X]$, this implies $n\in\mathscr B$, a contradiction.
Therefore,
\[
m_i\neq m_j\qquad(1\leq i<j\leq k).
\]

Define $Q$ as in \eqref{eq:Q-def}. Since each $r_s$ divides $n$, it follows that we have
$Q\mid n$. We distinguish two cases.

\medskip
\noindent\textbf{Case 1.} $Q\leq X^{1-\eps}$.
Let $\mathscr U$ be the set of integers counted in
Lemma~\ref{lem:nonsat}. The chosen divisors of $n$ satisfy
\[
d_1<\cdots<d_k\mid n,
\qquad
d_k\leq d_1\bigl(1+(\log n)^{-a}\bigr),
\]
their medium parts are pairwise distinct, and
$Q\leq X^{1-\eps}$. Thus $n\in\mathscr U$, and
Lemma~\ref{lem:nonsat} gives $\#\mathscr U=o(X)$.

\medskip
\noindent\textbf{Case 2.} $Q>X^{1-\eps}$.
Let $\mathscr S$ be the set defined in
Lemma~\ref{lem:saturated}. The same divisors satisfy all the
conditions defining $\mathscr S$, so $n\in\mathscr S$. That lemma gives $\#\mathscr S=o(X)$.

Every $n\in(\mathcal E_{a,k}\cap(X/2,X])\setminus\mathscr B$
belongs to at least one of these two sets. Therefore,
\[
\bigl(\mathcal E_{a,k}\cap(X/2,X]\bigr)\setminus\mathscr B
\subseteq\mathscr U\cup\mathscr S,
\]
and consequently
\[
\begin{aligned}
\#\bigl(\mathcal E_{a,k}\cap(X/2,X]\bigr)
&\leq\#\mathscr B+\#(\mathscr U\cup\mathscr S)\\
&\leq\#\mathscr B+\#\mathscr U+\#\mathscr S=o(X).
\end{aligned}
\]
This proves the claim.
\end{proof}

\begin{claim}\label{clm:main-density}
We have
\[
\lim_{x\to\infty}
\frac{\#\bigl(\mathcal E_{a,k}\cap[1,x]\bigr)}{x}=0.
\]
\end{claim}

\begin{proof}
Fix a positive integer $J$. For every integer $N\geq2^J$, put
\[
N_j=\left\lfloor\frac{N}{2^j}\right\rfloor
\qquad(0\leq j\leq J).
\]
Then $N_0=N$, and, for $0\leq j<J$,
\[
N_{j+1}
=\left\lfloor\frac{N}{2^{j+1}}\right\rfloor
=\left\lfloor\frac{\lfloor N/2^j\rfloor}{2}\right\rfloor
=\left\lfloor\frac{N_j}{2}\right\rfloor.
\]
For an integer $m$, the condition
$m>\lfloor N_j/2\rfloor$ is equivalent to $m>N_j/2$.
Thus
\[
(N_{j+1},N_j]\cap\N
=(N_j/2,N_j]\cap\N.
\]
The sets $[1,N_J]\cap\N$ and
$(N_{j+1},N_j]\cap\N$, $0\leq j<J$, are pairwise disjoint
and have union $[1,N]\cap\N$. It follows that
\[
\begin{aligned}
\#\bigl(\mathcal E_{a,k}\cap[1,N]\bigr)
&=\#\bigl(\mathcal E_{a,k}\cap[1,N_J]\bigr)
+\sum_{j=0}^{J-1}
\#\bigl(\mathcal E_{a,k}\cap(N_j/2,N_j]\bigr)\\
&\leq N_J+
\sum_{j=0}^{J-1}
\#\bigl(\mathcal E_{a,k}\cap(N_j/2,N_j]\bigr).
\end{aligned}
\]

For each fixed $0\leq j<J$, we have
\[
N_j\geq\frac{N}{2^j}-1\longrightarrow\infty
\qquad(N\to\infty).
\]
Applying Claim~\ref{clm:main-dyadic} with $X=N_j$, we obtain
\[
\frac{\#\bigl(\mathcal E_{a,k}\cap(N_j/2,N_j]\bigr)}{N_j}
=o(1)
\qquad(N\to\infty).
\]
Since $J$ is fixed and $N_j/N\leq2^{-j}$, division by $N$
gives
\[
\begin{aligned}
\frac{\#\bigl(\mathcal E_{a,k}\cap[1,N]\bigr)}{N}
&\leq\frac{N_J}{N}
+\sum_{j=0}^{J-1}
\frac{N_j}{N}
\frac{\#\bigl(\mathcal E_{a,k}\cap(N_j/2,N_j]\bigr)}{N_j}\\
&\leq2^{-J}
+\sum_{j=0}^{J-1}2^{-j}
\frac{\#\bigl(\mathcal E_{a,k}\cap(N_j/2,N_j]\bigr)}{N_j}\\
&=2^{-J}+o(1).
\end{aligned}
\]
The last equality holds because the sum has $J$ terms, each
of which tends to zero as $N\to\infty$. Hence
\[
0\leq
\limsup_{N\to\infty}
\frac{\#\bigl(\mathcal E_{a,k}\cap[1,N]\bigr)}{N}
\leq2^{-J}.
\]
This holds for every fixed positive integer $J$. Letting
$J\to\infty$ after taking the upper limit in $N$, we obtain
\[
\lim_{N\to\infty}
\frac{\#\bigl(\mathcal E_{a,k}\cap[1,N]\bigr)}{N}=0,
\]
since the ratios are nonnegative.

For a real number $x\geq1$, put $N=\lfloor x\rfloor$.
Because $\mathcal E_{a,k}\subseteq\N$, we have
\[
\mathcal E_{a,k}\cap[1,x]
=\mathcal E_{a,k}\cap[1,N].
\]
Moreover, $N\to\infty$ and $N/x\to1$ as $x\to\infty$.
Therefore,
\[
\frac{\#\bigl(\mathcal E_{a,k}\cap[1,x]\bigr)}{x}
=\frac{N}{x}
\frac{\#\bigl(\mathcal E_{a,k}\cap[1,N]\bigr)}{N}
\longrightarrow0.
\]
This proves the claim.
\end{proof}

We have proved the required limit for every fixed $a$ satisfying
\eqref{eq:supercritical-a} and $0<a<1/2$. To include $a\geq1/2$,
recall from the reduction preceding \eqref{eq:a-less-half} that
\[
\frac{\beta_k}{1-\beta_k}
\leq\log3-1<\frac12.
\]
Choose a fixed $a_0$ such that
$\frac{\beta_k}{1-\beta_k}<a_0<\frac12.$
For $n\geq3$, we have $\log n>1$ and $a>a_0$, so
\[
(\log n)^{-a}\leq(\log n)^{-a_0}.
\]
It follows directly from the definition of $\mathcal E_{a,k}$
that $\mathcal E_{a,k}\subseteq\mathcal E_{a_0,k}\cup\{2\}.$
Consequently,
\[
0\leq
\frac{\#\bigl(\mathcal E_{a,k}\cap[1,x]\bigr)}{x}
\leq
\frac{\#\bigl(\mathcal E_{a_0,k}\cap[1,x]\bigr)}{x}
+\frac1x
\longrightarrow0.
\]
Thus
\[
\lim_{x\to\infty}
\frac{\#\bigl(\mathcal E_{a,k}\cap[1,x]\bigr)}{x}=0
\]
for every fixed $a>\beta_k/(1-\beta_k)$.

By the definition of $\alpha_k$ in
Subsection~\ref{subsec:close-divisors},
\[
\alpha_k=
\sup\left\{
a\in\R
\,\middle|\,
\lim_{x\to\infty}
\frac{\#\bigl(\mathcal E_{a,k}\cap[1,x]\bigr)}{x}=1
\right\}.
\]
For every $a>\beta_k/(1-\beta_k)$, the preceding limit is zero
rather than one. Therefore,
\[
\left\{
a\in\R
\,\middle|\,
\lim_{x\to\infty}
\frac{\#\bigl(\mathcal E_{a,k}\cap[1,x]\bigr)}{x}=1
\right\}
\subseteq
\left(-\infty,\frac{\beta_k}{1-\beta_k}\right].
\]
Taking the supremum gives
$\alpha_k\leq\frac{\beta_k}{1-\beta_k}.$
Combining this with \eqref{eq:fgk-lower}, we conclude that
\[
\alpha_k=\frac{\beta_k}{1-\beta_k}.
\]
\end{proof}

\appendix
\section{Corrections to the FGK lower-bound argument}
\label{app:fgk-lower-corrections}

This appendix supplies the corrections needed in the application of
\cite[Proposition~5.5]{FGK} in Claim~\ref{clm:comparison-lower}
of Proposition~\ref{prop:threshold-comparison}. We correct the final
summation in the proof of \cite[Proposition~5.10]{FGK}, prove the
lattice estimate of \cite[Proposition~6.2]{FGK} using
\cite[Proposition~5.9]{FGK}, and correct the coordinate indices in
\cite[Lemma~6.3]{FGK}. We then explain how these estimates, together
with \cite[Lemma~6.5]{FGK}, give the conclusion of
\cite[Proposition~5.5]{FGK}.

The estimates preceding the final summation in
\cite[Section~5.4]{FGK}, \cite[Lemma~5.13]{FGK}, and
\cite[Lemma~6.5]{FGK} remain inputs from that paper. Thus the
argument below supplies local corrections to the cited proof.
Its application in the present paper is the inequality
$\widetilde\gamma_k\leq\beta_k$; the opposite inequality is proved
in Claim~\ref{clm:comparison-upper}. All references to numbered
results, equations, and printed pages of FGK refer to the published
version of \cite{FGK}.

Throughout this appendix, fix a system
$(\cV,\mathbf c,\boldsymbol\mu)$ satisfying conditions
\textup{(i)}--\textup{(iii)} of \cite[Proposition~5.5]{FGK}.
In particular, $V_r$ is non-degenerate,
\[
1=c_1>c_2>\cdots>c_{r+1}=c>0,
\qquad
\supp(\mu_j)=V_j\cap\{0,1\}^k\quad(1\leq j\leq r),
\]
and there is an $\eps>0$ such that
\[
\mathrm e(\cV',\mathbf c,\boldsymbol\mu)
\geq\mathrm e(\cV,\mathbf c,\boldsymbol\mu)+\eps
\qquad(\cV'<\cV).
\]
Unless otherwise stated, implied constants may depend on this fixed
system. We retain the logarithmic random set $\mathcal A$, in which
the events $a\in\mathcal A$ are independent and have probability
$1/a$.

\subsection{The final summation in the moment estimate}
\label{app:fgk-moment-gain}

Choose $\kappa>0$ sufficiently small for
\cite[Lemma~5.13]{FGK} to apply, and put
$I_j=(c_{j+1}+\kappa,c_j]$. Let $(\cW,\mathbf b)$ be one of the
saturated adapted pairs in \cite[Section~5.4]{FGK}, and write
\[
\Delta=\delta(\mathbf b)
=\max_{i,j:\,b_i\in I_j}(c_j-b_i).
\]
We use $E(p,\cW,\mathbf b)$ with the meaning assigned to it in
\cite[equation~(5.32)]{FGK}. Equation~(5.35), the interval-length
estimate preceding equation~(5.38), and Lemma~5.13\textup{(b)} of
\cite{FGK} give
\[
E(p,\cW,\mathbf b)
\geq(p-1)\mathrm e(\cV,\mathbf c,\boldsymbol\mu)
  +\frac{(p-1)\eps\Delta}{2}-C(p-1)^2\Delta
\]
for some constant $C>0$. Choose $p>1$ sufficiently close to $1$
that all preceding estimates apply and $C(p-1)\leq\eps/4$.
Then
\begin{equation}\label{eq:fgk-corrected-gain}
E(p,\cW,\mathbf b)
\geq(p-1)\mathrm e(\cV,\mathbf c,\boldsymbol\mu)
  +\frac{(p-1)\eps\Delta}{4}.
\end{equation}
Thus the two terms involving $\eps\Delta$ in
\cite[equation~(5.38), p.~1081]{FGK} must both include the factor
$p-1$.

We next justify the summation following that equation. By
\cite[equation~(5.21)]{FGK}, each coordinate $b_i$ belongs to
\[
1+\frac1{\log D}\Z.
\]
This does not require $\delta(\mathbf b)\log D$ to be an integer.
Fix a flag $\cW$. In the saturated case, its length satisfies
$s=\dim V_r-1\leq k-1$.
For each integer $m\geq0$, define
\[
\mathcal B_m
=
\left\{
\mathbf b:
(\cW,\mathbf b)\text{ is saturated and admissible},\quad
m\leq\delta(\mathbf b)\log D<m+1
\right\}.
\]
If $\mathbf b\in\mathcal B_m$ and $b_i\in I_j$, then
\[
b_i\in
\left(c_j-\frac{m+1}{\log D},\,c_j\right].
\]
This interval contains at most $m+2$ points of the indicated grid.
As $j$ has at most $r$ possible values, each coordinate $b_i$ has
at most $r(m+2)$ possible values. Consequently,
$\#\mathcal B_m\leq \bigl(r(m+2)\bigr)^s$.
Put $\eta=(p-1)\eps/4>0$. It follows that
\[
\begin{aligned}
\sum_{\mathbf b:\,(\cW,\mathbf b)\ {
m saturated}}
D^{-\eta\delta(\mathbf b)}
&\leq\sum_{m\geq0}\#\mathcal B_m\exp(-\eta m)\leq r^s\sum_{m\geq0}(m+2)^s\exp(-\eta m)
\ll1,
\end{aligned}
\]
where all pairs in the sum are required to be admissible.
The extracted flags are generated by $\one$ and vectors in
$\{-1,0,1\}^k$, so their number is bounded in terms of $k$.
Combining this fact with \eqref{eq:fgk-corrected-gain} yields
\begin{equation}\label{eq:fgk-saturated-moment-sum}
\sum_{(\cW,\mathbf b)\ {
m saturated}}
D^{-E(p,\cW,\mathbf b)}
\ll D^{-(p-1)\mathrm e(\cV,\mathbf c,\boldsymbol\mu)}.
\end{equation}
The number of coordinates is $\dim V_r-1$, even when the original
flag $\cV$ is not complete.

Together with \cite[equations~(5.32) and~(5.36)]{FGK},
\eqref{eq:fgk-saturated-moment-sum} completes the final summation
in the proof of \cite[Proposition~5.10]{FGK}. The deduction of
\cite[Proposition~5.9]{FGK} from that moment estimate, given in
\cite[Section~5.2]{FGK}, is unchanged.

\subsection{Elements in residue classes}
\label{app:fgk-residue-supply}

In the proof of \cite[Proposition~6.2, p.~1084]{FGK}, the interval
$\bigl(D^{c(1-\kappa/i)},\delta^{-1}D^{c(1-\kappa/i)}\bigr]$
is asserted to be contained in $I_i(D)\setminus I'_i(D)$.
For fixed $i$, $\kappa$, and $\delta$, however, its upper endpoint
is less than $D^c$ when $D$ is sufficiently large, whereas
$I_i(D)\subset(D^c,D]$.

The probability estimate also requires a longer interval. Indeed,
for fixed $M\geq1$, $0<\delta<1$, and $b\in\Z$, independence gives
\[
\begin{aligned}
&\Pp\bigl(\mathcal A\cap(Y,\delta^{-1}Y]
                 \cap(b+M\Z)=\varnothing\bigr)=
\prod_{\substack{Y<a\leq\delta^{-1}Y\\a\equiv b\pmod M}}
\left(1-\frac1a\right)
=\delta^{1/M}(1+o(1))
\qquad(Y\to\infty).
\end{aligned}
\]
This probability exceeds $\delta/2$ for sufficiently large $Y$.
For the lattice adjustment, we need a supply of random elements
just above $D^c$ in every residue class modulo $M$. The preceding
calculation shows that a fixed multiplicative interval does not give
the required high-probability estimate. We therefore use a longer
interval and record its occupancy bound in the next lemma.

\begin{lemma}\label{lem:fgk-residue-supply}
Fix integers $M,m\geq1$ and a real number $c>0$. There is a
constant $\delta_0=\delta_0(M,m)>0$ such that, for every
$0<\delta\leq\delta_0$ and sufficiently large $D$, the interval
\begin{equation}\label{eq:fgk-new-interval}
J_D=(D^c,\delta^{-32M}D^c]
\end{equation}
contains at least $m$ elements of $\mathcal A$ in each residue
class modulo $M$, with probability at least $1-\delta/2$.
\end{lemma}

\begin{proof}
For $0\leq b<M$, let
$X_b=\#\bigl(\mathcal A\cap J_D\cap(b+M\Z)\bigr)$ and
$L=\log(1/\delta)$.
Summation over the arithmetic progression gives
\[
\lambda_b=\E X_b
=\sum_{\substack{D^c<a\leq\delta^{-32M}D^c\\a\equiv b\pmod M}}
\frac1a
=32L+O_M(D^{-c}).
\]
Choose
\[
0<\delta_0\leq
\min\left\{\frac14,\frac1{2M},\exp(-m/8)\right\}.
\]
For $0<\delta\leq\delta_0$ and sufficiently large $D$, we have
$\lambda_b\geq16L\geq2m$ for every $b$.

If $X$ is a sum of independent Bernoulli random variables with
mean $\lambda$, then
\[
\E\exp(-tX)\leq
\exp\bigl((\exp(-t)-1)\lambda\bigr)
\qquad(t>0).
\]
Taking $t=\log2$ and applying Markov's inequality, we obtain
\[
\Pp(X\leq\lambda/2)
\leq\exp\left(-\frac{1-\log2}{2}\lambda\right)
\leq\exp(-\lambda/8).
\]
Consequently,
\[
\Pp(X_b<m)\leq\exp(-\lambda_b/8)\leq\delta^2.
\]
The probability that $X_b<m$ for at least one residue class is
at most $M\delta^2\leq\delta/2$, as required.
\end{proof}

\subsection{Sums satisfying the lattice condition}
\label{app:fgk-lattice-adjustment}

The preceding lemma supplies the elements needed to correct residue
classes. We next show how to assign these elements to the values of
a map so that its sum lies in a prescribed lattice. The same argument
controls how many distinct sum classes can merge under this correction.

\begin{lemma}\label{lem:fgk-lattice-adjustment}
Let $\Omega\subset\Z^k$ be finite, with $\mathbf0\in\Omega$ and
$|\Omega|=m$. Let $\Gamma\leq\Z^k$ contain $\one$, and suppose that
$M\Omega\subset\Gamma$ for some integer $M\geq1$. Let $B$ and $K$
be disjoint finite sets of integers, where $K$ contains exactly $m$
elements in each residue class modulo $M$.

Write
$\pi:\Z^k\longrightarrow\Z^k/\langle\one\rangle$
for the quotient map. Given a family $\mathscr F$ of maps
$\psi:B\to\Omega$, define its set of input classes by
\[
S_{\mathrm{in}}
=
\left\{
\pi\!\left(\sum_{a\in B}a\psi(a)\right):
\psi\in\mathscr F
\right\}.
\]
Every $\psi\in\mathscr F$ has an extension
$\widetilde\psi:B\cup K\to\Omega$ for which
$\sum_{a\in B\cup K}a\widetilde\psi(a)\in\Gamma$.
The extensions can be chosen so that their set of output classes
\[
S_{\mathrm{out}}
=
\left\{
\pi\!\left(\sum_{a\in B\cup K}a\widetilde\psi(a)\right):
\psi\in\mathscr F
\right\}
\]
satisfies
\[
|S_{\mathrm{out}}|
\geq \frac{|S_{\mathrm{in}}|}{(mM)^m}.
\]
\end{lemma}

\begin{proof}
Fix $\psi\in\mathscr F$ and, for each $\omega\in\Omega$, put
\[
N_\omega
=
\sum_{\substack{a\in B\\ \psi(a)=\omega}}a.
\]
Choose distinct elements $a_\omega\in K$ such that
$a_\omega\equiv-N_\omega\pmod M
~(\omega\in\Omega)$.

This is possible because at most $m$ elements are needed from any
one residue class, while $K$ contains $m$ elements in that class.
Set $\widetilde\psi(a_\omega)=\omega$ and assign $\mathbf0$ to
every unused element of $K$. Then
\[
\sum_{a\in B\cup K}a\widetilde\psi(a)
=
\sum_{\omega\in\Omega}(N_\omega+a_\omega)\omega
\in\Gamma,
\]
because each coefficient $N_\omega+a_\omega$ is divisible by $M$.

Now choose one map $\psi_c\in\mathscr F$ for each input class
$c\in S_{\mathrm{in}}$, and make the preceding choices for these
maps. Let
\[
v_c=\sum_{\omega\in\Omega}a_{\omega,c}\omega,
\qquad
y_c=
\pi\!\left(
\sum_{a\in B\cup K}a\widetilde\psi_c(a)
\right).
\]
By the definition of $c$, we have
$y_c=c+\pi(v_c)$.
The correction vector $v_c$ belongs to
\[
\mathcal T
=
\left\{
\sum_{\omega\in\Omega}a_\omega\omega:
(a_\omega)_{\omega\in\Omega}\in K^\Omega
\right\},
\qquad
|\mathcal T|\leq |K|^m=(mM)^m.
\]
For any fixed output class $y$, every input class $c$ with $y_c=y$
therefore has the form
$c=y-\pi(v)$
for some $v\in\mathcal T$.
Thus at most $(mM)^m$ input classes can produce the same output
class. Extending any remaining maps in $\mathscr F$ as above gives
$|S_{\mathrm{out}}|\geq |S_{\mathrm{in}}|/(mM)^m$.
\end{proof}

We apply these two lemmas to the fixed entropy system. Put
$d=\dim V_r-1$. Since the spaces of $\cV$ are cube-spanned,
we may choose a basis
$\one,\omega^1,\ldots,\omega^d$ of $V_r$ over $\Q$, consisting
of cube points and satisfying
\[
V_j=\Span_{\Q}
\{\one,\omega^1,\ldots,\omega^{\dim V_j-1}\}
\qquad(0\leq j\leq r).
\]
Set
\[
\Gamma=\Z\one+\sum_{t=1}^d\Z\omega^t,
\qquad
\Omega=V_r\cap\{0,1\}^k,
\qquad m=|\Omega|.
\]
The basis coordinates of every $\omega\in\Omega$ are rational.
Since $\Omega$ is finite, clearing their denominators gives an
integer $M\geq1$ such that $M\Omega\subset\Gamma$.

Let $\kappa^*>0$ be the constant in
\cite[Proposition~5.9]{FGK}, decreased if necessary so that
\[
\kappa^*<\min_{1\leq j\leq r}(c_j-c_{j+1}).
\]
Fix $0<\kappa\leq\kappa^*/2$. For $i\geq1$, define
\[
\begin{aligned}
I_i(D)&=\bigcup_{j=1}^r
(D^{c_{j+1}},D^{c_j(1-\kappa/i)}],\\
I'_i(D)&=\bigcup_{j=1}^r
(D^{(c_{j+1}+\kappa^*)(1-\kappa/i)},
 D^{c_j(1-\kappa/i)}].
\end{aligned}
\]
A map $\psi:\mathcal A\cap I_i(D)\to\{0,1\}^k$ is compatible
if $\psi(a)\in V_j$ whenever
$a\in(D^{c_{j+1}},D^{c_j(1-\kappa/i)}]$. It is non-degenerate
if, for every $s\ne t$, there is an $a$ in its domain such that
$\psi(a)_s\ne\psi(a)_t$.
Let $\mathscr L_i(\mathcal A)$ be the set of classes modulo
$\langle\one\rangle$ represented by sums in $\Gamma$ of such
non-degenerate, compatible maps.

We can now combine the unrestricted sum estimate of
\cite[Proposition~5.9]{FGK} with the two lemmas above. The next
proposition gives a lower bound for lattice-constrained classes,
uniformly in the index $i$.

\begin{proposition}\label{prop:fgk-restored-lattice-bound}
There are constants $C_0,\alpha>0$ such that, for every
$0<\delta<1$ and sufficiently large $D$, one has
\begin{equation}\label{eq:fgk-restored-lattice-bound}
\Pp\left(
|\mathscr L_i(\mathcal A)|
\geq C_0\delta^\alpha
D^{(1-\kappa/i)\sum_{j=1}^r c_j\dim(V_j/V_{j-1})}
\right)\geq1-\delta
\end{equation}
for every integer $i\geq1$. The constants and the lower bound
on $D$ are independent of $i$.
\end{proposition}

\begin{proof}
Since $c_{j+1}+\kappa^*<c_j\leq1$ and $\kappa/i\leq\kappa$,
\[
(c_{j+1}+\kappa^*)(1-\kappa/i)
\geq c_{j+1}+\kappa^*-\kappa
\geq c_{j+1}+\kappa^*/2.
\]
Thus $I'_i(D)\subset I_i(D)$. Also,
\[
c_r(1-\kappa/i)
\geq c_r-\kappa>c+\kappa^*/2.
\]
For fixed $\delta$, it follows that the interval $J_D$ in
\eqref{eq:fgk-new-interval} satisfies
\begin{equation}\label{eq:fgk-interval-location}
J_D\subset(D^c,D^{c+\kappa^*/2})
\subset I_i(D)\setminus I'_i(D)
\qquad(i\geq1)
\end{equation}
for sufficiently large $D$, uniformly in $i$.

Apply \cite[Proposition~5.9]{FGK} with
$D'=D^{1-\kappa/i}$, interval parameter $\kappa^*$, and failure
probability $\delta/2$. Its domains
\[
\mathcal A\cap
\bigl((D')^{c_{j+1}+\kappa^*},\exp(-1)(D')^{c_j}\bigr]
\]
are contained in the corresponding intervals of $I'_i(D)$.
Extend the resulting maps by zero to
$B=\mathcal A\cap I'_i(D)$. With probability at least
$1-\delta/2$, their sum classes form a set $S_i$ satisfying
\begin{equation}\label{eq:fgk-unrestricted-sums}
|S_i|\gg\delta^\alpha
D^{(1-\kappa/i)\sum_{j=1}^r c_j\dim(V_j/V_{j-1})},
\qquad\alpha=\frac1{p-1}>0.
\end{equation}
The lower bound on $D$ is independent of $i$, since
$D'\geq D^{1-\kappa}$.

First suppose that $0<\delta\leq\delta_0(M,m)$.
By Lemma~\ref{lem:fgk-residue-supply}, with probability at least
$1-\delta/2$ there is a set $K\subset\mathcal A\cap J_D$
containing exactly $m$ elements in each residue class modulo $M$.
By \eqref{eq:fgk-interval-location}, $B\cap K=\varnothing$.
Lemma~\ref{lem:fgk-lattice-adjustment} therefore supplies at least
$|S_i|/(mM)^m$ classes whose representing sums lie in $\Gamma$.
All added values belong to $V_r\cap\{0,1\}^k$, and $K$ lies in
the interval corresponding to $j=r$. Extending the maps by zero
on the rest of $\mathcal A\cap I_i(D)$ consequently preserves
compatibility. It also preserves non-degeneracy, since the maps
on $B$ are unchanged.

The two events hold simultaneously with probability at least
$1-\delta$. On their intersection,
\[
|\mathscr L_i(\mathcal A)|
\geq\frac{|S_i|}{(mM)^m},
\]
which proves \eqref{eq:fgk-restored-lattice-bound} for
$\delta\leq\delta_0$.
If $\delta_0<\delta<1$, apply this result with $\delta_0$.
Since $\delta_0^\alpha\geq\delta_0^\alpha\delta^\alpha$ and
$1-\delta_0\geq1-\delta$, decreasing $C_0$ by a fixed factor
proves the stated bound for this range as well.
\end{proof}

For integral vectors, equality modulo $\langle\one\rangle$
is equivalent to equality modulo $\Z\one$: an integral vector
of the form $t\one$ necessarily has $t\in\Z$. Moreover,
$\Gamma/\Z\one$ is generated freely by the classes of
$\omega^1,\ldots,\omega^d$. The preceding proposition therefore
gives the quantitative conclusion of \cite[Proposition~6.2]{FGK}.
The factor $(mM)^m$ counts indexed assignments of elements of $K$
to vectors of $\Omega$; it replaces the bound $2^{|K|}$ used in
the original proof.

\subsection{Coordinate bounds}
\label{app:fgk-coordinate-bounds}

The iteration also requires coordinate bounds for the classes
counted in Proposition~\ref{prop:fgk-restored-lattice-bound}.
We identify the relevant scale of each coordinate using the first
space of the flag in which that coordinate can occur.

For $1\leq t\leq d$, define $j(t)$ by
\begin{equation}\label{eq:fgk-coordinate-index}
\dim V_{j(t)-1}-1<t\leq\dim V_{j(t)}-1.
\end{equation}
The subtraction of $1$ accounts for the initial basis vector
$\one$ and corrects the index used in \cite[Section~6.2, p.~1085]{FGK}.
Let $\ell_t$ be the $\omega^t$-coordinate functional with respect
to the basis $\one,\omega^1,\ldots,\omega^d$, and put
\[
C=\max\left\{1,
|\ell_t(\omega)|:\omega\in\Omega,\ 1\leq t\leq d\right\}.
\]
By \eqref{eq:fgk-coordinate-index}, $\ell_t$ vanishes on
$V_{j(t)-1}$. Thus compatibility implies that every representing
sum $x=\sum_{a\in\mathcal A\cap I_i(D)}a\psi(a)$ satisfies
\[
|\ell_t(x)|
\leq C\sum_{\substack{a\in\mathcal A\\
2\leq a\leq D^{(1-\kappa/i)c_{j(t)}}}}a.
\]
Indeed, a value outside $V_{j(t)-1}$ can occur only in an interval
with index $j\geq j(t)$, whose upper endpoint is at most
$D^{(1-\kappa/i)c_{j(t)}}$.

For $\delta>0$ sufficiently small that $r\exp(-2/\delta)\leq\delta$,
\cite[Lemma~A.6]{FGK} gives, with probability at least $1-\delta$,
\[
\sum_{\substack{a\in\mathcal A\\2\leq a\leq
D^{(1-\kappa/i)c_j}}}a
\leq\delta^{-1}D^{(1-\kappa/i)c_j}
\qquad(1\leq j\leq r).
\]
Set
\[
N_j^{(i)}=C\delta^{-1}D^{(1-\kappa/i)c_j}.
\]
Identifying $\Gamma/\Z\one$ with $\Z^d$ by the selected basis,
we obtain
\[
\mathscr L_i(\mathcal A)
\subset\prod_{t=1}^d[-N_{j(t)}^{(i)},N_{j(t)}^{(i)}]
\]
with probability at least $1-\delta$. Furthermore,
$\#\{t:j(t)=j\}=\dim(V_j/V_{j-1})$,
and hence
\begin{equation}\label{eq:fgk-coordinate-product}
\prod_{t=1}^dN_{j(t)}^{(i)}
=C^d\delta^{-d}
D^{(1-\kappa/i)\sum_{j=1}^r c_j\dim(V_j/V_{j-1})}.
\end{equation}
Combining \eqref{eq:fgk-coordinate-product} with
Proposition~\ref{prop:fgk-restored-lattice-bound} gives the
cardinality and coordinate estimates used in
\cite[Proposition~6.4]{FGK}, with probability at least $1-2\delta$.

\subsection{The consequence for the threshold comparison}
\label{app:fgk-threshold-consequence}

We finish by recording how the corrected estimates give the
conclusion required from \cite[Proposition~5.5]{FGK}.
Fix a sufficiently small $\delta>0$. For each $i\geq1$, let
$F_i$ be the event that both the cardinality and coordinate
estimates above hold, and let
$E_i=\{0\notin\mathscr L_i(\mathcal A)\}$.
Then $\Pp(F_i^c)\leq2\delta$. Since a compatible map on
$\mathcal A\cap I_i(D)$ extends by zero to
$\mathcal A\cap I_{i+1}(D)$, the sets
$\mathscr L_i(\mathcal A)$ are increasing. Therefore
$E_{i+1}\subseteq E_i$.

By \cite[Lemma~6.5]{FGK}, with the indices in
\eqref{eq:fgk-coordinate-index}, there is a constant
$\rho=\rho(\delta)>0$, independent of $D$ and $i$, such that
\[
\Pp(E_{i+1}^c\mid E_i\cap F_i)\geq\rho
\qquad
\left(1\leq i\leq\left\lfloor(\log D)^{1/3}\right\rfloor\right)
\]
whenever the conditioning event has positive probability and $D$
is sufficiently large. Consequently,
\[
\begin{aligned}
\Pp(E_{i+1})
&\leq(1-\rho)\Pp(E_i\cap F_i)+\Pp(E_i\cap F_i^c)\leq(1-\rho)\Pp(E_i)+2\rho\delta.
\end{aligned}
\]
Put $T=\lfloor(\log D)^{1/3}\rfloor$. Iterating this inequality
gives
\[
\Pp(E_{T+1})\leq(1-\rho)^T+2\delta.
\]
For fixed $\delta$, the first term tends to zero as $D\to\infty$.

If $E_{T+1}$ does not occur, there is a non-degenerate map
$\psi:\mathcal A\cap I_{T+1}(D)\to\{0,1\}^k$ for which
\[
\sum_{a\in\mathcal A\cap I_{T+1}(D)}a\psi(a)
\in\langle\one\rangle.
\]
For $1\leq t\leq k$, define
\[
A_t=\{a\in\mathcal A\cap I_{T+1}(D):\psi(a)_t=1\}.
\]
The non-degeneracy of $\psi$ implies that the sets $A_t$ are
pairwise distinct, and the preceding vector identity implies that
their sums are equal. Since $I_{T+1}(D)\subset(D^c,D]$, we obtain
\[
\limsup_{D\to\infty}
\Pp\bigl(m(\mathcal A\cap(D^c,D])<k\bigr)\leq2\delta.
\]
Letting $\delta\to0$ proves
\begin{equation}\label{eq:fgk-lower-conclusion}
\Pp\bigl(m(\mathcal A\cap(D^c,D])\geq k\bigr)=1-o(1).
\end{equation}
This is the conclusion of \cite[Proposition~5.5]{FGK} needed
in the present paper; it also implies the version with
$[D^c,D]$.

In Claim~\ref{clm:comparison-lower}, the normalized and perturbed
system has final threshold $c=b+\zeta>0$ and satisfies the
hypotheses fixed at the beginning of this appendix. Applying
\eqref{eq:fgk-lower-conclusion} therefore gives
$\beta_k\geq b+\zeta\geq b$. Taking the supremum over the
original strict entropy systems yields
$\widetilde\gamma_k\leq\beta_k$.
This application uses neither
Theorem~\ref{thm:entropy-equality} nor Theorem~\ref{thm:main}.

\section*{Declaration on the use of AI}

ChatGPT was used for grammar checking, language polishing, and
improving the clarity of the exposition.

In particular, it assisted in revising the proof of
Lemma~\ref{lem:perspective}, including the treatment of cases in
which some coset weights vanish. It also assisted in clarifying
the uniformity and parameter dependence of the error terms in
the proof of Theorem~\ref{thm:affine}, and in revising the derivation
of the lower bound for $u=\log N/\log z$ in the proof of
Lemma~\ref{lem:nonsat}.

We discovered errors in two key lemmas in \cite{FGK},
namely Lemmas~4.5 and~4.6. These errors are addressed in
Lemmas~\ref{lem:residual-sum-bound} and~\ref{lem:finite-subflags}
of our paper, respectively. Both lemmas are used in our proofs.
We then used ChatGPT to review the paper further.
The additional proofs that ChatGPT found to be insufficiently
rigorous are presented in the appendix.

The authors take full responsibility for the accuracy and
originality of the paper.

\section*{Data availability}

No datasets were generated or analyzed in this study.

\end{document}